\documentclass{amsart}
\chardef\bslash=`\\

\usepackage{amssymb,amsmath,amsfonts,amsthm,amscd,stmaryrd,mathrsfs,color}
\usepackage{hyperref}
\usepackage[all,cmtip,poly]{xy}

\usepackage{silence}
\newtheorem{theorem}{Theorem}[section]

\newtheorem{lemma}[theorem]{Lemma}

\newtheorem{introthm}{Theorem}

\newtheorem{introconj}[introthm]{Conjecture}

\newtheorem{cor}[theorem]{Corollary}

\newtheorem{prop}[theorem]{Proposition}
\newtheorem{proposition}[theorem]{Proposition}

\newtheorem{defn}[theorem]{Definition}
\theoremstyle{remark}

\numberwithin{equation}{section}

\newif\iffinalrun
\iffinalrun
\else
 \fi

\iffinalrun
  \newcommand{\need}[1]{}
  \newcommand{\mar}[1]{}
\else
  \newcommand{\need}[1]{{\tiny *** #1}}
  \newcommand{\mar}[1]{\marginpar{\raggedright\tiny FIXME  #1}}\fi

\renewcommand\mathbb{\mathbf}

\newcommand{\Proj}{{\operatorname{Proj}}}

\newcommand{\Lie}{{\operatorname{Lie}\,}}

\newcommand{\rec}{{\operatorname{rec}}}

\newcommand{\rbar}{\overline{r}}

\newcommand{\wv}{{\widetilde{v}}}

\renewcommand{\ell}{l}

\def\PSL{\mathrm{PSL}}
\def\PGL{\mathrm{PGL}}

\def\Iw{\mathrm{Iw}}

\newcommand{\Sets}{\operatorname{Sets}}

\newcommand{\ad}{\operatorname{ad}}
\newcommand{\diag}{\operatorname{diag}}
\newcommand{\tr}{\operatorname{tr}}

\newcommand{\A}{\mathbf{A}}
\newcommand{\bA}{\ensuremath{\mathbf{A}}}

\newcommand{\bC}{\ensuremath{\mathbf{C}}}

\newcommand{\F}{\FF}
\newcommand{\FF}{{\mathbb F}}
\newcommand{\bF}{\ensuremath{\mathbf{F}}}
\newcommand{\bG}{\ensuremath{\mathbf{G}}}

\newcommand{\bQ}{\ensuremath{\mathbf{Q}}}
\newcommand{\Q}{\QQ}
\newcommand{\QQ}{{\mathbb Q}}
\newcommand{\bN}{{\mathbf N}}

\newcommand{\bR}{\ensuremath{\mathbf{R}}}
\newcommand{\R}{\RR}
\newcommand{\RR}{{\mathbb R}}
\newcommand{\bT}{\ensuremath{\mathbf{T}}}

\newcommand{\Z}{\ZZ}
\newcommand{\ZZ}{{\mathbb Z}}
\newcommand{\bZ}{\ensuremath{\mathbf{Z}}}

\newcommand{\bbZ}{\ensuremath{\mathbf{Z}}}
\newcommand{\bbQ}{\ensuremath{\mathbf{Q}}}

\newcommand{\cC}{{\mathcal C}}

\newcommand{\cG}{{\mathcal G}}
\newcommand{\cH}{{\mathcal H}}

\newcommand{\cO}{{\mathcal O}}

\newcommand{\cS}{{\mathcal S}}

\newcommand{\cP}{{\mathcal{P}}}

\newcommand{\m}{\frakm}
\newcommand{\ffrm}{{\mathfrak m}}

\newcommand{\frakm}{\mathfrak{m}}

\newcommand{\frakp}{\mathfrak{p}}
\newcommand{\p}{\frakp}
\newcommand{\frakq}{\mathfrak{q}}
\newcommand{\q}{\frakq}

\newcommand{\Fbar}{\overline{\F}}
\newcommand{\Qbar}{\overline{\Q}}
\newcommand{\Zbar}{\overline{\Z}}

\newcommand{\Fpbar}{\Fbar_p}

\newcommand{\Fpbarx}{\Fpbar^{\times}}

\newcommand{\Zpbar}{\Zbar_p}

\newcommand{\Qpbar}{\Qbar_p}

\DeclareMathOperator{\Aut}{Aut}
\DeclareMathOperator{\Ad}{Ad}

\DeclareMathOperator{\Gal}{Gal}
\newcommand{\GL}{\mathrm{GL}}

\DeclareMathOperator{\Hom}{Hom}

\DeclareMathOperator{\Ind}{Ind}

\DeclareMathOperator{\SL}{SL}
\DeclareMathOperator{\Sp}{Sp}
\DeclareMathOperator{\Spec}{Spec}
\DeclareMathOperator{\Supp}{Supp}

\DeclareMathOperator{\Sym}{Sym}

\newcommand{\Frob}{\mathrm{Frob}}

\newcommand{\rhobar}{\overline{\rho}}

\newcommand{\Art}{{\operatorname{Art}}}
\newcommand{\Res}{\operatorname{Res}}

\newcommand{\univ}{\mathrm{univ}}

\newcommand{\doubleslash}{/\kern-0.2em{/}}

\newenvironment{psmallmatrix}
  {\left(\begin{smallmatrix}}
  {\end{smallmatrix}\right)}
\begin{document}
\author{James Newton and Jack A. Thorne}
\title[Symmetric power functoriality]{Symmetric power functoriality for Hilbert modular forms}
\begin{abstract} Let $F$ be a totally real field. We prove the existence of all symmetric power liftings of those cuspidal automorphic representations of $\GL_2(\bA_F)$ associated to Hilbert modular forms of regular weight. 
\end{abstract}
\maketitle
\setcounter{tocdepth}{1}
\tableofcontents

\section{Introduction} 

In this paper, we prove the following theorem.
\begin{introthm}\label{introthm_symmetric_power_functoriality}
Let $F$ be a totally real field, and let $\pi$ be a cuspidal automorphic representation of $\GL_2(\bA_F)$, without CM, such that $\pi_\infty$ is essentially square-integrable. Then for each $n \geq 2$, the $(n-1)^\text{th}$ symmetric power $\Sym^{n-1} \pi$ exists, in the sense that there is a cuspidal automorphic representation $\Pi_n$ of $\GL_n(\bA_F)$ such that for each place $v$ of $F$, we have an isomorphism of Weil--Deligne representations
\[ \rec_{F_v}(\Pi_n) \cong \Sym^{n-1} \circ \, \rec_{F_v}(\pi_v).\]
\end{introthm}
These automorphic representations are precisely the ones associated to cuspidal, non-CM Hilbert modular forms of weight $k_v \geq 2$ at infinite places $v$ of $F$ (see \cite[Proposition 5.2.3]{Oht83} for a precise formulation of the dictionary).\footnote{The symmetric power liftings in the CM case can be seen to exist using the theory of Eisenstein series (and are never cuspidal when $n > 2$). We can say nothing about the forms of irregular weight (i.e. when $k_v = 1$ for some $v$) except when $k_v = 1$ for all $v$, in which case one can argue as in \cite[Appendix A]{New21b}.} The condition that $\pi_\infty$ is essentially square-integrable may be rephrased as saying that each of the archimedean local factors $\pi_v$ ($v | \infty)$ is a character twist of a square-integrable (i.e.~discrete series) representation of $\GL_2(\bR)$.

\subsection*{The strategy of \cite{Clo14}} Our approach to the proof of Theorem \ref{introthm_symmetric_power_functoriality} is based on the framework developed in \cite{Clo14} by Clozel and the second author, using Galois representations.\footnote{This means that for most of the paper, we will only consider those $\pi$ in Theorem \ref{introthm_symmetric_power_functoriality} that  are regular algebraic (which, as observed in \cite{Clo90}, is the right condition for them to admit associated Galois representations; up to a character twist, this means imposing the condition that the parity of $k_v$ is independent of $v$). The apparently more general statement of Theorem \ref{introthm_symmetric_power_functoriality} will be deduced from the regular algebraic case, as in \cite{ctiii, Tho22a}, at the very end.}
We begin by describing this framework.  

Let $F$ be a totally real field, and let $\Pi$ be a cuspidal automorphic representation of $\GL_n(\A_F)$ that is regular algebraic. It is now known that $\Pi$ admits an associated compatible family of Galois representations: for any prime $p$ and isomorphism $\iota : \overline{\bQ}_p \to \bC$, there is an associated semisimple representation $r_{\Pi, \iota} : G_F \to \GL_n(\overline{\bQ}_p)$ satisfying local-global compatibility in a weak sense, namely that for all but finitely many places $v$ of $F$ such that $\Pi_v$ is unramified, the restriction $r_{\Pi, \iota}|_{G_{F_v}}$ to a decomposition group is unramified, and we have the equality of characteristic polynomials
\[ \det(X - r_{\Pi, \iota}(\Frob_v)) = \iota^{-1} \det(X - \rec_{F_v}^T(\Pi_v)(\Frob_v)). \]
This is a form of the usual determination of the characteristic polynomial of Frobenius in terms of the Satake parameter of $\Pi_v$. In particular, the Galois representation $r_{\Pi, \iota}$ encodes the isomorphism class of all but finitely many local components of $\Pi$; by the strong multiplicity one theorem, this means that $r_{\Pi, \iota}$ uniquely determines $\Pi$. 

If $\rho, \rho' : G_F \to \GL_n(\overline{\bQ}_p)$ are continuous, semisimple representations, we say that they are congruent if their associated semisimple residual representations $\overline{\rho}, \overline{\rho}' : G_F \to \GL_n(\overline{\bF}_p)$ are isomorphic. Equally, if $\Pi, \Pi'$ are cuspidal, regular algebraic automorphic representations of $\GL_n(\A_F)$, we say that they are congruent if there is an isomorphism $\overline{r}_{\Pi, \iota} \cong \overline{r}_{\Pi', \iota}$ of automorphic residual representations. By the Chebotarev density theorem, this is equivalent to asking for the Hecke eigenvalues at all but finitely many unramified places of $\Pi$ and $\Pi'$ to have the same image in $\overline{\bF}_p$ (after application of $\iota$ and reduction modulo $p$). 

The possibility of applying Galois representations to the problem of functoriality arises from the existence of automorphy lifting theorems: these assert that, under suitable conditions, the automorphy of one of a pair of congruent Galois representations $\rho, \rho'$ (i.e. the existence of an isomorphism $\rho \cong r_{\Pi, \iota}$ for some $\Pi$) implies the automorphy of the other. Here we consider starting with a cuspidal, regular algebraic automorphic representation $\pi$ of $\GL_2(\A_F)$. In view of local-global compatibility, the automorphy of the symmetric power Galois representation $\Sym^{n-1} r_{\pi, \iota} : G_F \to \GL_n(\overline{\bQ}_p)$ implies the existence of the Langlands functorial lift, as an automorphic representation of $\GL_n(\A_F)$. (More precisely, local-global compatibility in the weak sense formulated above implies the existence of a lift $\Sym^{n-1} \pi$ with the correct local behaviour at all but finitely many places. For essentially self-dual automorphic representations, a class including all those considered in this paper, local-global compatibility is known in the strong sense of compatibility with the local Langlands correspondence at every place, including ramified places. Precise statements are given in Lemma \ref{lem_the_meaning_of_existence} below.)

This seems like a promising start, since there are plentiful congruences between automorphic representations of $\GL_2(\A_F)$. The existence of a congruence between a pair $\pi, \pi'$ then implies the existence of a congruence $\Sym^{n-1} \overline{r}_{\pi, \iota} \cong \Sym^{n-1} \overline{r}_{\pi', \iota}$ of $n$-dimensional Galois representations; the existence of $\Sym^{n-1} \pi$, combined with a suitable automorphy lifting theorem, could then be used to deduce the existence of $\Sym^{n-1} \pi'$. However, this does not allow you to pass from degree $n$ to some degree $n' > n$. 

The starting point of \cite{Clo14} is the observation that if $n = p + r$, for some prime $p$ and $0 < r < p$, then there is a congruence of residual representations
\begin{equation}\label{eqn_CT14_congruence} \Sym^{p+r-1}  \overline{r}_{\pi, \iota} \overset{ss}{\cong}((\det \overline{r}_{\pi, \iota})^r \otimes \Sym^{p-r-1} \overline{r}_{\pi, \iota} ) \oplus ( {}^{\varphi_p} \overline{r}_{\pi, \iota} \otimes \Sym^{r-1} \overline{r}_{\pi, \iota}) 
\end{equation}
(where $\varphi_p \in G_{\bF_p}$ is the arithmetic Frobenius, and we are twisting the coefficients; the superscript $ss$ indicates that there is in fact only an isomorphism after passage to semisimplification.) With a suitable automorphy lifting theorem, one could then hope to establish the automorphy of the symmetric power in degree $n = p + r$ by first verifying the residual automorphy of the direct summands of the residual representation, which might be feasible if one knew the existence of the symmetric power in smaller degrees. 

An automorphy lifting theorem that applies to residually reducible Galois representations was proved in \cite{jackreducible} and generalized in \cite{All20}; these results require the automorphic representation verifying the residual automorphy to have a Steinberg (in particular, square-integrable) local component. In \cite{Clo14}, Clozel and the second author exploited these automorphy lifting theorems to establish an implication of the form 
 \begin{equation}\label{eqn_implication}  \mathrm{SP}_{r}+ \mathrm{SP}_{p-r} + \mathrm{LR}_{p+r} + \mathrm{TP}_r \Rightarrow \mathrm{SP}_{p+r}, 
\end{equation}
where $\mathrm{TP}_r$ is a conjecture asserting the existence of a certain instance of $\GL_2 \otimes \GL_r \to \GL_{2r}$ tensor product functoriality, $\mathrm{LR}_{p+r}$ is a conjecture asserting the existence of certain level-raising congruences in degree $p+r$, and $\mathrm{SP}_n$ is the conjectural existence of the symmetric power in degree $n$; since this latter conjecture is particularly important for us, we record it as 
\begin{introconj}\label{conj_SP_n}
\textup{(}$\mathrm{SP}_n$\textup{)} For any totally real field $F$ and any cuspidal, regular algbraic automorphic representation $\pi$ of $\GL_2(\bA_F)$, without CM, the lift $\Sym^{n-1} \pi$ exists (as a RAESDC\footnote{The meaning of this acronym, introduced in \cite{blght}, is recalled in \S \ref{subsec_notation} below.} automorphic representation of $\GL_n(\bA_F)$).
\end{introconj}
The argument of \cite{Clo14} leading to the implication (\ref{eqn_implication}) can be summarized as follows: first, verify the automorphy of the factors of the residual representation appearing in (\ref{eqn_CT14_congruence}) using  $ \mathrm{SP}_{r} $, $ \mathrm{SP}_{p-r} $, and $\mathrm{TP}_r$. Second, construct a suitable Eisenstein series on $\GL_{p+r}(\bA_F)$ (or, in practice, an endoscopic lift from $U(2r) \times U(p-r)$ to a suitable unitary group $U(p+r)$) whose associated Galois representation is a sum of factors lifting those in (\ref{eqn_CT14_congruence}). Third, apply $\mathrm{LR}_{p+r}$ to find a congruence to a cusp form on $\GL_{p+r}(\bA_F)$ (or stable form on $U(p+r)$) which satisfies the necessary Steinberg local condition. Finally, apply the automorphy lifting theorems proved in \cite{jackreducible, All20} to deduce the automorphy of $\Sym^{n-1} r_{\pi, \iota}$. 

A primary motivation for considering the implication (\ref{eqn_implication}) is the possibility of establishing $\mathrm{SP}_n$ for all $n \geq 1$ by induction, using Bertrand's postulate: indeed, any integer $n \geq 6$ can be written as $n = p + r$ for some prime $p \geq 5$ and integer $0 < r < p$. Unconditional knowledge of conjectures $\mathrm{LR}_m$ and $\mathrm{TP}_m$, together with the known cases of $\mathrm{SP}_n$ for $1 \leq n \leq 5$, would therefore imply $\mathrm{SP}_n$ for all $n \geq 1$.

\subsection*{The results of \cite{ctii,ctiii,Tho22a}}
In the papers \cite{ctii, ctiii}, Clozel and the second author carried out the above strategy to establish $\mathrm{SP}_n$ unconditionally for $n \le 9$, under mild technical hypotheses on $F$, using known cases of $\mathrm{TP}_r$ (i.e.~the cases $r \le 3$) and proving new (but rather restrictive) cases of $\mathrm{LR}_{n}$ along the way. 

More recently, all the required cases of $\mathrm{LR}_{n}$ have been proved by the second author \cite{Tho22a}, using a new method. This shows that the existence of tensor product functoriality (i.e.~$\mathrm{TP}_r$ for all $r \ge 1$) implies the existence of symmetric power functoriality ($\mathrm{SP}_n$ for all $n \geq 1$). Using the known cases of $\mathrm{TP}_r$, this establishes $\mathrm{SP}_n$ for integers $n \le 9$ and even integers $n \le 26$. 

\subsection*{The new contribution of this paper}
In this paper, we are able to carry out the entire strategy of \cite{Clo14} unconditionally, establishing $\mathrm{SP}_n$ for all $n \geq 1$, and therefore proving Theorem \ref{introthm_symmetric_power_functoriality}.

How do we do this? Taking in hand the cases of $\mathrm{LR}_{n}$ proved in \cite{Tho22a}, the main remaining problem, tackled in this paper, is to obtain the necessary cases of tensor product functoriality. We do not establish tensor product functoriality $\GL_2 \times \GL_r \to \GL_{2r}$ in general; rather, we do just enough in order to be able to proceed to our intended application, proving the residual automorphy of ${}^{\varphi_p} \overline{r}_{\pi, \iota} \otimes \Sym^{r-1} \overline{r}_{\pi, \iota}$ for suitable cuspidal, regular algebraic automorphic representations $\pi$ of $\GL_2(\A_F)$. We now sketch how this is achieved.

As a first step, we choose a congruent automorphic representation $\sigma$ (i.e. another regular algebraic automorphic representation of $\GL_2(\bA_F)$ such that $\overline{r}_{\sigma, \iota} \cong \overline{r}_{\pi, \iota}$). This representation is chosen so that ${}^{\varphi_p} r_{\pi, \iota} \otimes \Sym^{r-1} r_{\sigma, \iota}$ is Hodge--Tate regular (where now $\varphi_p \in G_{\bQ_p}$ is an arbitrary choice of arithmetic Frobenius lift), and therefore has the chance of being associated to a regular algebraic automorphic representation of $\GL_{2r}(\bA_F)$. (In practice we choose $\pi$ to have `parallel weight 2' and $\sigma$ to have `parallel weight 3'. Lemma \ref{lem_irreducibility_of_tensor_products} shows that the associated tensor products are then always irreducible, provided that $\pi, \sigma$ are non-CM.)

As a second step, we establish the automorphy of the tensor product $r_{\pi, \iota} \otimes \Sym^{r-1} r_{\sigma, \iota}$ (without the Frobenius twist in the coefficients). This is possible because there is an isomorphism
\begin{equation}\label{eqn_untwisted_congruence} \overline{r}_{\pi, \iota} \otimes \Sym^{r-1} \overline{r}_{\sigma, \iota} \cong (\det \overline{r}_{\pi, \iota} \otimes \Sym^{r-2} \overline{r}_{\pi, \iota}) \oplus \Sym^r \overline{r}_{\pi, \iota}, 
\end{equation}
and we can therefore use our automorphy lifting theorem for residually reducible representations, verifying the residual automorphy by induction (i.e.~knowing $\mathrm{SP}_{r-1}$ and $\mathrm{SP}_{r+1}$ -- and without knowing any instances of tensor product functoriality!). 

An ideal situation would then be if $\pi$, $\sigma$ could be chosen so that their fields of definition $K_\pi$, $K_\sigma$ were linearly disjoint over $\bQ$. (Here we recall that the group $\Aut(\bC)$ acts on the set of cuspidal, regular algebraic automorphic representations of $\GL_n(\A_F)$ by acting on the isomorphism class of the associated admissible representation of $\GL_n(\A_F^\infty)$. The orbit is finite and the field of definition is a number field contained in $\bC$. These results, which are due to Clozel \cite{Clo90}, are recalled below in \S \ref{subsec_preparations}.) We could then find an automorphism $\gamma \in \Aut(\bC)$ which acts as $\iota \varphi_p \iota^{-1}$ on $K_\pi$ and as the identity on $K_\sigma$. This automorphism would transform the automorphic representation of $\GL_{2r}(\bA_F)$ with associated Galois representation $r_{\pi, \iota} \otimes \Sym^{r-1} r_{\sigma, \iota}$ to one with the correct residual representation. Unfortunately, we don't know how to arrange the coefficient fields to have this property. 

Instead, we give a new argument to deduce the automorphy of $r_{{}^\gamma \pi, \iota} \otimes \Sym^{r-1} r_{\sigma, \iota}$ from the automorphy of the untwisted tensor product $r_{\pi, \iota} \otimes \Sym^{r-1} r_{\sigma, \iota}$, for any element $\gamma \in \Aut(\bC)$. Let $\widetilde{K}_\pi$ denote the Galois closure of the field of definition $K_\pi \subset \bC$. Then $\Gamma_\pi = \Gal(\widetilde{K}_\pi / \bQ)$ is generated by inertia groups at varying primes (because no extension of $\bQ$ is everywhere unramified). We can therefore write the image of $\gamma$ in $\Gamma_\pi$ as $\gamma = \delta_1 \dots \delta_s$ for prime numbers $l_1, \dots, l_s$, places $w_1, \dots, w_s$ of $\widetilde{K}_\pi$ above $l_1, \dots, l_s$, and elements $\delta_1, \dots, \delta_s$ of the inertia groups $I_{w_1 / l_1}, \dots, I_{w_s / l_s}$. We prove the automorphy of $r_{{}^\gamma \pi, \iota} \otimes \Sym^{r-1} r_{\sigma, \iota}$ by induction on $s \geq 0$, the base case $s = 0$ (corresponding to $\gamma = 1$) having already been established. 

In general, if $\gamma = \delta \gamma'$ and $r_{{}^{\gamma'} \pi, \iota} \otimes \Sym^{r-1} r_{\sigma, \iota}$ is automorphic, and $\delta \in I_{w / l}$, then we choose a isomorphism $\iota_l : \overline{\bQ}_l \to \bC$ such that $\iota_l^{-1}$ induces the place $w$ of $\widetilde{K}_\pi$. The automorphy of $r_{{}^\gamma \pi, \iota} \otimes \Sym^{r-1} r_{\sigma, \iota}$ is equivalent to that of $r_{{}^\gamma \pi, \iota_l} \otimes \Sym^{r-1} r_{\sigma, \iota_l}$ (because they are members of the same compatible system).
Moreover, there is an isomorphism
\[ r_{{}^{\delta \gamma'} \pi, \iota_l} \cong {}^{\widetilde{\delta}} r_{{}^{\gamma'} \pi, \iota_l}, \]
where $\widetilde{\delta} \in \Gal(\overline{\bQ}_l / \bQ_l)$ is an element, acting trivially on the residue field $\overline{\bF}_l$, such that the restriction of $\iota_l \widetilde{\delta} \iota_l^{-1} \in \Aut(\bC)$ to $\widetilde{K}_\pi$ equals $\delta$. We there obtain a congruence
\[  \overline{r}_{{}^{\delta \gamma'} \pi, \iota_l} \cong \overline{r}_{{}^{\gamma'} \pi, \iota_l}, \]
and hence a  congruence
\[  \overline{r}_{{}^{\gamma} \pi, \iota_l} \otimes \Sym^{r-1} \overline{r}_{\sigma, \iota_l} \cong \overline{r}_{{}^{\gamma'} \pi, \iota_l} \otimes \Sym^{r-1} \overline{r}_{\sigma, \iota_l} . \]
We can now again hope to use an automorphy lifting theorem to propagate automorphy from $r_{{}^{\gamma'} \pi, \iota_l} \otimes \Sym^{r-1} r_{\sigma, \iota_l}$ to $r_{{}^\gamma \pi, \iota_l} \otimes \Sym^{r-1} r_{\sigma, \iota_l}$. This is possible, using a `functoriality lifting theorem' for tensor products, Theorem \ref{thm_FLT_for_TP}, which is modelled after an analogous result for symmetric powers \cite[Theorem 2.1]{New21b}. Note that we might need to use congruences modulo $l$ for each prime $l$ which is ramified in $K_\pi$. We have no control at all over the set of such primes, but luckily we can prove a functoriality lifting theorem that works in this level of generality. 

\subsection*{Functoriality lifting theorems}

We sketch the idea of a functoriality lifting theorem. The name was introduced in the second author's CDM lectures 
\cite[Part II]{thorne-cdm} and a sketch of the symmetric powers case appears in the first author's expository article \cite{newton-iisc} (with the alternative terminology `relative modularity lifting theorem'). Suppose we have split reductive groups $G, H$ over a number field $F$, with Langlands dual groups $\widehat{G}, \widehat{H}$ defined over $\Zpbar$, and a homomorphism $f : \widehat{G}\to \widehat{H}$. In the examples relevant to symmetric powers and tensor product functoriality $G$ and $H$ will be general linear groups, and we can make sense of the statement that a Galois representation with image in $\widehat{G}(\Zpbar)$ or $\widehat{H}(\Zpbar)$ is automorphic. An idealized form of a functoriality lifting theorem is as follows:

\emph{Suppose $\rho_1, \rho_2 : G_F \to \widehat{G}(\Zpbar)$ are two automorphic Galois representations with isomorphic residual representation $\rhobar_1\cong \rhobar_2: G_F \to \widehat{G}(\Fpbar)$. If $f\circ\rho_1 : G_F \to \widehat{H}(\Zpbar)$ is automorphic, then so is $f\circ\rho_2$.}

Naturally, we have to impose many technical assumptions before we can unconditionally prove such a statement. The key advantage over standard modularity lifting theorems is that we can obtain results when $\rhobar_1$ is irreducible but $f\circ\rhobar_1$ is not assumed to be irreducible. A relevant example here is when $G = \GL_2 \times \GL_r$, $H = \GL_{2r}$, $f$ is the tensor product of standard representations of each factor, and $\rho_1 = (r_{\pi,\iota},\Sym^{r-1}r_{\sigma,\iota})$ with $\pi$ and $\sigma$ congruent automorphic representations as above. In this example, we have already observed in (\ref{eqn_untwisted_congruence}) that $f\circ\rhobar_1$ is necessarily reducible. We prove our functoriality lifting theorem by combining the Taylor--Wiles--Kisin patching method for $\widehat{G}$ with our results on vanishing of adjoint Bloch--Kato Selmer groups for the $\widehat{H}$-valued representation $f\circ\rho_1$ \cite{New19a,Tho22} (in fact, \cite{New19a} is sufficient for our purposes here). These vanishing results only assume irreducibility of the characteristic $0$ representation $(f\circ\rho_1) \otimes \Qpbar$. More precisely, with $G = \GL_2 \times \GL_r$ we will be taking tensor products with a fixed $r$-dimensional Galois representation, so we only need to carry out patching with respect to the $\GL_2$ factor of the dual group. 

\subsection*{The relation with the case $F = \bQ$}
The authors have previously established Theorem \ref{introthm_symmetric_power_functoriality} in the special case $F = \Q$ \cite{New21a,New21b}. The proof given in this paper is new even in this case. The main tool in \cite{New21a} was a result on the analytic continuation of symmetric power functoriality in families of overconvergent $p$-adic  modular forms, which was applied together with the explicit and simple structure,  discovered by Buzzard--Kilford, of the $2$-adic Coleman--Mazur eigencurve of tame level 1 \cite{Buz05}. This paper does not use analytic continuation, or the explicit determination of any space of modular forms of fixed level or weight. 

As we explained above, the main technical result of \cite{New21b} inspired the functoriality lifting theorem for tensor products that we use here. Both these results, together with the analytic continuation results of \cite{New21a}, rely on the vanishing of adjoint Bloch--Kato Selmer groups proved in \cite{New19a}. 

\subsection{Organization of this paper}

The first section, \S \ref{sec_comforting_lemma} is preparatory. We clarify the meaning of the existence of $\Sym^{n-1} \pi$, and give several definitions and lemmas that will be used repeatedly throughout the body of the paper. In particular, we show that any representation of the form $r_{\pi, \iota} \otimes \Sym^{r-1} r_{\sigma, \iota}$, where $\pi, \sigma$ are cuspidal, regular algebraic automorphic representations of $\GL_2(\A_F)$, $F$ a totally real field, of `parallel weight 2' and `parallel weight 3', respectively, and without CM, is irreducible, and should therefore be associated to a \emph{cuspidal} automorphic representation of $\GL_{2r}(\A_F)$. This observation plays an important motivating role.

In \S \ref{sec:aut untwisted}, we prove our first main result, showing that when $\pi, \sigma$ are chosen as in the previous paragraph and are \emph{congruent} (and satisfy certain auxiliary local conditions), then the tensor product $r_{\pi, \iota} \otimes \Sym^{r-1} r_{\sigma, \iota}$ is automorphic. The approach is based on the congruence (\ref{eqn_untwisted_congruence}), together with the automorphy lifting theorems of \cite{jackreducible, All20} and the level-raising theorems proved in \cite{Tho22a}. 

In \S \ref{sec:flt tp}, we prove our functoriality lifting theorem for tensor product representations; given cuspidal, regular algebraic automorphic representations $\pi, \pi'$ of $\GL_2(\A_F)$ and $\tau$ of $\GL_r(\A_F)$ for some $r \geq 1$, this makes it possible to propagate the automorphy of $r_{\pi, \iota} \otimes r_{\tau, \iota}$ to $r_{\pi', \iota} \otimes r_{\tau, \iota}$, given a congruence $\overline{r}_{\pi, \iota} \cong \overline{r}_{\pi', \iota}$. In \S \ref{sec: aut twisted}, we apply this to prove our second main result, the automorphy of the twisted tensor product $r_{{}^\gamma \pi, \iota} \otimes \Sym^{r-1} r_{\sigma, \iota}$, using the approach via coefficient fields sketched above. We obtain a complete result only under stringent local conditions, that ensure that every member of the compatible system associated to $\pi$ has a residual representation with large image. This is needed in order to be able to apply the main theorem of \S \ref{sec:flt tp}. 

Finally, in \S \ref{sec: the end}, we tie this result back in to the inductive scheme laid out in \cite{Clo14}, based on Bertrand's postulate. The main task is to show that the local conditions imposed appearing in \S \ref{sec: aut twisted} are sufficiently general --- we use a chain of congruences and soluble base changes to replace a general $\pi$ by one satisfying the necessary local conditions; given the techniques already developed in \cite{Clo14}, this is a mopping-up exercise, and together with the work done already quickly leads to the proof of our main Theorem \ref{introthm_symmetric_power_functoriality}. 

\subsection*{Acknowledgments} This paper represents the completion of a strategy first outlined by Laurent Clozel and JT in \cite{Clo14}. We would like to again express our thanks to Clozel for his inspirational contributions to the theory of automorphic forms. We would like to thank Toby Gee for useful comments on an earlier draft of this paper. Finally, we would like to thank the anonymous referees for their careful reading of our paper. 

JN was supported by a UKRI Future Leaders Fellowship, grant MR/V021931/1. For the purpose of Open Access, the authors have applied a CC BY public copyright licence to any Author Accepted Manuscript (AAM) version arising from this submission.

\subsection{Notation}\label{subsec_notation}

If $F$ is a perfect field, we  fix an algebraic closure $\overline{F} / F$ and write $G_F$ for the absolute Galois group of $F$ with respect to this choice. When the characteristic of $F$ is not equal to $p$, we write $\epsilon : G_F \to \bbZ_p^\times$ for the $p$-adic cyclotomic character. We write $\zeta_n \in \overline{F}$ for a fixed choice of primitive $n^\text{th}$ root of unity (when this exists). If $F$ is a number field, then we will also fix embeddings $\overline{F} \to \overline{F}_v$ extending the map $F\to F_v$ for each place $v$ of $F$; this choice determines a homomorphism $G_{F_v} \to G_F$. When $v$ is a finite place, we will write $W_{F_v} \subset G_{F_v}$ for the Weil group, $\cO_{F_v} \subset F_v$ for the valuation ring, $\varpi_v \in \cO_{F_v}$ for a fixed choice of uniformizer, $\Frob_v \in G_{F_v}$ for a fixed choice of (geometric) Frobenius lift, $k(v) = \cO_{F_v} / (\varpi_v)$ for the residue field, and $q_v = \# k(v)$ for the cardinality of the residue field. If $R$ is a ring and  $\alpha \in R^\times$, then we write $\operatorname{ur}_\alpha : W_{F_v} \to R^\times$ for the unramified character which sends $\Frob_v$ to $\alpha$. When $v$ is a real place, we write $c_v \in G_{F_v}$ for complex conjugation. If $S$ is a finite set of finite places of $F$ then we write $F_S / F$ for the maximal subextension of $\overline{F}$ unramified outside $S$ and $G_{F, S} = \Gal(F_S / F)$. If $A$ is a local ring, we write $\ffrm_A$ for the maximal ideal of $A$.

If $p$ is a prime, then we call a coefficient field a finite extension $E / \bbQ_p$ contained inside our fixed algebraic closure $\overline{\bbQ}_p$, and write $\cO$ for the valuation ring of $E$, $\varpi \in \cO$ for a fixed choice of uniformizer, and $k = \cO / (\varpi)$ for the residue field. We write $\cC_\cO$ for the category of complete Noetherian local $\cO$-algebras with residue field $k$. If $G$ is a profinite group and $\rho : G \to \GL_n(\overline{\bQ}_p)$ is a continuous representation, then we write $\overline{\rho} : G \to \GL_n(\overline{\bF}_p)$ for the associated semisimple residual representation (which is well-defined up to conjugacy). If $F$ is a number field, $v$ is a finite place of $F$, and $\rho_v, \rho_v' : G_{F_v} \to \GL_n(\overline{\bZ}_p)$ are continuous representations, which are potentially crystalline if $v | p$, then we use the notation $\rho \sim \rho'$ established in \cite[\S 1]{BLGGT} (which implies that $\rho_v \text{ mod }\ffrm_{\overline{\bZ}_p} \cong \rho'_v \text{ mod } \ffrm_{\overline{\bZ}_p}$, and indicates that these two representations define points on a common component of a suitable deformation ring). We adopt the convention of \emph{loc.~cit.} that if $\rho, \rho' : G_F \to \GL_n(\overline{\bQ}_p)$ are continuous representations such that $\overline{\rho}$, $\overline{\rho}'$ are irreducible, then $\rho|_{G_{F_v}} \sim \rho'|_{G_{F_v}}$ means $\rho_v \sim \rho'_v$, where $\rho_v$ is the restriction to $G_{F_v}$ of any conjugate of $\rho$ which takes values in $\GL_n(\overline{\bZ}_p)$ (and similarly for $\rho'_v$). 

We write $T_n \subset B_n \subset \GL_n$ for the standard diagonal maximal torus and upper-triangular Borel subgroup. Let $K$ be a non-archimedean characteristic $0$ local field, 
and let $\Omega$ 
be an algebraically 
closed field of characteristic 0. If $\rho : G_K \to \GL_n(\overline{\bbQ}_p)$ is a continuous 
representation (which is de Rham if $p$ equals the residue characteristic of 
$K$), then we write $\mathrm{WD}(\rho) = (r, N)$ for the associated 
Weil--Deligne representation of $\rho$, and $\mathrm{WD}(\rho)^{F-ss}$ for its 
Frobenius semisimplification. We use the cohomological normalization of 
class field theory: it is the isomorphism $\Art_K: K^\times \to W_K^{ab}$ which 
sends uniformizers to geometric Frobenius elements. When $\Omega = \bC$, we 
have the local Langlands correspondence $\rec_{K}$ for $\GL_n(K)$: a bijection 
between the sets of isomorphism classes of irreducible, admissible 
$\bC[\GL_n(K)]$-modules and Frobenius-semisimple Weil--Deligne representations 
over $\bC$ of rank $n$. In general, we have the Tate normalization of the local 
Langlands correspondence 
for $\GL_n$ as described in \cite[\S 
2.1]{Clo14}. When $\Omega = 
\bC$, we have $\rec^T_K(\pi) = \rec_K(\pi \otimes | \cdot |^{(1-n)/2})$. 

If $F$ is a number field and $\chi : F^\times \backslash \A_F^\times \to \bC^\times$ is a Hecke character of type $A_0$ (equivalently: algebraic), then for any isomorphism $\iota : \overline{\bQ}_p \to \bC$ there is a continuous character $r_{\chi, \iota} : G_F \to \overline{\bQ}_p^\times$ which is de Rham at the places $v | p$ of $F$ and such that for each finite place $v$ of $F$, $\mathrm{WD}(r_{\chi, \iota}) \circ \Art_{F_v} = \iota^{-1} \chi|_{F_v^\times}$. Conversely, if $\chi' : G_F \to \overline{\bQ}_p^\times$ is a continuous character which is de Rham and unramified at all but finitely many places, then there exists a Hecke character $\chi : F^\times \backslash \A_F^\times \to \bC^\times$ of type $A_0$ such that $r_{\chi, \iota} = \chi'$. 

If $F$ is a totally real or CM number field and $\pi$ is an automorphic representation of $\GL_n(\A_F)$, we say that $\pi$ is regular algebraic if $\pi_\infty$ has the same infinitesimal character as an irreducible algebraic representation $W$ of $(\Res_{F/ \bQ} \GL_n)_\bC$. We identify $X^\ast(T_n)$ with $\Z^n$ in the usual way, and write $\Z^n_+ \subset \Z^n$ for the subset of weights which are $B_n$-dominant. If $W^\vee$ has highest weight $\lambda = (\lambda_\tau)_{\tau \in \Hom(F, \bC)} \in (\Z^n_+)^{\Hom(F, \bC)}$, then we say that $\pi$ has weight $\lambda$. We refer to \cite[\S 1]{blght} for the definition of what it means for $\pi$ to be RAESDC or RAECSDC (that is, regular algebraic, essentially (conjugate) self-dual, and cuspidal). If $L / F$ is a cyclic extension, then we write $\mathrm{BC}_{L / F}(\pi)$ for the base change of $\GL_n(\bA_F)$ (see \cite{MR1007299}): it is an automorphic representation of $\GL_n(\bA_L)$ with the property that for any place $w$ of $L$ lying above a place $v$ of $F$, we have $\rec_{L_w} \mathrm{BC}_{K / F}(\pi)_w \cong (\rec_{F_v} \pi_v )|_{W_{L_w}}$.

If $\pi$ is RAESDC or RAECSDC, then for any isomorphism $\iota : \overline{\bQ}_p \to \bC$ there exists a continuous, semisimple representation $r_{\pi, \iota} : G_F \to \GL_n(\overline{\bQ}_p)$ such that for each finite place $v$ of $F$, $\mathrm{WD}(r_{\pi, \iota}|_{G_{F_v}})^{F-ss} \cong \rec_{F_v}^T(\iota^{-1} \pi_v)$ (see e.g.\ \cite{Caraianilp}). (When $n = 1$, this is compatible with our existing notation.) We use the convention that the Hodge--Tate weight of the cyclotomic character is $-1$. Thus if $\pi$ is of weight $\lambda$, then for any embedding $\tau : F \to \overline{\bQ}_p$ the $\tau$-Hodge--Tate weights of $r_{\pi, \iota}$ are given by 
\[ \mathrm{HT}_\tau(r_{\pi, \iota}) = \{ \lambda_{\iota \tau, 1} + (n-1), \lambda_{\iota \tau, 2} + (n-2), \dots, \lambda_{\iota \tau, n} \}. \] Suppose $\pi$ is RAECSDC, with $\pi^\vee \cong \pi^c \otimes(\chi\circ \bN_{F/F^+}\circ\det)$ for $\chi: (F^+)^\times \backslash \A_{F^+}^\times \to \bC^\times$ a continuous character and $\chi_v(-1)= (-1)^n$ for all $v|\infty$. Then \cite[Theorem 2.1.1(1)]{BLGGT} implies that $r_{\pi,\iota}$ extends to a continuous homomorphism $r_{\pi,\iota}:G_{F^+} \to \cG_n(\Qpbar)$ with multiplier $\nu\circ r_{\pi,\iota} = \epsilon^{1-n}r_{\chi,\iota}$. Here $\cG_n$ is the (disconnected) algebraic group introduced in \cite{cht}.

A particular case of interest is when $n = 2$ and $F$ is totally real. In this case, the automorphic representations which are of weight $\lambda = 0$ may be generated by Hilbert modular forms of parallel weight $2$. In \cite{New21b}, we said such automorphic representations were `of weight 2'. We do not use this terminology in this paper, where automorphic representations on higher rank general linear groups play a more significant role. We remark that any cuspidal, regular algebraic automorphic representation of $\GL_2(\A_F)$ is necessarily RAESDC (as any automorphic representation of $\GL_2(\A_F)$ is self-dual up to twisting by the central character). 

If $F$ is a number field, $G$ is a reductive group over $F$, $v$ is a finite place of $F$, and $U_v$ is an open compact subgroup of $G(F_v)$, then we write $\cH(G(F_v), U_v)$ for the convolution algebra of compactly supported $U_v$-biinvariant functions $f : G(F_v) \to \bZ$ (convolution defined with respect to the Haar measure on $G(F_v)$ which gives $U_v$ volume 1). Then $\cH(G(F_v), U_v)$ is a free $\bZ$-module, with basis given by the characteristic functions $[U_v g_v U_v]$ of double cosets for $g_v \in U_v \backslash G(F_v) / U_v$.

If $1 \leq i \leq n$, let $\alpha_{\varpi_v, i} = \diag(\varpi_v, \dots, \varpi_v, 1, \dots, 1) \in \GL_n(F_v)$ (where there are $i$ occurrences of $\varpi_v$ on the diagonal). We define 
\[ T_v^{(i)} = [\GL_n(\cO_{F_v}) \alpha_{\varpi_v, i} \GL_n(\cO_{F_v})] \in \cH(\GL_n(F_v), \GL_n(\cO_{F_v})). \]
We write $\Iw_v \subset \GL_n(\cO_{F_v})$ for the standard Iwahori subgroup (elements which are upper-triangular modulo $\varpi_v$) and $\Iw_{v, 1} \subset \Iw_v$ for the kernel of the natural map $\Iw_v \to (k(v)^\times)^n$ given by reduction modulo $\varpi_v$, then projection to the diagonal. If $U_v \subset \Iw_v$ is a subgroup containing $\Iw_{v, 1}$, and $1 \leq i \leq n$, then we define
\[ U_{\varpi_v}^{(i)} = [U_v \alpha_{\varpi_v, i} U_v] \in \cH(\GL_n(F_v), U_v). \]

\section{Some comforting lemmas}\label{sec_comforting_lemma}

In this section, we give some lemmas and definitions that will be of use throughout the paper. 

We first clarify the meaning of the phrase ``$\Sym^n \pi$ exists''. If $F$ is a totally real or CM number field, and $p$ is a prime number, we say that a continuous representation $\rho : G_F \to \GL_n(\overline{\bQ}_p)$ is automorphic if there exists a RAESDC or RAECSDC automorphic representation $\pi$ of $\GL_n(\bA_F)$ and an isomorphism $\iota : \overline{\bQ}_p \to \bC$ such that $\rho \cong r_{\pi, \iota}$. This is not the most general definition that one might adopt (for example, it forces $\rho$ itself to be (conjugate-)self-dual up to twist and conjecturally forces $\rho$ to be irreducible) but it will be the one most useful for our purposes.
\begin{lemma}\label{lem_the_meaning_of_existence}
Let $F$ be a totally real field, and let $\pi$ be a RAESDC automorphic representation of $\GL_2(\A_F)$ that is non-CM, in the sense that for any non-trivial character $\chi : F^\times \backslash \bA_F^\times \to \bC^\times$, $\pi \not\cong \pi \otimes (\chi \circ \det)$. Let $n \geq 1$. Then the following are equivalent:
\begin{enumerate}
\item There exists a cuspidal automorphic representation $\Pi$ of $\GL_n(\bA_F)$ such that for each place $v$ of $F$, we have
\[ \rec_{F_v}(\Pi_v) \cong \Sym^{n-1} \circ\, \rec_{F_v}(\pi_v). \]
\item For each prime $p$ and isomorphism $\iota : \overline{\bQ}_p \to \bC$, the representation $\Sym^{n-1} r_{\pi, \iota}$ is automorphic.
\item For some prime $p$ and isomorphism $\iota : \overline{\bQ}_p \to \bC$,  the representation $\Sym^{n-1} r_{\pi, \iota}$ is automorphic.
\end{enumerate} 
\end{lemma}
\begin{proof}
Certainly we have $(1) \Rightarrow (2) \Rightarrow (3)$. For $(3) \Rightarrow (1)$, we are given that there is a RAESDC automorphic representation $\Pi$ of $\GL_n(\bA_F)$ and an isomorphism $\jmath : \overline{\bQ}_p \to \bC$ such that $r_{\Pi, \jmath} \cong \Sym^{n-1} r_{\pi, \iota}$. Let $\sigma = \iota \jmath^{-1}$. Then (cf.~Lemma \ref{lem_galois_conjugation} below) there are isomorphisms
\[ r_{{}^\sigma \Pi, \iota} \cong r_{\Pi, \jmath} \cong \Sym^{n-1} r_{\pi, \iota}. \]
Replace $\Pi$ by ${}^\sigma \Pi$. Local-global compatibility at finite places then implies that if $v$ is a finite place of $F$, then $\rec_{F_v} \Pi_v \cong \Sym^{n-1} \circ \,\rec_{F_v} \pi_v$. The Weil group representation $\rec_{F_v} \Pi_v$ for an infinite place $v$ is determined by the labelled Hodge--Tate weights of $r_{\Pi,\iota}$ (cf.~the proof of \cite[Theorem 8.1]{Tho22a}), so we have the desired relation at infinite places also. This completes the proof. 
\end{proof}
The following lemma will come in handy on occasion to verify irreducibility conditions. Recall that a representation $\rho$ of an absolute Galois group $G_F$ is said to be strongly irreducible if its restriction to $G_L$ is irreducible for every finite extension $L/F$. According to \cite[Example 2.29]{New19a}, this condition holds if $F$ is totally real and $\rho = r_{\pi, \iota}$ for a RAESDC automorphic representation $\pi$ of $\GL_2(\A_F)$ without CM. 
\begin{lemma}\label{lem_irreducibility_of_tensor_products}
Let $F$ be a number field. Let $\rho, \rho' : G_F \to \GL_2(\overline{\bQ}_p)$ be continuous representations which are strongly irreducible. Suppose that there is an embedding $\tau : F \to \overline{\bQ}_p$ inducing a place $v$ of $F$ such that $\rho|_{G_{F_v}}$, $\rho'|_{G_{F_v}}$ are Hodge--Tate and $\mathrm{HT}_\tau(\rho) = \{ a, b \}$, $\mathrm{HT}_\tau(\rho') = \{ a', b'\}$, with $a - b \neq \pm (a' - b')$. Then for any $m, m' \geq 1$, $(\Sym^m \rho \otimes \Sym^{m'} \rho')|_{G_{F(\zeta_{p^\infty})}}$ is strongly irreducible.
\end{lemma}
\begin{proof}
We use the basic fact that if $\alpha : H \to H'$ is a morphism of linear algebraic groups over $\overline{\bQ}_p$, and $K \subset H(\overline{\bQ}_p)$ is a subgroup, then the Zariski closure in $H'$ of $\alpha(K)$ is the image under $\alpha$ of the Zariski closure of $K$ in $H$ (see \cite[Ch. I, \S 2.1(f)]{Bor91}). 

Let $\operatorname{Proj} \rho : G_F \to \PGL_2(\overline{\bQ}_p)$ denote the associated projective representation (and similarly for $\rho'$). Then $\operatorname{Proj} \rho$ has Zariski dense image in $\PGL_2$. Indeed, the Zariski closure $G_\rho \subset \GL_2$ of $\rho(G_F)$ is a reductive subgroup (as it has a faithful irreducible representation). The identity component $G_\rho^\circ$ is connected reductive and still acts irreducibly (because $\rho$ is strongly irreducible), so it must be equal either to $\SL_2$ or $\GL_2$. In either case the image of $G_\rho$ in $\PGL_2$ is $\PGL_2$.

Let $G \subset \PGL_2 \times \PGL_2$ denote the Zariski closure of the image of 
\[ \operatorname{Proj} \rho \times \operatorname{Proj} \rho' : G_F \to \PGL_2(\overline{\bQ}_p) \times \PGL_2(\overline{\bQ}_p). \]
Then $G$ is a reductive subgroup which projects surjectively to each factor. By Goursat's lemma for the group $G(\Qpbar)$, we either have $G = \PGL_2 \times \PGL_2$ or $G(\Qpbar)$ is the graph of an automorphism of $\PGL_2(\Qpbar)$. In the second case, the first projection map $G \to \PGL_2$ is a bijection on $\Qpbar$-points and is therefore an isomorphism of algebraic groups. We deduce that $G$ is the graph of an automorphism of $\PGL_2$, and $\operatorname{Proj}  \rho, \operatorname{Proj}  \rho'$ are conjugate by this automorphism. Every automorphism of $\PGL_2$ is inner, so this would imply that $\Ad \rho, \Ad \rho'$ (where $\Ad$ denotes the adjoint action on the Lie algebra $\mathfrak{sl}_2$) have the same Hodge--Tate weights, a contradiction. 

To complete the proof, note that the Zariski closure of $\rho(G_{F(\zeta_{p^\infty})}) \times \rho'(G_{F(\zeta_{p^\infty})})$ in $\GL_2 \times \GL_2$ contains the derived group of the Zariski closure of $\rho(G_F) \times \rho'(G_F)$ (by \cite[Ch. I, \S 2.4, Proposition]{Bor91}). This latter group contains $\SL_2 \times \SL_2$, as it surjects to $\PGL_2 \times \PGL_2$. Since $\Sym^m \otimes \Sym^{m'}$ is irreducible as a representation of $\SL_2 \times \SL_2$, this completes the proof. 
\end{proof}
We next give a criterion for certain subgroups of $\GL_n(\overline{\bF}_p)$ to be adequate, in the sense of \cite[Definition 2.3]{jack}. Let $k \leq \overline{\bF}_p$ be a finite field, and let  $G \leq \GL_n(k)$ be a subgroup such that each element of $G$ has all its eigenvalues in $k$. If $G$ is adequate, then it is absolutely irreducible, and the converse holds when $p > 2n + 2$ \cite[Lemma 2.4]{jack}. The condition that residual Galois representations have adequate image frequently appears in the statements of automorphy lifting theorems, due to its connection with the Taylor--Wiles method. The following result, which is  \cite[Lemma 7.3]{Tho22a}, gives conditions under which subgroups arising by taking symmetric powers are adequate.
\begin{lemma}\label{lem_tensor_adequacy}
Let $p \geq 5$ be a prime and let $\varphi_p \in G_{\bQ_p}$ be a lift of arithmetic Frobenius. There exists an integer $a_0 = a_0(p) \geq 3$ such that if $a \geq a_0$ and  $G \leq \GL_2(\overline{\bF}_p)$ is a finite subgroup which contains a conjugate of $\SL_2(\bF_{p^a})$, then $G$ has the following properties:
\begin{enumerate}
    \item For each $0 < r  < p$, the image of the homomorphism $\Sym^{r-1} : G \to \GL_r(\overline{\bF}_p)$ is adequate,  in the sense of \cite[Definition 2.3]{jack}. 
    \item For each $0 < r < p$, the image of the homomorphism ${}^{\varphi_p} \mathrm{Std} \otimes \Sym^{r-1} : G \to \GL_{2r}(\overline{\bF}_p)$ is adequate.
    \item If $H \leq G$ is a subgroup of index $[G : H] < 2p$, then $H$ contains a conjugate of $\SL_2(\bF_{p^a})$. Consequently, $H$ also satisfies (1) and (2).
\end{enumerate}
\end{lemma}
We next give a definition that singles out a class of irreducible admissible representations of $\GL_2(F_v)$ ($F_v$ non-archimedean) with useful associated Langlands parameters. 
\begin{defn}\label{def_tamely_dihedral}
Let $F$ be a number field, let $v$ be a  finite place of $F$, and let $p$ be an odd prime such that $q_v \equiv -1 \text{ mod }p$. We say that an irreducible admissible $\bC[\GL_2(F_v)]$-representation $\pi_v$ is tamely dihedral of order $p$ if there is an isomorphism $\rec_{F_v}(\pi_v) \cong \Ind_{W_{F_v}}^{W_{F'_v}} \chi_v$, where $F'_v / F_v$ is the quadratic unramified extension and $\chi_v : W_{F'_v} \to \bC^\times$ is a character such that $\chi_v|_{I_{F'_v}}$ has order $p$. 
\end{defn}
Note in particular that if $\pi_v$ is tamely dihedral of order $p$, then $\rec_{F_v}(\pi_v)$ is irreducible and hence $\pi_v$ is supercuspidal. In applications, $\pi_v$ will be the local component of an automorphic representation $\pi$ of $\GL_2(\A_F)$. These representations $\pi_v$ are useful since they can frequently be forced to appear by the use of level-raising congruences. On the other hand, the presence of suitable tamely dihedral local factors can be used to show that the residual Galois representations associated to $\pi$ have large image --- see, for example, Proposition \ref{prop_large_image_everywhere}.

Finally, we recall the notion of an $\iota$-ordinary automorphic representation, and the related notion of ordinary Galois representation, following \cite[\S 2]{jackreducible}.
\begin{defn}\label{def_ordinary} Let $F$ be a totally real or CM number field, and let $p$ be a prime.
\begin{enumerate} 
\item Let $\pi$ be a RAESDC or RAECSDC automorphic representation of $\GL_n(\A_F)$ of weight $\lambda = (\lambda_\tau)_{\tau \in \Hom(F, \bC)}$, and let $\iota : \overline{\bQ}_p \to \bC$ be an isomorphism. Then $\pi$ is said to be $\iota$-ordinary if for each place $v | p$ of $F$ there exist smooth characters $\chi_{v, 1}, \dots, \chi_{v, n} : F_v^\times \to \bC^\times$ such that $\pi_v$ is a subquotient (necessarily the unique generic subquotient) of the normalized induction $i_{B_n(F_v)}^{\GL_n(F_v)} \chi_{v, 1} \otimes \dots \otimes \chi_{v, n}$, and for each $i = 1, \dots, n$ we have
\[ v_p( \iota \chi_i(\varpi_v) ) = \frac{1}{e_v} \sum_{\tau \in \Hom(F_v, \overline{\bQ}_p)}\left( \lambda_{\iota \tau, n + 1 - i} - \frac{n-1}{2} + i - 1\right), \]
where $v_p : \overline{\bQ}_p^\times \to \bQ$ is the valuation with $v_p(p) = 1$, and $e_v$ is the ramification index of $F_v / \bQ_p$.
\item Let $\rho : G_F \to \GL_n(\overline{\bQ}_p)$ be a continuous representation, and let $\lambda = (\lambda_\tau) \in (\Z^n_+)^{\Hom(F, \overline{\bQ}_p)}$. We say that $\rho$ is ordinary of weight $\lambda$ if for each place $v | p$ there exist continuous characters $\psi_{v, 1}, \dots, \psi_{v, n} : G_{F_v} \to \overline{\bQ}_p^\times$ and an isomorphism
\[ \rho|_{G_{F_v}} \sim \left( \begin{array}{cccc} \psi_{v, 1} & * & * & * \\ 0 & \psi_{v, 2} & * & * \\ \vdots & \ddots & \ddots & * \\ 0 & \cdots & 0 & \psi_{v, n} \end{array} \right), \]
where for each $i = 1, \dots, n$, the character 
\[ x \mapsto \psi_{v, i}(\Art_{F_v}(x)) \prod_{\tau \in \Hom(F_v, \overline{\bQ}_p)} \tau(x)^{\lambda_{\tau, n-i+1} + i - 1}\]
 of $\cO_{F_v}^\times$ has finite order. 
\end{enumerate} 
\end{defn}
According to \cite[Corollary 2.6]{jackreducible}, if $\pi$ is an $\iota$-ordinary automorphic representation of $\GL_n(\A_F)$, then $r_{\pi, \iota}$ is ordinary of weight $\iota \lambda$. Our arguments will frequently be restricted to representations which are ordinary (in either sense). In particular, our key automorphy lifting theorems in the residually reducible case (proved in \cite{jackreducible, All20}) apply only to ordinary representations, as do the level-raising results proved in \cite{Tho22a}. One of the tasks  in our final `mopping up' \S \ref{sec: the end} will be to show that the general case to be reduced to this one. The following lemma will be useful in this connection.
\begin{lemma}
Let $F$ be a totally real or CM number field, and let $\pi$ be a RAESDC or RAECSDC automorphic representation of $\GL_n(\A_F)$ of weight 0. Let $p$ be a prime such that for each place $v | p$, $\pi_v$ is a character twist of the Steinberg representation of $\GL_n(F_v)$. Then for any isomorphism $\iota : \overline{\bQ}_p \to \bC$, $\pi$ is $\iota$-ordinary.
\end{lemma}
\begin{proof}
See \cite[Lemma 5.6]{ger} (or carry out the easy check using Definition \ref{def_ordinary}). 
\end{proof}

\section{Automorphy of an untwisted tensor product}\label{sec:aut untwisted}

In this section we will carry out the first main step of our argument, namely the exploitation of the congruence (\ref{eqn_untwisted_congruence}) to prove the automorphy of an untwisted tensor product. We first introduce the notation necessary to give a precise statement, and then describe the structure of the proof. 

Let $F$ be a totally real number field, let $p \geq 5$ be a prime, let $0 < r < p$ be an integer, let $n = 2r$, let $\iota : \overline{\bQ}_p \to \bC$ be an isomorphism, and let $\pi$ be a regular algebraic, cuspidal automorphic representation of $\GL_2(\bA_F)$ satisfying the following conditions:
\begin{enumerate}
\item $\det r_{\pi, \iota} = \epsilon^{-1}$.
\item $\pi$ is of weight 0 and for each place $v | p$, $\pi_v$ is an unramified twist of the Steinberg representation; in particular, $\pi$ is $\iota$-ordinary.
\item There exists a place $v_0$ of $F$ such that $q_{v_0} \equiv -1 \text{ mod }p$ and $\pi_{v_0}$ is tamely dihedral of order $p$, in the sense of Definition \ref{def_tamely_dihedral}. 
\item $\overline{r}_{\pi, \iota}(G_F)$ contains a conjugate of $\SL_2(\bF_{p^a})$ for some $a > a_0(p)$, where $a_0(p)$ is as defined in the statement of Lemma \ref{lem_tensor_adequacy}. 
\item For each place $v \nmid p$ of $F$, $\pi_v$ is potentially unramified.
\end{enumerate}
Increasing $a$ if necessary, and replacing $\overline{r}_{\pi, \iota}$ by a conjugate, we can (and do) assume that the image of $\operatorname{Proj} \overline{r}_{\pi, \iota}$ is equal either to $\PSL_2(\bF_{p^a})$ or $\PGL_2(\bF_{p^a})$. Since $a_0(p) \geq 3$ by definition, we have $a > 3$.
\begin{lemma}\label{lem_hida_theory}
There exists a RAESDC automorphic representation $\sigma$ of $\GL_2(\bA_F)$ satisfying the following conditions:\begin{enumerate}
\item $\sigma$ is $\iota$-ordinary and for each embedding $\tau : F \to \overline{\bQ}_p$, $\mathrm{HT}_\tau(r_{\sigma, \iota}) = \{ 0, 2 \}$. 
\item We have $\det r_{\sigma, \iota} = \epsilon^{-2} \omega$, where $\omega$ denotes the Teichm\"uller lift of $\epsilon \text{ mod }p$.
\item There is an isomorphism $\overline{r}_{\sigma, \iota} \cong \overline{r}_{\pi, \iota}$.
\item $\sigma_{v_0}$ is an unramified twist of the Steinberg representation.
\item For each place $v \nmid p$ of $F$ such that $v \neq v_0$, $r_{\sigma, \iota}|_{G_{F_v}} \sim r_{\pi, \iota}|_{G_{F_v}}$.\end{enumerate}
\end{lemma}
Before giving the proof of Lemma \ref{lem_hida_theory}, we state the main theorem of this section. 
\begin{theorem}\label{thm_untwisted_tensor_product}
Let $0 < r < p$ be an integer and let $n = 2r$. Suppose that $\mathrm{SP}_{r-1}$ and $\mathrm{SP}_{r+1}$ hold. Let $\pi$ be the automorphic representation of $\GL_2(\A_F)$ introduced at the start of this section, and let $\sigma$ be an automorphic representation as in Lemma \ref{lem_hida_theory}. Then $r_{\pi, \iota} \otimes \Sym^{r-1} r_{\sigma, \iota}$ is automorphic, associated to an $\iota$-ordinary, RAESDC automorphic representation of $\GL_n(\bA_F)$.
\end{theorem} 
Here we remind the reader that the property $\mathrm{SP}_k$, stated in Conjecture \ref{conj_SP_n} of the introduction, asserts the existence of the $(k-1)$-fold symmetric power of any non-CM RAESDC automorphic representation of $\GL_2(\A_{F'})$, for any totally real field $F'$, in the several equivalent senses given in the statement of Lemma \ref{lem_the_meaning_of_existence}. By contrast, the representations $\pi$, $\sigma$ referred to in the statement of Theorem \ref{thm_untwisted_tensor_product} are the particular automorphic representations of $\GL_2(\A_F)$ that have been either given or whose existence is asserted by Lemma \ref{lem_hida_theory}. 
\begin{proof}[Proof of Lemma \ref{lem_hida_theory}]
This can be proved by a standard argument, along the lines of \cite[Proposition 3.1.7]{gee061}. We describe the salient details. First, we need to know the existence, for each $v \in S_p$, of a lift $\widetilde{r}_v : G_{F_v} \to \GL_2(\overline{\bZ}_p)$ of $\overline{r}_{\pi, \iota}|_{G_{F_v}}$ which is ordinary of Hodge--Tate weights $\{0, 2 \}$ and satisfies $\det \widetilde{r}_v|_{G_{F_v}} = \epsilon^{-2} \omega$. This can be proved by a global argument, using Hida theory as developed for Hilbert modular forms in \cite{Wil88}: put $\pi$ into a Hida family, take  a specialization in the appropriate weight to yield a new RAESDC automorphic representation $\pi'$, and set $\widetilde{r}_v = r_{\pi', \iota}|_{G_{F_v}}$.

We can then consider an appropriate deformation ring associated to $\overline{r}_{\pi, \iota}$. A convenient reference for us is \cite[\S 4.2]{Bel19}, according to which we need only specify the determinant of the lifts to be considered, a set of $S$ of places of $F$ where ramification is permitted, and for each $v \in S$ an irreducible component $\cC_v$ of an appropriate lifting ring (and therefore quotient $R_v$ of the universal lifting ring with $\Spec R_v = \cC_v$). We take lifts of determinant $\epsilon^{-2} \omega$, $S$ to be the set of places at which $\pi$ is ramified (which includes $S_p$ by assumption), and the irreducible component $\cC_v$ as follows:
\begin{itemize}
\item If $v \in S$ and $v \nmid p v_0$ then $\cC_v$ is the unique irreducible component of $R_v^\square$ containing the point determined by  $r_{\pi, \iota}|_{G_{F_v}}$ ($\cC_v$ is unique since this point is regular, by \cite[Proposition 1.2.2]{All16}).
\item If $v | p$ then $\cC_v$ is the unique irreducible component of $R_v^{\square, \tau_v, \mathbf{v}_v, N = 0}$ containing the lift $\widetilde{r}_v$ (here the notation $R_v^{\square, \tau_v, \mathbf{v}_v, N = 0}$ is that of \cite[Theorem 3.3.8]{Bel19}, denoting a potentially crystalline lifting ring with fixed inertial type $\tau_v$ and Hodge type $\mathbf{v}_v$, chosen to match the types of the lift $\widetilde{r}_v$; the uniqueness follows because, as stated in \emph{loc.~cit.}, $R_v^{\square, \tau_v, \mathbf{v}_v, N = 0}[1/p]$ is regular). An argument like the one in \cite[Lemma 3.10]{ger} then shows that every $\overline{\bQ}_p$-point of $\cC_v$ corresponds to an ordinary Galois representation (noting that any potentially semistable ordinary representation of the same Hodge type as $\widetilde{r}_v$ is necessarily potentially crystalline; comparing dimensions shows that the ordinary locus in $\Spec R_v^{\square, \tau_v, \mathbf{v}_v}$ is a union of irreducible components in $\Spec R_v^{\square, \tau_v, \mathbf{v}_v, N = 0}$).
\item If $v = v_0$, then $\cC_{v_0}$ is any irreducible component of $R_{v_0}^{\square, \mathbf{1}}$ on which $N$ is generically non-zero ($\mathbf{1}$ denotes the trivial inertial type). Since $\pi_{v_0}$ is tamely dihedral of order $p$, the reduction $\overline{r}_{\pi, \iota}|_{G_{F_{v_0}}}$ has the same semisimplification as $\Ind_{G_{F_{v_0}}}^{G_{F'_{v_0}}}\overline{\psi}$, where $F'_{v_0}/F_{v_0}$ is the quadratic unramified extension and $\overline{\psi}:G_{F'_{v_0}} \to \Fpbarx$ is a character whose inertial restriction $\overline{\psi}|_{I_{F'_{v_0}}}$ has order dividing $p$, and is therefore trivial. We deduce that $\overline{r}_{\pi, \iota}|_{G_{F_{v_0}}}$ is reducible, unipotently ramified, and Frobenius lifts have eigenvalues with ratio $-1 \equiv q_{v_0} \text{ mod }{p}$, so we can write down unipotently ramified lifts to $\overline{\bQ}_p$ by hand in order to see that $\cC_{v_0}$ exists.
\end{itemize}
We can now cite \cite[Proposition 4.2.6]{Bel19} to conclude that the deformation ring $R$ corresponding to this data has Krull dimension at least 1. We can then apply the Khare--Wintenberger method: \cite[Theorem 10.2]{jack} implies that $R$ is a finite $\cO$-algebra, therefore that $R[1/p]$ is non-zero, implying the existence of a lift $\rho : G_F \to \GL_2(\overline{\bQ}_p)$ unramified outside $S$ and with the prescribed local behaviour at places of $S$. An automorphy lifting theorem (for example, \cite[Theorem 9.1]{jack}, valid over a CM field, which can be applied by first making a soluble base change) can then be applied to conclude that $\rho = r_{\sigma, \iota}$ for some RAESDC automorphic representation $\sigma$ satisfying the requirements of the lemma. 
\end{proof}

The rest of \S \ref{sec:aut untwisted} will describe the proof of Theorem \ref{thm_untwisted_tensor_product}. We make some introductory remarks to orient the reader. First, the theorem is trivial if $r = 1$, so we assume for the remainder of \S \ref{sec:aut untwisted} that $r > 1$. The arguments will be carried out with RACSDC twists of our RAESDC representations over CM extensions of $F$, a detail we suppress in these remarks.

The basic idea is simple. The representation $r_{\pi, \iota} \otimes \Sym^{r-1} r_{\sigma, \iota}$ is Hodge--Tate regular and essentially self-dual, by construction, and irreducible, by Lemma \ref{lem_irreducibility_of_tensor_products}. One therefore expects it to be associated to a RAESDC automorphic representation of $\GL_n(\A_F)$. On the other hand, the isomorphism (\ref{eqn_untwisted_congruence}) shows that the associated residual representation is reducible, being a direct sum of two summands, which are residually automorphic, by our assumption that conjectures  $\mathrm{SP}_{r-1}$ and $\mathrm{SP}_{r+1}$ hold. We would like to conclude the automorphy of the tensor product, using an automorphy lifting theorem. 

However, the results of \cite{jackreducible,All20} require the Galois representation in question to be locally Steinberg at some prime-to-$p$ place, a condition that can never be satisfied for a tensor product of two representations of dimension $> 1$. This means we need to work harder. The local behaviour of the representations $\pi$, $\sigma$ at the place $v_0$ is chosen so that $(r_{\pi, \iota} \otimes \Sym^{r-1} r_{\sigma, \iota})|_{G_{F_{v_0}}}$ is indecomposable. One could use the existence of this place to ensure that the space of reducible deformations of the associated residual representation has low dimension,
 which is the main role of the locally Steinberg place in \cite{jackreducible, All20}. To be able to conclude using the theorems proved in these papers, we make use of various different auxiliary deformation problems defined over extensions $F_1 \subset F_2 \subset F_3$ of $F$ to reduce to the case of locally Steinberg deformation problems which we have already treated. The deformation problems will differ only in the local conditions at places over $v_0$ (which will split completely in $F_1$). We indicate the nature of these deformation problems here:
\begin{enumerate}
	\item $\cS_{F_1}$ will have a local deformation condition over $v_0$ corresponding to the representation $r_{\pi, \iota} \otimes \Sym^{r-1} r_{\sigma, \iota}$ which we are interested in.
	\item $\cS'_{F_1}$ will have a Steinberg-type local condition at places over $v_0$. The problems $\cS_{F_1}$ and $\cS'_{F_1}$ will be congruent, in the sense that the special fibres of their deformation rings will have isomorphic reduced subschemes.
	\item $\cS'_{F_2}$ will have a Steinberg-type local condition at the places over $v_0$. We will moreover choose $F_2$ so that if $v | v_0$ is a place of $F_2$ then $q_v \equiv 1 \text{ mod } p$ and the residual representation is trivial at $v$ (the local Steinberg lifting ring is then the one considered in \cite{tay, jackreducible}, where it is shown to be a domain). The local conditions defining $\cS'_{F_1}$, and $\cS'_{F_2}$ are therefore compatible, implying the existence of a morphism $R_{\cS'_{F_2}} \to R_{\cS'_{F_1}}$ of deformation rings, induced by `restriction of deformations to $G_{F_2}$'. (The existence of this morphism will be proved in Lemma \ref{lem_isomorphism_of_residue_rings} below.)
		\item $\cS_{F_3}$ will have a ``unipotent-ramification'' condition at places $v | v_0$ of $F_3$. Any lift of type $\cS'_{F_2}$ or $\cS_{F_1}$ will have a restriction to $F_3$ of type $\cS_{F_3}$. 
\end{enumerate}
The deformation problem $\cS'_{F_2}$ fits in to the framework of \cite{jackreducible,All20}, and we can therefore prove finiteness over an Iwasawa algebra $\Lambda$ of the corresponding deformation ring. From there we can deduce finiteness of  $R_{\cS'_{F_1}}$ and (using the isomorphism of reduced special fibres) $R_{\cS_{F_1}}$, and also show the existence of sufficiently many `generic' dimension 1, characteristic $p$ prime ideals in $R_{\cS_{F_1}}$ (which have several useful properties, including that the associated Galois representations are absolutely irreducible). The methods and results of \cite{jackreducible,All20} will then be sufficient to prove modularity of $r_{\pi, \iota} \otimes \Sym^{r-1} r_{\sigma, \iota}$, first considering its restriction to $F_3$ and then applying soluble descent.

Now we give the details of the proof. We first need to choose an auxiliary place, whose existence is the content of the following lemma. This place is chosen so that the $n$-dimensional residual representation, restricted to the decomposition group at this place, has no ramified deformations, a convenient property that is used in \cite{jackreducible} to simplify the analysis of certain spaces of automorphic forms. Indeed, we can then shrink the level at this place in our spaces of automorphic forms without acquiring contributions from new lifts of our fixed residual representation. This allows us to work with neat or `sufficiently small' level subgroups.  
\begin{lemma}\label{lem_genericity}
There exists a place $v \nmid p$ of $F$ at which $\pi$ is unramified such that if $\alpha_v, \beta_v \in \overline{\bF}_p$ are the eigenvalues of $\overline{r}_{\pi, \iota}(\Frob_v)$ (with multiplicity), then $(\alpha_v / \beta_v)^i \neq q_v^{\pm 1} \text{ mod }p$ for each $i = 0, \dots, {r}$. In particular, $H^2(F_v, \ad (\overline{r}_{\pi, \iota} \otimes \Sym^{r-1} \overline{r}_{\pi, \iota})^{ss}) = 0$.
\end{lemma}
Here $\ad$ denotes the adjoint representation on $\mathfrak{gl}_n$. 
\begin{proof}
The image of the homomorphism $\operatorname{Proj}  \overline{r}_{\pi, \iota} \times \epsilon : G_F \to \PGL_2(\bF_{p^a}) \times \bF_p^\times$ contains $\PSL_2(\bF_{p^a}) \times \{ 1 \}$ and an element of the form $(h, -1)$ (image of complex conjugation), where $h^2 = 1$, $h \neq 1$. Conjugating $\overline{r}_{\pi, \iota}$ by an element of $\GL_2(\bF_{p^a})$, we can assume that $h = \diag(1, -1)$ mod centre. By the Chebotarev density theorem, the first assertion of the lemma will be proved if we can find $g \in \PSL_2(\bF_{p^a})$ such that if $\alpha_v / \beta_v \in \overline{\bF}_p^\times$ is the eigenvalue ratio of $gh$, then $\alpha_v / \beta_v$ has order greater than $2r$. (In particular, we will have $q_v = -1 \text{ mod }p$.) We consider elements of the form $g = \diag(\alpha, \alpha^{-1}) \text{ mod centre}$, $\alpha \in \bF_{p^a}^\times$. The eigenvalue ratio of $gh$ is then $-\alpha^2$, and we're done if we can choose e.g. $\alpha$ to have order greater than $4r$; this is possible provided that $p^a - 1 > 4r$. This inequality holds because we assume that $a > 3$ and $p > \max(4, r)$.

The final sentence follows from Tate duality and existence of an isomorphism 
\[ (\overline{r}_{\pi, \iota} \otimes \Sym^{r-1} \overline{r}_{\pi, \iota})^{ss} \cong \Sym^r \overline{r}_{\pi, \iota} \oplus \det \overline{r}_{\pi, \iota} \otimes \Sym^{r-2} \overline{r}_{\pi, \iota} \]
(showing in particular that the eigenvalues of $(\overline{r}_{\pi, \iota} \otimes \Sym^{r-1} \overline{r}_{\pi, \iota})^{ss}(\Frob_v)$ are among $\alpha_v^{r-i} \beta_v^i$ ($i = 0, \dots, r$)). 
\end{proof}
We fix a choice of place $v_a$ of $F$ satisfying the conclusion of Lemma \ref{lem_genericity}. Let $M / F$ be the extension of $F(\zeta_p)$ cut out by $\overline{r}_{\pi, \iota}$. Let $X_0$ be a finite set of prime-to-$p$ places of $F$ at which $\pi$ is unramified such that for each subfield $M / M' / F$ Galois over $F$ of degree $ > 1$, there is an element of $X_0$ which does not split in $M'$, and such that $v_a \in X_0$. One important consequence of this property is that if $L/F$ is an $X_0$-split extension (i.e.~every place in $X_0$ splits completely in $L$), then $L$ and $M$ are linearly disjoint over $F$.

We now choose suitable RACSDC twists of our RAESDC representations over CM extensions of $F$. Let $K_0 / \bQ$ be an imaginary quadratic field in which $p$ and the residue characteristic of any element of $X_0 \cup \{ v_0 \}$ split. We fix a choice of place $u_0 | p$ of $K_0$.
\begin{lemma}\label{lem_numerology_of_extensions}
We can find totally real extensions $F_2 / F_1 / F$, with $F_2 / F$ and $F_1 / F$ both soluble, and satisfying the following conditions:
\begin{enumerate}
\item $F_1$, $F_2$ are $X_0$-split and the extension $K_1 = K_0 \cdot F_1 / F_1$ is everywhere unramified. $F_1 / F$ has even degree. 
\item For each place $v | p$ of $F_1$, $\overline{r}_{\pi, \iota}|_{G_{F_{1, v}}}$ is trivial, $\zeta_p \in F_{1, v}$, and $[F_{1, v} : \bQ_p] > n(n+1)/2 +1$. For each place $v \nmid p v_0$ of $K_1$, $\mathrm{BC}_{F_1 / F}(\pi)_v$ is unramified. The place $v_0$ splits in $F_1 / F$.
\item For each place $v | v_0$ of $F_2$, $\mathrm{BC}_{F_2 / F}(\pi)_v$ is unramified, $\overline{r}_{\pi, \iota}|_{G_{F_{2, v}}}$ is trivial (in particular, considering the determinant, $q_v = 1 \text{ mod }p$), and if $p^N || (q_v - 1)$ then $p^N > n$. Each place $v | p$ of $F_1$ splits in $F_2$.
\item Let $T_{F_2}$ be the set of places of $F_2$ above $v_0$. If $d_{0, F_2}$ denotes the $\bZ_p$-rank of $\ker(\Delta \to \Delta_0)^{c=-1}$, where $\Delta$ is the Galois group of the maximal abelian pro-$p$ extension of $K_2 = K_0 \cdot F_2$ unramified outside $p$, and $\Delta_0$ is the Galois group of the maximal $T_{F_2}$-split subextension of this extension, then $d_{0, F_2}> 1$.
\item There exist everywhere unramified characters $X, Y : K_1^\times \backslash \bA_{K_1}^\times \to \bC^\times$ of type $A_0$ such that if $x = r_\iota(X)$, $y = r_\iota(Y)$, then $\mathrm{HT}_\tau(x) = \mathrm{HT}_\tau(y) = \{ 0 \}$ if $\tau : K \to \overline{\bQ}_p$ induces the place $u_0$ and moreover $x x^c = \epsilon$, $y y^c = \omega$, and $\overline{x} = \overline{y}$. 
\end{enumerate} 
\end{lemma}
\begin{proof}
We may construct characters $X', Y' : K_0^\times \backslash \bA_{K_0}^\times \to \bC^\times$ of type $A_0$, unramified above the residue characteristic $p_0$ of $v_0$ and the residue characteristic of each element of $X_0$, such that $\mathrm{HT}_{u_0}(x') = \mathrm{HT}_{u_0}(y') = \{ 0 \}$ and $x' (x')^c = \epsilon$, $y' (y')^c = \omega$, and $\overline{x}' = \overline{y}'$, by an application of \cite[Lemma A.2.5]{BLGGT}. We then choose $F_1 / F$ to satisfy the first two points and moreover so that if $X, Y$ denote the composites of $X', Y'$ with $\mathbf{N}_{K_1 / K_0}$, then $X, Y$ are everywhere unramified. We assume moreover that $F_1 / F$ contains an abelian extension of some odd degree $d > 1$ (in which $v_0$ necessarily splits, since it splits in $F_1$).
To satisfy the third point, we can choose $F_2 = F_1 \cdot F_3$ for an $X_0$-split soluble extension $F_3$ such that for each place $v | v_0$ of $F_3$, $F_{3, v}$ is sufficiently large. 

To complete the proof, it remains to show that the fourth point is satisfied. We do this by repeating part of the proof of \cite[Theorem 7.1]{jackreducible}. By class field theory, $d_{0, F_2}$ is equal to the $\bZ_p$-rank of $\left(\overline{\cO_{K_2, T_{F_2}}^\times}\right)^{c=-1}$ (where the overline denotes closure in $\prod_{v | p} \cO_{K_2}^\times(p)$). This is at least the $\bZ_p$-rank of $\left(\overline{\cO_{K_1, T_{F_1}}^\times}\right)^{c=-1}$, where $T_{F_1}$ is the set of places of $F_1$ lying above $v_0$. By \cite[Proposition 19]{Mai02}, and our assumption that $F_1 / F$ contains an abelian extension of degree $d$, this quantity is at least $d > 1$. Therefore $d_{0, F_2} > 1$ and we are done. 
\end{proof}
Let \[\pi_1 = \mathrm{BC}_{K_1 / F}(\Sym^r \sigma) \otimes X Y^{-r} | \cdot |^{(r-1)/2},\] \[\pi_2 = \mathrm{BC}_{K_1 / F}(\Sym^{r-2} \sigma) \otimes X^{-1} Y^{2-r} | \cdot |^{(r-1)/2}.\] Let $s = (r_{\pi, \iota} \otimes \Sym^{r-1} r_{\sigma, \iota})|_{G_{K_1}} \otimes y^{1-r}$.
\begin{lemma}
\begin{enumerate} \item There is an isomorphism $s^c \cong s^\vee \otimes \epsilon^{1-2r}$.
\item $\pi_1, \pi_2$ are cuspidal and conjugate self-dual, $\pi_1 \boxplus \pi_2$ is regular algebraic and $\iota$-ordinary, and there is an isomorphism
\[ \overline{r}_{\pi_1 \boxplus \pi_2, \iota} \cong \overline{s}. \]
\end{enumerate} 
\end{lemma} 
\begin{proof}
The first part is a routine calculation. For the second, we observe that since $\overline{r}_{\pi, \iota} \cong \overline{r}_{\sigma, \iota}$, we have
\begin{multline*}  \overline{s} \cong (\overline{r}_{\sigma, \iota} \otimes \Sym^{r-1} \overline{r}_{\sigma, \iota})|_{G_{K_1}} \otimes \overline{y}^{1-r}
\\ \cong  (\Sym^r \overline{r}_{\sigma, \iota})|_{G_{K_1}} \otimes \overline{y}^{1-r} \oplus (\Sym^{r-2} \overline{r}_{\sigma, \iota})|_{G_{K_1}} \otimes \overline{y}^{1-r} \epsilon^{-1}.
\end{multline*} 
Another routine calculation shows that $\pi_1, \pi_2$ are cuspidal, conjugate self-dual, and that $\pi_1 \boxplus \pi_2$ is regular algebraic. We can apply \cite[Lemma 2.6]{Clo14} to see that $\pi_1 \boxplus \pi_2$ is moreover $\iota$-ordinary. It remains to check the associated residual representation is as claimed. We have 
\[ r_{\pi_1 \boxplus \pi_2, \iota} = r_{\pi_1 | \cdot |^{(1-r)/2},\iota}\oplus r_{\pi_2 | \cdot |^{(-1-r)/2},\iota}, \]
and
\[ r_{\pi_1 | \cdot |^{(1-r)/2},\iota}\cong (\Sym^r r_{\sigma, \iota})|_{G_{K_1}} \otimes x y^{-r}, \]
\[r_{\pi_2 | \cdot |^{(-1-r)/2},\iota} \cong (\Sym^{r-2} r_{\sigma, \iota})|_{G_{K_1}} \otimes x^{-1} y^{2-r} \epsilon^{-1}. \]
The residual representations match up because $\overline{x} = \overline{y}$.
\end{proof}
Let $E / \bQ_p$ be a sufficiently large coefficient field. We can and do conjugate $s$ so that it takes values in $\GL_n(\cO)$ and extends to a homomorphism $s : G_{F_1} \to \cG_n(\cO)$ such that $\nu \circ s = \epsilon^{1-n} \delta_{K_1 / F_1}^{n}$ with $\delta_{K_1 / F_1}$ the quadratic character for the extension $K_1/F_1$ (see, e.g.~, \cite[\S1.1]{BLGGT} for a discussion of multipliers of tensor products of polarized Galois representations). Let $S_{p, F_1}$ the set of $p$-adic places of $F_1$, let $T_{F_1}$ denote the set of places of $F_1$ lying above $v_0$, and let $T'_{F_1}$ denote the set of places lying above $v_a$. Let $S_{F_1} = S_{p, F_1} \cup T_{F_1}$ and choose for each $v \in S_{F_1}$ a place $\wv$ of $K_1$ lying above $v$. Let $\widetilde{S}_{F_1} = \{ \wv \mid v \in S_{F_1} \}$. In the remainder of this section, we will use deformation theory for the group $\cG_n$, as set up in \cite{jackreducible}. In particular, in the following lemma we are considering all liftings (not just those with fixed determinant). 
\begin{lemma}\label{lem_application_of_Breuil_Mezard}
Let $v \in T_{F_1}$ and let $R_{v}^\square$ denote the unrestricted lifting ring of $\overline{s}|_{G_{K_{1, \wv}}}$, and let $R_{v}$ denote the quotient, defined in \cite{shottonGLn}, which corresponds to the inertial type $\tau_{s}$ of $s|_{G_{K_{1, \wv}}}$.
Let $R'_{v}$ denote the quotient corresponding to the inertial type $\tau_{\mathrm{Sp}_n}$ of the special representation (i.e.\ the image of the Steinberg representation under $\rec_{K_{1, \wv}}^T$). Then both $R_{v}$ and $R_{v}'$ are non-zero and there is an isomorphism $(R_{v} / (\varpi))^{red} \cong (R'_{v} / (\varpi))^{red}$ of $R_{v}^\square$-algebras.
\end{lemma}
\begin{proof}
Let $D_v$ be a central division algebra over $K_{1,\wv}$ of rank $n$, and let $\pi_v \in \cO_{D_v}$ be a uniformizer. Then $\cO_{D_v}^\times / (\pi_v) = k_{D_v}$ is a degree $n$ extension of $k(v)$. Let $K'_{1,\wv} / K_{1,\wv}$ be the quadratic unramified extension, let $k(v)'$ be its residue field, and let $\chi_v : k(v)^{\prime, \times} \to \overline{\bQ}_p^\times$ be a character of order $p$ such that $\chi_v \circ \Art_{K_{1,\wv}}^{-1}$ appears in $r_{\pi, \iota}|_{I_{K_{1,\wv}}}$ (thus there are two possible choices which are conjugate under $\Gal(k(v)' / k(v))$). Choose an arbitrary $k(v)$-embedding of $k(v)'$ into $k_{D_v}$ and let $\sigma_v : \cO_{D_v}^\times \to \overline{\bQ}_p^\times$ be the inflation of $\chi_v \circ \mathbf{N}_{k_{D_v} / k(v)'}$. Then $\sigma_v$ is a character of order $p$ and in the notation of \cite[\S 6]{Dot18}, we claim that we have $\operatorname{cyc}_{D_v^\times}(\sigma_v^{-1}) = R_v$, $\operatorname{cyc}_{D_v^\times}(\mathbf{1}) = R'_v$. To verify the claim, we explain a bit more notation. 

Following \cite{shottonGLn}, for us an inertial type is an isomorphism class of continuous representations $\tau:I_{K_{1,\wv}}\to \GL_n(\Qpbar)$ which can be extended to $W_{K_{1,\wv}}$. (This is different to the terminology of \cite{Dot18}, which considers smooth $\tau$. We are therefore incorporating a monodromy operator into our inertial types.) If $\tau$ is an inertial type, we write $\pi_\tau$ for any irreducible generic $\Qpbar$-representation of $\GL_n(K_{1,\wv})$ such that $\rec_{K_{1,\wv}}^T(\pi_\tau)$ has inertial type $\tau$. We need to check that (cf.~the proof of \cite[Theorem 6.1]{Dot18})
$\Hom_{\Qpbar[\GL_n(\cO_{K_{1,\wv}})]}\left(\mathrm{JL}_{\mathbf{K}}(\sigma_v),\pi_\tau \right)$ is non-zero if and only if $\tau = \tau_s$, in which case it is one-dimensional. Similarly, we need to check that $\Hom_{\Qpbar[\GL_n(\cO_{K_{1,\wv}})]}\left(\mathrm{JL}_{\mathbf{K}}(\mathbf{1}),\pi_\tau \right)$ is non-zero if and only if $\tau = \tau_{\mathrm{Sp}_n}$, in which case it is one-dimensional. The Jacquet--Langlands transfer of types (modified to pick out essentially square-integrable representations), $\mathrm{JL}_{\mathbf{K}}$, is defined in \cite[Definition 5.2]{Dot18}. Considering \cite[Example 3.8]{Dot18}, we need to check that if we take an irreducible smooth $D_v^\times$ representation of type $\sigma_v, \mathbf{1}$ respectively, then its Jacquet--Langlands transfer is in the Bernstein component of $\pi_{\tau_s}$, $\pi_{\tau_{\mathrm{Sp}_n}}$ respectively. For the second case, we simply note that the trivial representation of $D_v^\times$ transfers to the Steinberg representation of $\GL_n(K_{1,\wv})$. The first case follows from the main theorem of \cite{bushnell-henniart-zero}, which explicitly describes the image of a level zero representation of $D_v^\times$ under the composition of local Langlands and Jacquet--Langlands. 

We can now invoke \cite[Theorem 6.1]{Dot18} to conclude that the elements of $Z^{n^2}(R_v^\square / (\varpi))$ associated to $R_v$ and $R_v'$ are equal, since $\sigma_v$ and $\mathbf{1}$ have the same reduction to $\Fpbar$. What we need to explain is why this implies that $(R_{v} / (\varpi))^{red} \cong (R'_{v} / (\varpi))^{red}$. Both $R_v / (\varpi)$ and $R_v' / (\varpi)$ are equidimensional of dimension $n^2$ (as $R_v$ and $R'_v$ are equidimensional of dimension $1 + n^2$ and $\cO$-flat, by definition). They are quotients of $R_v^\square / (\varpi)$, which is again equidimensional of dimension $n^2$ (cf. \cite[Theorem 2.5]{shottonGLn}). We see that $(R_{v} / (\varpi))^{red} \cong (R'_{v} / (\varpi))^{red}$ if and only if $\Spec R_v / (\varpi)$ and $\Spec R_v' / (\varpi)$ are equal as closed subspaces of $\Spec R_v^\square / (\varpi)$, and that this is the case if and only if they have the same generic points. This is indeed a consequence of the equality of the images of $R_v$ and $R_v'$ in $Z^{n^2}(R_v^\square / (\varpi))$.
\end{proof}
The following Lemma will be used to verify some technical conditions appearing in an automorphy lifting theorem. 
\begin{lemma}\label{lem_conditions_for_ALT}
Let $F' / F$ be an $X_0$-split totally real extension and let $K' = K_0 F'$. Then:
\begin{enumerate} \item $K' \not\subset F'(\zeta_p)$.
\item There is an isomorphism $\overline{s}|_{G_{K'}} \cong \overline{\rho}_1 \oplus \overline{\rho}_2$, where $\overline{\rho}_i^c \cong \overline{\rho}_i^{\vee} \otimes \epsilon^{1-n}$ ($i = 1, 2$) and $\overline{\rho}_1|_{G_{K'(\zeta_p)}} \not\cong \overline{\rho}_2|_{G_{K'(\zeta_p)}}$.
\item $\overline{s}|_{G_{K'}}$ is primitive (i.e. not isomorphic to a representation induced from a proper closed subgroup of $G_{K'}$).
\item $\overline{s}(G_{K'})$ has no quotient of order $p$.
\end{enumerate}
\end{lemma}
\begin{proof}
The first part holds because $K'$ is $X_0$-split whilst the proper extensions of $F'$ in $F'(\zeta_p)$ are not. The second part holds by construction; note that $\overline{s}(G_{K'}) = \overline{s}(G_K)$, $\overline{s}(G_{K'(\zeta_p)}) = \overline{s}(G_{K(\zeta_p)})$, and that $\overline{s}|_{G_K}$ has two irreducible constituents of distinct dimension which remain irreducible on restriction to $G_{K(\zeta_p)}$. 

For the third part, suppose that there is an isomorphism $\overline{s}|_{G_{K'}} \cong \Ind_{G_L}^{G_{K'}} \overline{t}$ for some proper finite extension $L / K'$ and representation $\overline{t} : G_L \to \GL_{n / [L : K']}(\overline{\bF}_p)$. Then $\overline{r}_{\pi, \iota}(G_L) \subset \overline{r}_{\pi, \iota}(G_{K'})$ is a subgroup of index at most $n < 2p$, so Lemma \ref{lem_tensor_adequacy} shows that $\overline{r}_{\pi, \iota}(G_L)$ contains a conjugate of $\SL_2(\bF_{p^a})$, and hence that the two irreducible constituents of $\overline{s}$ remain irreducible on restriction to $G_L$. By Mackey's formula, $\overline{t}$ is a proper subrepresentation of $\overline{s}|_{G_L}$, so $\overline{t}$ must be irreducible, isomorphic to one of the constituents of $\overline{s}|_{G_L}$, and therefore extend to a representation $\overline{t} : G_{K'} \to \GL_{n / [L : K']}(\overline{\bF}_p)$. Then we have
\[ \overline{s}|_{G_{K'}} \cong \Ind_{G_L}^{G_{K'}} \overline{t} \cong \overline{t} \otimes \Ind_{G_L}^{G_{K'}} \overline{\bF}_p. \] 
The constituents of $\overline{s}$ have dimension $r-1$ and $r+1$, whilst the dimension of $\overline{t}$ is a divisor of $2r$. Taking this into account, we must have $r = 2$ or $3$, $\dim \overline{t}= r-1$ and $[L:K'] = 2r/(r-1) = 3$ or $4$. In particular, the Galois closure $\widetilde{L}$ of $L$ over $K'$ is soluble. This means that $\overline{r}_{\pi, \iota}(G_{\widetilde{L}})$ still contains a conjugate of $\SL_2(\bF_{p^a})$, and the two irreducible constituents of $\overline{s}$ remain irreducible on restriction to $G_{\widetilde{L}}$.  On the other hand, the restriction of $\Ind_{G_L}^{G_{K'}} \overline{\bF}_p$ to $G_{\widetilde{L}}$ is a direct sum of copies of the trivial representation, so we obtain a contradiction.

 The fourth part holds because $\overline{r}_{\pi, \iota}(G_{K'})$ has no quotient of order $p$. 
\end{proof}
If $K' / K_1$ is any CM extension with totally real subfield $F'$, we write $S_{F'}$ (resp. $\widetilde{S}_{F'}$) for the set of places of $F'$  (resp. $K'$) lying above an element of $S_{F_1}$ (resp. $\widetilde{S}_{F_1}$). We define $S_{p, F'}$, $\widetilde{S}_{p, F'}$ etc.\ similarly. If $v \in S_{F'}$ we define $\Lambda_v = \cO \llbracket I_{K'_\wv}^{ab}(p)^n \rrbracket$ and $\Lambda_{F'} = \widehat{\otimes}_{v \in S_{p, F'}} \Lambda_v$. Thus the Krull dimension of $\Lambda_{F'}$ is $n [ F' : \bQ] + 1$. If $v \in S_{p, F'}$, then there are $n$ tautological characters $\psi^v_1, \dots, \psi^v_n : I_{K'_\wv}^{ab}(p)^n \to \Lambda_{F'}^\times$, pushed forward from $\Lambda_v^\times$. These characters play an important role in the definition of `generic prime ideal' later in \S \ref{sec:aut untwisted}.

We consider the global deformation problems (in the sense of \cite[\S 3]{jackreducible}, and where $R_v^\triangle$ is as defined there)
\[ \cS_{F_1} = ( K_1 / F_1, S_{F_1}, \widetilde{S}_{F_1}, \Lambda_{F_1}, \overline{s}, \epsilon^{1-n} \delta_{K_1 / F_1}^n, \{ R_v^\triangle \}_{v \in S_{p, F_1}} \cup \{ R_{v} \}_{v \in T_{F_1}}), \]
\[ \cS'_{F_1} =  ( K_1 / F_1, S_{F_1}, \widetilde{S}_{F_1}, \Lambda_{F_1}, \overline{s}, \epsilon^{1-n} \delta_{K_1 / F_1}^n, \{ R_v^\triangle \}_{v \in S_{p, F_1}} \cup \{ R'_{v} \}_{v \in T_{F_1}}), \]
and
\[ \cS'_{F_2} = (  K_2 / F_2, S_{F_2}, \widetilde{S}_{F_2}, \Lambda_{F_2}, \overline{s}|_{G_{F_2}}, \epsilon^{1-n} \delta_{K_2 / F_2}^n, \{ R_v^\triangle \}_{v \in S_{p, F_2}} \cup \{ R^{St}_{v} \}_{v \in T_{F_2}}), \]
We write $R_{\cS_{F_1}}$, $R_{\cS'_{F_1}},$ etc.\  for the representing objects (the existence is \cite[Proposition 3.8]{jackreducible}).
\begin{lemma}\phantomsection\label{lem_isomorphism_of_residue_rings}
\begin{enumerate} 
\item  There is an isomorphism $(R_{\cS_{F_1}} / (\varpi))^{red} \cong (R_{\cS'_{F_1}} / (\varpi))^{red}$.
\item There is a finite morphism $R_{\cS'_{F_2}} \to R_{\cS'_{F_1}}$.
\end{enumerate}
\end{lemma}
\begin{proof}
The first part is an immediate consequence of Lemma \ref{lem_application_of_Breuil_Mezard}. For the second part, we first note that for any finite place $v\nmid p$ of $F_1$, inertial type $\tau:I_{K_{1,\wv}} \to \GL_n(\Qpbar)$ and place $w|v$ of $F_2$, the natural map $R_{w}^{\square} \to R_{\wv}(\tau)$ factors through $R_{\tilde{w}}(\tau|_{I_{K_{2,\tilde{w}}}})$, where $R_{\wv}(\tau)$, $R_{\tilde{w}}(\tau|_{I_{K_{2,\tilde{w}}}})$ are the respective fixed type lifting rings  of \cite[Definition 3.5]{shottonGLn}. Indeed, it follows from the definitions that there is a Zariski dense set of points in $\Spec(R_{\wv}(\tau))$ which map to the closed subset $\Spec(R_{\tilde{w}}(\tau|_{I_{K_{2,\tilde{w}}}})) \subset \Spec(R_w^\square)$. Since $R_{\wv}(\tau)$ is reduced, we get the desired factorization. Using \cite[Corollary 3.12]{jackreducible} to handle the $p$-adic places, this gives us a natural map $R_{\cS'_{F_2}} \to R_{\cS'_{F_1}}$, provided we can check that if $v \in T_{F_2}$, then the ring $R_v^{St}$ defined in \cite[\S 3.3.4]{jackreducible} coincides with the ring $R_\wv(\operatorname{Sp}_{2r})$ associated in \cite{shottonGLn} to the inertial type of the special representation. It follows from the definitions that there is a natural surjection $R_v^{St} \to R_\wv(\operatorname{Sp}_{2r})$. Since $R_v^{St}$ is a domain, by \cite[Proposition 3.17]{jackreducible}, and these two rings have the same Krull dimension, it must in fact be an isomorphism. 

The finiteness of the map $R_{\cS'_{F_2}} \to R_{\cS'_{F_1}}$ follows from \cite[Lemma 5.3]{New21a} and \cite[Proposition 3.29(2)]{jackreducible}. 
\end{proof}

Now we can apply the level raising results of \cite{Tho22a} to find an automorphic Galois representation, locally of Steinberg type, lifting $\overline{s}$ and establish a finiteness result for $R_{\cS'_{F_2}}$.
\begin{proposition}\label{prop_steinberg_lift_over_F_4}
$R_{\cS'_{F_2}}$ is a finite $\Lambda_{F_2}$-algebra and there exists a RACSDC automorphic representation $\Pi$ of $\GL_n(\bA_{K_2})$ satisfying the following conditions:
\begin{enumerate}
\item There is an isomorphism $\overline{r}_{\Pi, \iota} \cong \overline{s}|_{G_{K_2}}$.
\item $\Pi$ is $\iota$-ordinary.
\item For each $v \in T_{F_2}$, $\Pi_{\wv}$ is an unramified twist of the Steinberg representation.
\item $\Pi$ is unramified away from $S_{F_2}$.
\end{enumerate}
\end{proposition}
\begin{proof}
We choose a place $v_2$ of $F_2$
(and place $\wv_2$ of $K_2$ lying above it) with the following properties:
\begin{itemize} 
\item The residue characteristic of $v_2$ is prime to that of each element of $X_0 \cup \{ v_0 \}$.
\item $v_2$ is split in $K_0$.
\item $q_{v_2} \equiv 1 \text{ mod }p$, and the representations $\overline{s}|_{G_{K_{2, v_2}}}$ and $\overline{r}_{\pi, \iota}|_{G_{F_{2, v_2}}}$ are both trivial. 
\end{itemize}
By \cite[Theorem 6.11]{Tho22a}, we can find an $\iota$-ordinary RACSDC automorphic representation $\Pi_2$ of $\GL_n(\bA_{K_2})$ satisfying the following conditions:
\begin{itemize}
\item There is an isomorphism $\overline{r}_{\Pi_2, \iota} \cong \overline{s}|_{G_{K_2}}$.
\item $\Pi_{2, \wv_2}$ is an unramified twist of the Steinberg representation.
\end{itemize}
(We indicate why the application of  \cite[Theorem 6.11]{Tho22a} is valid. The genericity hypothesis holds by Lemma \ref{lem_genericity}. The existence of $\iota$-ordinary, locally Steinberg automorphic lifts of the irreducible constituents of $\overline{s}|_{G_{K_2}}$ follows by applying a level-raising theorem on $\GL_2$ and then taking symmetric powers. Note that we have introduced the new place $v_2$, rather than using a place over $v_0$, because  \emph{loc.~cit.}~requires us to be unramified at other places with the same residue characteristic as $v_2$.) 

Consider the auxiliary deformation problem
\begin{multline*} \cS''_{F_2} = (  K_2 / F_2, S_{F_2} \cup \{ v_2 \}, \widetilde{S}_{F_2} \cup \{ \wv_2 \}, \Lambda_{F_2}, \overline{s}|_{G_{F_2}}, \epsilon^{1-n} \delta_{K_2 / F_2}^n, \\ \{ R_v^\triangle \}_{v \in S_{p, F_2}} \cup \{ R^{St}_{v} \}_{v \in T_{F_2} \cup \{ v_2 \}}).  \end{multline*}
By \cite[Theorem 6.2]{All20}, modified as in \cite[Theorem 7.5]{Tho22a} to weaken the condition $K_2(\zeta_p) \not\subset \overline{K}_2^{\ker \ad \overline{s}}$, $R_{\cS''_{F_2}}$ is a finite $\Lambda_{F_2}$-algebra. By the Khare--Wintenberger method (i.e.\ a very similar argument to the one given in the proof of Lemma \ref{lem_hida_theory}), and by \cite[Theorem 7.5]{Tho22a}, we can find an $\iota$-ordinary RACSDC automorphic representation $\Pi_2'$ of $\GL_n(\bA_{K_2})$ satisfying the following conditions:
\begin{itemize}
\item There is an isomorphism $\overline{r}_{\Pi'_2, \iota} \cong \overline{s}|_{G_{K_2}}$.
\item $\Pi_2'$ is unramified away from $S_{F_2} \cup \{ v_2 \}$. For each $v \in T_{F_2} \cup \{ v_2 \}$, $\Pi'_{2, \wv}$ is an unramified twist of the Steinberg representation.
\end{itemize}
(The main step in applying the Khare--Wintenberger method, as in the proof of Lemma \ref{lem_hida_theory}, is to check that a suitable fixed weight ordinary lifting ring is non-empty. We can easily check here that e.g. the crystalline weight 0 ordinary lifting ring is non-empty at $v \in S_{p,F_2}$, by writing down lifts by hand. Since $\zeta_p \in F_{2,v}$ and $\overline{s}|_{G_{K_{2,\wv}}}$ is scalar for $v|p$, a twist of $\oplus_{i=0}^{n-1}\epsilon^{-i}$ gives a lift and shows that the ordinary lifting ring is non-zero.) By  \cite[Theorem 6.2]{All20} again (applied to the deformation problem obtained from $\cS_{F_2}''$ by replacing the local condition at $v_2$ with $R_{v_2}^{\square}$), $R_{\cS'_{F_2}}$ is a finite $\Lambda_{F_2}$-algebra. We can now apply the Khare--Wintenberger method once more to obtain the desired representation $\Pi$ of $\GL_n(\bA_{K_2})$. 
\end{proof}
\begin{proposition}\label{prop_graph}
Consider the graph whose vertices are minimal primes of $R_{\cS'_{F_2}} / (\varpi)$, and where two such vertices $\overline{Q}, \overline{Q}'$ are joined by an edge if there is a prime ideal $\p \subset R_{\cS'_{F_2}} / (\varpi)$ satisfying the following conditions:
\begin{enumerate} \item $(\overline{Q}, \overline{Q}') \subset \p$. 
\item $\p$ is of dimension 1.
\item $\p$ is generic, in the sense of \cite[Definition 3.7]{All20} (we recall the definition below).
\end{enumerate} 
Then this graph is connected. 
\end{proposition}
For the convenience of the reader, we recall here what it means for a prime ideal $\p \subset R_{\cS'_{F_2}}$ (or rather, the specialization $r_\p : G_{F_2, S_{F_2}} \to \cG_n(R_{\cS'_{F_2}} / \p)$ of the universal deformation to $R_{\cS'_{F_2}} / \p$) to be generic, in the sense of \cite[Definition 3.7]{All20}. First, the representation $r_\p|_{G_{K_2}} \otimes \operatorname{Frac} R_{\cS'_{F_2}} / \p$ must be absolutely irreducible. Second, for each $v \in S_{p, F_2}$, the universal characters $\psi_1^v, \dots, \psi_n^v : I_{K_{2, \wv}}^{ab}(p) \to \Lambda_{K_{2, \wv}}^\times$ must remain distinct after pushforward to $(R_{\cS'_{F_2}} / \p)^\times$. Third, there must exist $v \in S_{p, F_2}$ and $\sigma \in I_{K_{2, \wv}}^{ab}(p)$ such that the elements $\psi_i^v(\sigma) \text{ mod } \p \in (R_{\cS'_{F_2}} / \p)^\times$ satisfy no non-trivial $\bZ$-linear relation.
\begin{proof}
We have proved that $R_{\cS'_{F_2}}$ is a finite $\Lambda_{F_2}$-algebra. According to \cite[Lemma 3.9]{All20}, any ideal $I \subset R_{\cS'_{F_2}} / (\varpi)$ satisfying
\[ \dim R_{\cS'_{F_2}} / (\varpi, I) > \max(n [F_2 : \bQ] - d_{0, F_2}, \{ n [F_2 : \bQ] - [F_{2, v} : \bQ_p] \}_{v \in S_{p, F_2}}) \]
is contained in a prime $\p$ of dimension 1 and characteristic $p$ which is generic in the given sense. (Recall the integer $d_{0, F_2}$ was defined in the statement of Lemma \ref{lem_numerology_of_extensions}.) Let $c$ denote the connectedness dimension of $R_{\cS'_{F_2}} / (\varpi)$ (defined as in e.g.\ \cite[Definition 1.7]{jackreducible}), and suppose that we have
\[ c > \max(n [F_2: \bQ] - d_{0, F_2}, \{ n [F_2 : \bQ] - [F_{2, v} : \bQ_p] \}_{v \in S_{p, F_2}}) . \]
We claim that the proposition then follows. Indeed, we suppose for contradiction that the graph is disconnected. By definition of connectedness dimension, we can find $\overline{Q}, \overline{Q}'$ in different connected components such that $\dim R_{\cS'_{F_2}} / (\varpi, \overline{Q}, \overline{Q}') \geq c$, hence
\begin{multline*} \dim R_{\cS'_{F_2}} / (\varpi, \overline{Q}, \overline{Q}') > \max(n [F_2 : \bQ] - d_{0, F_2}, \{ n [F_2 : \bQ] - [F_{2, v} : \bQ_p] \}_{v \in S_{p, F_2}}).
\end{multline*}
This inequality implies (taking the ideal $I$ at the beginning of the proof to be $I = (\overline{Q}, \overline{Q}')$) the existence of a prime $\p$ witnessing the an edge between $\overline{Q}$ and $\overline{Q}'$ in the graph, leading to a contradiction. To finish the proof, we therefore just need to explain why $c$ satisfies the claimed inequality. However, \cite[Lemma 3.21]{jackreducible} (or rather its proof -- since there we consider $c(R_{\cS})$, whereas here we consider $c(R_{\cS'_{F_2}} / (\varpi))$) show that $c \geq n[F_2 : \bQ] - 1$, so we just need $d_{0, F_2} > 1$ and $[F_{2, v} : \bQ_p] > 1$. These inequalities hold by construction (see the statement of Lemma \ref{lem_numerology_of_extensions}, noting that if $v | p$ then $[F_{2, v} : \bQ_p] = [F_{1, v} : \Q_p] > n(n+1)/2 + 1 >1$).
\end{proof}
\begin{proposition}\label{prop:generic SF1}
$R_{\cS_{F_1}}$ is a finite $\Lambda_{F_1}$-algebra. Each of its minimal primes has dimension $\dim \Lambda_{F_1}$ and is of characteristic 0. For each such minimal prime $Q \subset R_{\cS_{F_1}}$, we can find a prime $\p \supset Q$ of dimension 1 and characteristic $p$ which is generic, in the sense of \cite[Definition 3.7]{All20}.
\end{proposition}
\begin{proof}
Proposition \ref{prop_steinberg_lift_over_F_4} and Lemma  \ref{lem_isomorphism_of_residue_rings} together imply that $R_{\cS'_{F_1}} / (\varpi)$ is a finite $\Lambda_{F_1} / (\varpi)$-algebra; it follows from this and Lemma \ref{lem_isomorphism_of_residue_rings} that $R_{\cS_{F_1}}$ is a finite $\Lambda_{F_1}$-algebra. The presentation given in \cite[Proposition 3.9]{jackreducible} then implies that $R_{\cS_{F_1}}$ is equidimensional of dimension $\dim \Lambda_{F_1}$. If $Q \subset R_{\cS_{F_1}}$ is a minimal prime then $\dim R_{\cS_{F_1}} / (Q, \varpi) = n[F_1 : \bQ]$ so we can find $\p$ by arguing as in the proof of Proposition \ref{prop_graph} (and using the trivial estimate $d_{0, F_1} > 0$).
\end{proof} 
Fix $\cP$, a finite set of prime ideals of $R_{\cS'_{F_2}}$ of dimension 1 and characteristic $p$ which are generic, in the sense of \cite[Definition 3.7]{All20}, satisfying the following conditions:
\begin{itemize}
\item For each edge $\overline{Q} - \overline{Q}'$ of the graph described in Proposition \ref{prop_graph}, there is $\p \in \cP$ such that $(\overline{Q}, \overline{Q}') \subset \p$. 
\item For each minimal prime $Q \subset R_{\cS_{F_1}}$, there is a generic prime $\p \supset Q$ of $R_{\cS_{F_1}}$ such that $\p_{F_2} = \ker(R_{\cS'_{F_2}} \to R_{\cS'_{F_1}} / \p) \in \cP$.
\end{itemize}
It follows from Proposition \ref{prop:generic SF1} that such a set exists (noting for the second bullet point that if $\p \subset R_{\cS'_{F_1}}$ is generic, then its pullback to $R_{\cS'_{F_2}}$ is also generic: the properties locally at $p$ are preserved, since each place $v | p$ of $F_1$ splits in $F_2$, and the absolute irreducibility is preserved, by \cite[Proposition 5.3]{jackreducible}).
\begin{proposition}
We can find a totally real extension $F_3 / F_2$ satisfying the following conditions:
\begin{enumerate}
\item $F_3 / F$ is soluble and $X_0$-split.
\item Each place $v | p$ of $F_2$ splits in $F_3$.
\item Let $K_3 = K_0 \cdot F_3$. Then for each $\p \in \cP$, and each $v \in T_{F_3}$, $r_\p|_{G_{K_{3, \wv}}}$ is a scalar representation (i.e.~it is unramified with scalar image on Frobenius). 
\end{enumerate} 
\end{proposition}
\begin{proof}
We can construct $F_3$ by adjoining to $F_2$ an $X_0$-split soluble extension of $F$, sufficiently tamely ramified at $v_0$ (see the proof of \cite[Proposition 6.2]{jackreducible}).
\end{proof}
We consider the deformation problem
\[ \cS_{F_3} = (  K_3 / F_3, S_{F_3}, \widetilde{S}_{F_3}, \Lambda_{F_3}, \overline{s}|_{G_{F_3}}, \epsilon^{1-n} \delta_{K_3 / F_3}^n, \{ R_v^\triangle \}_{v \in S_{p, F_3}} \cup \{ R^1_{v} \}_{v \in T_{F_3}}),  \]
where $R_v^1$ is the ``unipotently ramified'' lifting ring defined in \cite[\S 3.3]{jackreducible}.
We introduce two (big, ordinary) Hecke algebras $\bT_{F_2}$ and $\bT_{F_3}$, following \cite[\S\S 4.1--4.2]{All20}. They will be finite faithful algebras over $\Lambda_{F_2}$ and $\Lambda_{F_3}$ respectively. The Hecke algebras $\bT_\chi^T(U(\mathfrak{l}^\infty), \cO)_\m$ defined there depend on a choice of CM number field, coefficient field, sets $T = S_p \sqcup S(B) \sqcup R \sqcup S_a$ of places (the set $S(B)$ in particular being used to define a definite unitary group), and a maximal ideal $\mathfrak{m}$ of the abstract Hecke algebra, that appears in the support of a certain space of ordinary automorphic forms.

To define $\bT_{F_2}$, we take the CM number field $K_2$, coefficient field $E$, $S(B) = T_{F_2}$, $R = \emptyset$, $S_a = T'_{F_2}$, and $\m$ to be the maximal ideal associated to the automorphic representation $\Pi$ constructed in Proposition \ref{prop_steinberg_lift_over_F_4} (which descends to the associated definite unitary group by \cite[Proposition 3.3.2]{cht}). (Note that $S(B)$ has even cardinality because $v_0$ splits in $F_1 / F$ and $[F_1 : F]$ is even.) Let $P_{\cS'_{F_2}} \subset R_{\cS'_{F_2}}$ be the closed $\Lambda_{F_2}$-subalgebra topologically generated by the coefficients of the characteristic polynomials $\det(X - r_{\cS'_{F_2}}(g))$ ($g \in G_{K_2}$) of a representative of the universal deformation; then $P_{\cS'_{F_2}} \in \cC_{\Lambda_{F_2}}$ (as one sees from the equivalent characterization given in \cite[Definition 3.27]{jackreducible}) and there is a diagram (cf.~\cite[\S 4.2]{All20})
\[ \xymatrix{ R_{\cS'_{F_2}} & \ar[l] P_{\cS'_{F_2}} \ar[r] & \bT_{F_2}}. \]
To define $\bT_{F_3}$, we take the CM number field $K_3$, coefficient field $E$, $S(B) = \emptyset$, $R = T_{F_3}$, $S_a = T'_{F_3}$, and $\m$ to be the maximal ideal associated to $\mathrm{BC}_{K_3 / K_2}(\Pi)$ (which descends to the associated definite unitary group by \cite[Th\'eor\`eme 5.4]{Lab11}). Define $P_{\cS_{F_3}} \subset R_{\cS_{F_3}}$ and also $P_{\cS_{F_1}} \subset R_{\cS_{F_1}}$ as in the previous paragraph. Then we can extend the above diagram to a commutative diagram of $\Lambda_{F_3}$-algebras:
\[ \xymatrix{ R_{\cS'_{F_2}} & \ar[l] P_{\cS'_{F_2}} \ar[r] & \bT_{F_2} \\
R_{\cS_{F_3}}  \ar[u] \ar[d] & \ar[l] P_{\cS_{F_3}} \ar[d] \ar[u]  \ar[r] & \bT_{F_3} \ar[u] \\
R_{\cS_{F_1}} & \ar[l] P_{\cS_{F_1}},}
 \]
 where the left-most vertical arrows arise from the universal property of $R_{\cS_{F_3}}$ (cf.~the proof of Lemma \ref{lem_isomorphism_of_residue_rings})
 and the right-most vertical arrow may be constructed using soluble base change (cf.~\cite[Proposition 4.18]{jackreducible}, \cite[Lemma 6.7]{New21a}). Let  $J_{\cS'_{F_2}} = \ker(P_{\cS'_{F_2}} \to \bT_{F_2})$ and $J_{\cS_{F_3}} = \ker(P_{\cS_{F_3}} \to \bT_{F_3})$. The commutativity of the diagram implies that $J_{\cS_{F_3}} P_{\cS'_{F_2}} \subset J_{\cS'_{F_2}}$. We will use the following result (a disguised `$R_\p = \bT_\p$' theorem):
 \begin{proposition}\label{prop_R_equals_T}
 Let $\p \subset R_{\cS_{F_3}}$ be a prime of dimension 1 and characteristic $p$ which is generic, in the sense of \cite[Definition 3.7]{All20}. Suppose $\p \supset J_{\cS_{F_3}}$ and that for each $v \in T_{F_3}$, $r_\p|_{G_{K_{3, \wv}}}$ is a scalar representation. Then for each prime ideal $Q \subset \p$, $Q \supset J_{\cS_{F_3}}$.
 \end{proposition}
 \begin{proof}
 In the case that for each $v \in T_{F_3}$, $r_\p|_{G_{K_{3, \wv}}}$ is the trivial representation, this is \cite[Theorem 4.1]{All20} (which in turn is deduced from the results of \cite{jackreducible}). We need to explain why this condition can be weakened. The stronger hypothesis is used only to be able to invoke \cite[Lemma 3.40]{jackreducible}, which essentially asserts that the irreducible components of the localization and completion of $R_{\cS_{F_3}}^{loc}$ at the pullback of $\p$ are in bijection with those of $\Lambda_{F_3}$. The same twisting argument which is used in \emph{loc.~cit.} to deal with the possibility that the representations $r_\p|_{G_{F_\wv}}$ ($v \in S(B)$) are scalar but non-trivial can be used at the places $v \in R := T_{F_3}$ to get the desired conclusion here. 
 \end{proof}
\begin{proposition}\label{prop_propagation}
Let $\overline{Q} \subset R_{\cS'_{F_2}} / (\varpi)$ be a prime ideal, and let $\overline{Q}_{F_3}$ denote its pullback to $R_{\cS_{F_3}}$. Then $J_{\cS_{F_3}} \subset \overline{Q}_{F_3}$. 
\end{proposition}
\begin{proof}
It suffices to consider minimal such primes $\overline{Q}$. We first show that there is at least one minimal prime $\overline{Q}$ such that $J_{\cS_{F_3}} \subset \overline{Q}_{F_3}$. Since $R_{\cS'_{F_2}}$ is a finite $\Lambda_{F_2}$-algebra, $P_{\cS'_{F_2}}$ is a finite $\Lambda_{F_2}$-algebra. Since $\bT_{F_2}$ is a faithful finite 
$\Lambda_{F_2}$-algebra, $\Spec(\bT_{F_2})$ has an irreducible component of dimension $\dim \Lambda_{F_2}$, so there is a minimal prime $Q \subset P_{\cS'_{F_2}}$ such that $J_{\cS'_{F_2}} \subset Q$. Since $P_{\cS'_{F_2}} \to R_{\cS'_{F_2}}$ is finite, we can find a minimal prime $\overline{Q} \subset R_{\cS'_{F_2}} / (\varpi)$ such that $J_{\cS'_{F_2}} \subset \overline{Q}$.

Now let $\overline{Q}' \subset R_{\cS'_{F_2}} / (\varpi)$ be some other minimal prime. By Proposition \ref{prop_graph}, we can find minimal primes $\overline{Q} = \overline{Q}_1, \overline{Q}_2, \dots, \overline{Q}_N = \overline{Q}' \subset  R_{\cS'_{F_2}} / (\varpi)$ and prime ideals $\p_1, \dots, \p_{N-1} \in \cP$ such that $(\overline{Q}_i, \overline{Q}_{i+1}) \subset \p_i$ for each $i = 1, \dots, N-1$. By construction $J_{\cS_{F_3}} \subset \overline{Q}_{1, F_3}$. Applying Proposition \ref{prop_R_equals_T} to $\p_{1, F_3}$, we find that $J_{\cS_{F_3}} \subset \overline{Q}_{2, F_3}$. Repeating this argument $N-2$ times, we find $J_{\cS_{F_3}} \subset \overline{Q}'_{F_3}$, as required.
\end{proof}
We can now complete the proof of Theorem \ref{thm_untwisted_tensor_product}. Let $f : R_{\cS_{F_1}} \to \cO$ be the homomorphism associated to the lifting $s$, and let $\q = \ker f$. Let $Q \subset \q$ be a minimal prime. By construction, we can find a prime $\p \supset Q$ of dimension 1 and characteristic $p$ such that $\p_{F_2} \in \cP$, hence (by Proposition \ref{prop_propagation}) $J_{\cS_{F_3}} \subset \p_{F_3}$. Applying Proposition \ref{prop_R_equals_T} one more time, we find that $J_{\cS_{F_3}} \subset Q_{F_3}$, hence $J_{\cS_{F_3}} \subset \q_{F_3}$, hence that there is a homomorphism $\bT_{F_3} \to \cO$ associated to $s|_{G_{K_3}}$. It now follows from \cite[Lemma 2.25]{ger} and \cite[Corollaire 5.3]{Lab11} that $s|_{G_{K_3}}$ is automorphic, associated to a RACSDC automorphic representation of $\GL_n(\bA_{K_3})$. Applying soluble descent and untwisting, we find that $r_{\pi, \iota} \otimes \Sym^{r-1}r_{\sigma, \iota}$ is automorphic, as required. \qed

\section{A functoriality lifting theorem for tensor products}\label{sec:flt tp}

Let $F$ be a totally real field, let $p$ be a prime, and let $\iota : \overline{\bQ}_p \to \bC$ be an isomorphism. In \S \ref{sec:aut untwisted} we proved the automorphy of certain untwisted tensor products $r_{\pi, \iota} \otimes \Sym^{r-1} r_{\sigma, \iota}$, where $\pi, \sigma$ are cuspidal, regular algebraic automorphic representations of $\GL_2(\A_F)$ which have \emph{isomorphic} residual representations $\overline{r}_{\pi, \iota} \cong \overline{r}_{\sigma, \iota}$. Our next goal will be to deduce from this the automorphy of the twisted tensor product ${}^{\varphi_p} r_{\pi, \iota} \otimes \Sym^{r-1} r_{\sigma, \iota}$, where $\varphi_p \in G_{\bQ_p}$ is an arithmetic Frobenius lift.  We will accomplish this in \S \ref{sec: aut twisted} using an argument, sketched in the introduction, that exploits the field of definition of $\pi$. 

An indispensable tool in this argument will be Theorem \ref{thm_FLT_for_TP} below, an example of a `functoriality lifting theorem'. The proof of this theorem will occupy \S \ref{sec:flt tp}. Before proceeding to the statement, we give an idea of the argument. For this theorem, we take two cuspidal, regular algebraic automorphic representations $\pi, \pi'$ of $\GL_2(\A_F)$ of weight 0 such that $\overline{r}_{\pi, \iota} \cong \overline{r}_{\pi', \iota}$. We fix an integer $r \geq 1$ and a RAESDC automorphic representation $\sigma$ of $\GL_r(\A_F)$ such that $r_{\sigma, \iota}$ has Hodge--Tate weights $\{0, 2, \dots, 2r-2\}$ (in applications, $\sigma$ will be the degree-$r$ symmetric power of a cuspidal, regular algebraic automorphic representation of $\GL_2(\A_F)$, as in the previous paragraph). We wish to deduce the automorphy of the tensor product of $r_{\pi'. \iota} \otimes r_{\sigma, \iota}$ from that of $r_{\pi, \iota} \otimes r_{\sigma, \iota}$. 

The key insight is that although this is a statement about the automorphy of a Galois representation of degree $2r$, it is enough, in applying the Taylor--Wiles--Kisin method, to be able to control the patched deformation ring $R_\infty$ of the degree 2 residual representation $\overline{r}_{\pi, \iota}$, in combination with the relatively coarse information on the localization and completion of the patched pseudodeformation ring at the point corresponding to $r_{\pi, \iota} \otimes r_{\sigma, \iota}$ given by the vanishing of the adjoint Bloch--Kato Selmer group (which follows from e.g.\ the automorphy of this Galois representation and the main result of \cite{New19a}). In fact, $\Spec R_\infty$ is irreducible when $p$ is odd, as follows from the theory developed in \cite{Kis09}; when $p = 2$, we need to give a slight modification of the argument, following the ideas introduced in \cite{Kha09b}. 

We now give the statement. 
\begin{theorem}\label{thm_FLT_for_TP}
Let $F$ be a totally real number field, let $r \geq 1$ be an integer, let $p$ be a prime number, and let $\iota : \overline{\bQ}_p \to \bC$ be an isomorphism.  Suppose given RAESDC automorphic representations $\pi, \pi'$ of $\GL_2(\bA_F)$ and a RAESDC automorphic representation $\sigma$ of $\GL_r(\bA_F)$, all satisfying the following conditions:
\begin{enumerate}
\item $\pi$ and $\pi'$ are of weight 0 and are non-CM.
\item There is an isomorphism $\overline{r}_{\pi, \iota} \cong \overline{r}_{\pi', \iota}$.
\item $\overline{r}_{\pi, \iota}(G_F)$ contains a conjugate of $\SL_2(\bF_{p^a})$ for some $p^a > 5$.
\item For each place $v | p$ of $F$, $r_{\pi, \iota}|_{G_{F_v}}$ is ordinary (resp. potentially crystalline) if and only if $r_{\pi', \iota}|_{G_{F_v}}$ is ordinary (resp. potentially crystalline). If $p = 2$, they are both potentially crystalline. 
\item For each place $v \nmid p$ of $F$, $\pi_v$ is a character twist of the Steinberg representation if and only if $\pi'_v$ is.
\item\label{assm:H} There is a perfect subgroup $H \leq (\overline{r}_{\pi, \iota} \otimes \overline{r}_{\sigma, \iota})^{ss}(G_F)$ and $h \in H$ with
$2r$ distinct eigenvalues. 
\item For each embedding $\tau : F \to \overline{\bQ}_p$, we have $\mathrm{HT}_\tau(r_{\sigma, \iota}) = \{ 0, 2, \dots, 2r-2 \}$.
\item The tensor product $r_{\pi, \iota} \otimes r_{\sigma, \iota}$ is automorphic and strongly irreducible.
\item The tensor product $r_{\pi', \iota} \otimes r_{\sigma, \iota}$ is irreducible. 
\end{enumerate} 
Then the tensor product $r_{\pi', \iota} \otimes r_{\sigma, \iota}$ is automorphic. 
\end{theorem}
We write $n = 2r$. We remark that we could make do with a weaker condition on the Galois representation $r_{\sigma, \iota}$ than its automorphy. However, it is convenient here and is satisfied in our intended application. Likewise, it is not necessary to assume that the Hodge--Tate weights of $r_{\sigma, \iota}$ are $0, 2, \dots, 2r-2$, but to treat a more general case would require additional notation (note that this implies that the Hodge--Tate weights of $r_{\pi, \iota} \otimes r_{\sigma, \iota}$ are $\{ 0, 1, \dots, 2r-1 \}$, and therefore that $r_{\pi, \iota} \otimes r_{\sigma, \iota}$ will correspond to an automorphic representation of weight 0). The assumption that $H$ is perfect in condition (\ref{assm:H}) ensures that this condition will still be satisfied after any soluble base change.

The proof will follow closely the proof of \cite[Theorem 2.1]{New21b}, with appropriate modifications. We have also made a small simplification to the level structures chosen at Taylor--Wiles places. We begin with some preliminary reductions. After applying \cite[Lemma 4.1.4]{cht} to construct suitable Hecke characters and replacing $F$ by a soluble totally real extension, and $\pi, \pi'$ by their base changes, we can assume that the following additional conditions are satisfied:
\begin{enumerate}
 \setcounter{enumi}{9}
 \item $\pi, \pi'$ are self-dual (in particular, $\det r_{\pi, \iota} = \det r_{\pi', \iota} = \epsilon^{-1}$).
\item $[F : \bQ]$ is even.
\item For each finite place $v$ of $F$, $\pi_v$, $\pi'_v$ and $\sigma_v$ are all Iwahori-spherical. The number of places $v$ such that $\pi_v$ and $\pi'_v$ are ramified is even.
\item Let $S_p$ denote the set of $p$-adic places of $F$, let $\Sigma$ denote the set of places where $\pi$ is ramified, and let $\Sigma_p = \Sigma \cap S_p$, $\Sigma^p = \Sigma - \Sigma_p$, $S = S_p \cup \Sigma$. Then for each $v \in S$, $\overline{r}_{\pi, \iota}|_{G_{F_v}}$ is trivial. Moreover, for each $v \in \Sigma^p$, $q_v \equiv 1 \text{ mod }p$ and $\pi_v$ and $\pi'_v$ are isomorphic to the Steinberg representation (not just a twist of the Steinberg representation). We observe that if $\Sigma_p$ is non-empty, then we must have $p > 2$ (as if $v \in \Sigma_p$ then $\pi_v$ is an unramified twist of the Steinberg representation, and so $r_{\pi, \iota}|_{G_{F_v}}$ is not potentially crystalline, by local-global compatibility, implying that $p > 2$ by assumption (4) of the Theorem).
\item There exists a quadratic CM extension $K / F$ which is $S$-split and everywhere unramified. We fix a CM type $\Phi$ of $K$, as well as a set $\widetilde{S}$ of finite places of $K$ such that for each $v \in S$, there is exactly one place $\wv$ of $K$ such that $\wv$ lies above $v$ and $\wv \in \widetilde{S}$. 
\item Since $\sigma$ is RAESDC, there is given a character $\chi : F^\times \backslash \bA_F^\times \to \bC^\times$ of type $A_0$ such that $\sigma \cong \sigma^\vee \otimes \chi$ and $\chi_v(-1)$ is independent of $v | \infty$. This character necessarily has the property that $\chi |\cdot|^{r-1}$ is of finite order (compare the Hodge--Tate weights of $r_{\sigma,\iota}$ and $r_{\sigma^\vee,\iota} \cong r_{\sigma,\iota}^\vee(1-r)$). We assume that there exists an everywhere unramified character $X : K^\times \backslash \bA_K^\times \to \bC^\times$ of type $A_0$ such that $X \circ \mathbf{N}_{K / F} = \chi \circ \mathbf{N}_{K / F}$, and an everywhere unramified character $Y : K^\times \backslash \bA_K^\times \to \bC^\times$ of type $A_0$ such that $Y \circ \mathbf{N}_{K / F} = | \cdot |$. We assume moreover that if $\tau \in \Phi$ induces a place $v$ of $K$, then $X|_{K_v^\times}(z) = \tau(z)^{1-r}$ and $Y|_{K_v^\times}(z) = \tau(z)$ for each $\tau \in \Phi$.
\item The representations $(r_{\pi, \iota} \otimes r_{\sigma, \iota})|_{G_{K(\zeta_{p^\infty})}}$ and $(r_{\pi', \iota} \otimes r_{\sigma, \iota})|_{G_{K}}$ are both irreducible.
\end{enumerate}
Let $\pi_K = \mathrm{BC}_{K / F}(\pi)$, $\pi'_K = \mathrm{BC}_{K / F}(\pi')$, and $\sigma_K = \mathrm{BC}_{K / F}(\sigma) \otimes X^{-1}$. Then $\pi_K, \pi'_K, \sigma_K$ are all RACSDC automorphic representations. We write $T$ for the RACSDC automorphic representation of $\GL_{n}(\bA_K)$ such that $r_{T, \iota} \cong r_{\pi_K, \iota} \otimes r_{\sigma_K, \iota} \otimes r_{Y, \iota}^{1-r}$. Our assumptions imply that $T$ is of weight 0. 

Let $E / \bQ_p$ be a coefficient field. After possibly enlarging $E$, we can find conjugates $r, r'$ and $s$ of the representations $r_{\pi, \iota}, r_{\pi', \iota}$ and $r_{\sigma, \iota}$ which take values in $\GL_2(\cO)$, $\GL_2(\cO)$ and $\GL_r(\cO)$ (respectively) and assume that the characters $r_{X, \iota}$, $r_{Y, \iota}$ take values in $\cO^\times$. We can also assume that the residual representations $\overline{r} = r \text{ mod }\varpi$ and $\overline{r}' = r' \text{ mod }\varpi$ are equal (not just conjugate) and that the eigenvalues of each element in the images of $\overline{r} = \overline{r}'$ and $\overline{s}$ lie in the residue field $k$. 

Recall from \S \ref{subsec_notation} that we write $\cC_\cO$ for the category of complete Noetherian local $\cO$-algebras with residue field $k$. If $v$ is a place of $F$, we write $R_v^\square \in \cC_\cO$ for the object representing the functor $\operatorname{Lift}_v : \cC_\cO \to \mathrm{Sets}$ which associates to $A \in \cC_\cO$ the set of homomorphisms $\widetilde{r} : G_{F_v} \to \GL_2(A)$ lifting $\overline{r}|_{G_{F_v}}$ and such that $\det \widetilde{r} = \epsilon^{-1}|_{G_{F_v}}$. As in \cite{New21b}, we introduce certain quotients of $R_v^\square$:
\begin{itemize}
\item If $v \in S_p$ and $r_{\pi, \iota}|_{G_{F_v}}$ is non-ordinary, the smallest reduced $\cO$-torsion-free quotient $R_v$ of $R_v^\square$ such that if $f : R_v^\square \to \overline{\bQ}_p$ is a homomorphism such that the pushforward of the universal lifting to $\overline{\bQ}_p$ is crystalline of Hodge--Tate weights $\{ 0, 1 \}$ and is not ordinary, then $f$ factors through $R_v$. By \cite[Corollary 2.3.13]{Kis09b}$, R_v$ is a domain of dimension $4 + [F_v : \bQ_p]$. 
\item If $v \in S_p$ and $r_{\pi, \iota}|_{G_{F_v}}$ is ordinary and crystalline, the smallest reduced $\cO$-torsion-free quotient $R_v$ of $R_v^\square$ such that if $F : R_v^\square \to \overline{\bQ}_p$ is a homomorphism such that the pushforward of the universal lifting to $\overline{\bQ}_p$ is crystalline of Hodge--Tate weights $\{ 0, 1 \}$ and is ordinary, then $F$ factors through $R_v$. By \cite[Proposition 2.4.6]{Kis09b}, $R_v$ is a domain of dimension $4 + [F_v : \bQ_p]$. 
\item If $v \in S_p$ and $r_{\pi, \iota}|_{G_{F_v}}$ is ordinary and non-crystalline (i.e. if $v \in \Sigma_p$), the smallest reduced $\cO$-torsion-free quotient $R_v$ of $R_v^\square$ such that if $F : R_v^\square \to \overline{\bQ}_p$ is a homomorphism such that the pushforward of the universal lifting to $\overline{\bQ}_p$ is semi-stable non-crystalline of Hodge--Tate weights $\{ 0, 1 \}$, then $F$ factors through $R_v$. By \cite[Proposition 4.3.1]{Sno18}, $R_v$ is a domain of dimension $4 + [F_v : \bQ_p]$. (Here it is helpful to again note that if $\Sigma_p$ is non-empty, then $p > 2$.)
\item If $v \in \Sigma^p$, the smallest reduced $\cO$-torsion-free quotient $R_v$ of $R_v^\square$ such that if $f : R_v^\square \to \overline{\bQ}_p$ is a homomorphism such that the pushforward of the universal lifting to $\overline{\bQ}_p$ is an extension of $\epsilon^{-1}$ by the trivial character, then $f$ factors through $R_v$. By \cite[Proposition 2.5.2]{Kis09b}, $R_v$ is a domain of dimension 4.
\item If $v \in S_\infty$ (the set of infinite places of $F$), the quotient $R_v$ of $R_v^\square$ denoted $R_{V_{\mathbf{F}}}^{-1, \square}$ in \cite[Proposition 2.5.6]{Kis09b}. Then $R_v$ is a domain of dimension 3.
\end{itemize}
If $Q$ is a finite set of finite places of $F$, disjoint from $S$, then we write $\operatorname{Def}_Q : \cC_\cO \to \Sets$ for the functor which associates to $A \in \cC_\cO$ the set of $1 + M_2(\mathfrak{m}_A)$-conjugacy classes of lifts $\widetilde{r} : G_F \to \GL_2(A)$ of $\overline{r}$ satisfying the following conditions:
\begin{itemize}
\item $\widetilde{r}$ is unramified outside $S \cup Q$, and $\det \widetilde{r} = \epsilon^{-1}$.
\item For each $v \in S \cup S_\infty$, the homomorphism $R_v^\square \to A$ determined by $\widetilde{r}|_{G_{F_v}}$ factors through the quotient $R_v^\square \to R_v$ introduced above.
\end{itemize} 
Then $\operatorname{Def}_Q$ is represented by an object $R_Q \in \cC_\cO$; we write $R_\emptyset = R$. We also have some ``framed'' variants of these rings. We let $R_{loc} = \widehat{\otimes}_{v \in S \cup S_\infty} R_v$. Then $R_{loc}$ is an $\cO$-flat domain of dimension $1 + 3 [F : \bQ] + 3 |S|$. We write $\operatorname{Def}_Q^\square : \cC_\cO \to \Sets$ for the functor of $1 + M_2(\mathfrak{m}_A)$-conjugacy classes of tuples $(\widetilde{r}, \{A_v \}_{v \in S \cup S_\infty})$, where $\widetilde{r}$ is of type $\operatorname{Def}_Q$ and $A_v \in 1 + M_2(\mathfrak{m}_A)$, and the action of $1 + M_2(\mathfrak{m}_A)$ is as defined on \cite[p.\ 124]{New21b}. The functor $\operatorname{Def}_Q^\square$ is represented by an object $R_Q^\square$ and there is a tautological homomorphism $R_{loc} \to R_Q^\square$.

We also introduce the functor $\operatorname{Def}'_Q : \cC_\cO \to \Sets$ of $1 + M_2(\mathfrak{m}_A)$-conjugacy classes of lifts $\widetilde{r} : G_F \to \GL_2(A)$ of $\overline{r}$ satisfying the following conditions:
\begin{itemize}
\item $\widetilde{r}$ is unramified outside $S \cup Q$, and for each $v \in S$, $\det \widetilde{r}|_{G_{F_v}} = \epsilon^{-1}|_{G_{F_v}}$.
\item For each $v \in S \cup S_\infty$, the homomorphism $R_v^\square \to A$ determined by $\widetilde{r}|_{G_{F_v}}$ factors through the quotient $R_v^\square \to R_v$ introduced above.
\end{itemize}
(Thus we are not fixing the determinant of our deformations on whole Galois group, but only its restriction to decomposition groups at places in $S$.) This functor is represented by an object $R'_Q \in \cC_\cO$, and there are framed variants $\operatorname{Def}^{\prime, \square}_Q$, $R^{\prime, \square}_Q$ and a homomorphism $R_{loc} \to R_Q^{\prime, \square}$. These objects are used in implementing the modification of the Taylor--Wiles method introduced by Khare--Wintenberger \cite{Kha09b} in the case $p = 2$. 

Let $t, t'$ denote the group determinants (in the sense of \cite{chenevier_det}) of $G_K$ over $\cO$ associated to the tensor products $(r \otimes s)|_{G_K}  \otimes r_{X^{-1} Y^{1-r}, \iota}$ and $(r' \otimes s)|_{G_K} \otimes r_{X^{-1} Y^{1-r}, \iota}$, respectively, and let $\overline{t}$ denote their common reduction modulo $(\varpi)$. Let $a < b$ be integers such that the Hodge--Tate weights of $(r \otimes s)|_{G_K}  \otimes r_{X^{-1} Y^{1-r}, \iota}$ all lie in the interval $[a, b]$. We write $P \in \cC_\cO$ for the quotient $R_S$ of $R_{\overline{t}, S}^{[a, b]}$ introduced in \cite[\S 2.4]{New19a} following \cite{WWE_pseudo_def_cond}. Informally, $P$ represents the functor of conjugate self-dual group determinants of $G_{K, S}$ lifting $\overline{t}$ which are semi-stable with Hodge--Tate weights in the interval $[a, b]$. 
\begin{lemma}\label{lem:Kisin vs WWE def rings}
Let $A \in \cC_\cO$ be Artinian, let $v \in S_p$, and let $\widetilde{r}_A : G_{F_v} \to \GL_2(A)$ be a lift of $\overline{r}|_{G_{F_v}}$ of determinant $\epsilon^{-1}|_{G_{F_v}}$ such that the associated homomorphism $R_v^\square \to A$ factors through $R_v$. View $A^{2r} \cong A^2 \otimes_\cO \cO^r$ as an $A[G_{F_v}]$-module via the tensor product representation $\widetilde{r}_A \otimes_\cO s|_{G_{K_\wv}}  \otimes r_{X^{-1} Y^{1-r}, \iota}$. Then $A^{2r}$ is isomorphic, as $\bZ_p[G_{F_v}]$-module, to a subquotient of a lattice in a semistable $\bQ_p[G_{F_v}]$-module with all Hodge--Tate weights in the interval $[a, b]$.
\end{lemma}
\begin{proof}
In the crystalline case (i.e.\ when $v \in S_p - \Sigma_p$), we can employ essentially the same argument as in the proof of \cite[Lemma 2.2]{New21b} -- $A^2$ is isomorphic, as $\bZ_p[G_{F_v}]$-module, to a subquotient of a lattice in a semistable (even crystalline) $\bQ_p[G_{F_v}]$-representation with Hodge--Tate weights in $\{0, 1 \}$. It follows that  $A^2 \otimes_\cO s|_{G_{K_\wv}} \otimes r_{X^{-1} Y^{1-r}, \iota}$ is also isomorphic to a subquotient of a lattice in a crystalline $\bQ_p[G_{F_v}]$-representation with bounded Hodge--Tate weights. 

In the semistable non-crystalline case (i.e.\ when $v \in \Sigma_p$), the same argument applies, provided we can show that $A^2$ is isomorphic, as  $\bZ_p[G_{F_v}]$-module, to a subquotient of a lattice in a semistable $\bQ_p[G_{F_v}]$-representation with Hodge--Tate weights in $\{ 0, 1\}$. An additional argument is required in this case since the local lifting ring $R_v$ is defined in a different way. Let $R_v^\square$ denote the ring representing the functor  of all lifts $\overline{r}|_{G_{F_v}}$ of determinant $\epsilon^{-1}$, and let $R_v'$ denote the quotient of $R_v^\square$ such that a map $R_v^\square \to B$ with $B \in \cC_\cO$ Artinian factors through $R_v'$ if and only if $B^2$ is isomorphic, as $\bZ_p[G_{F_v}]$-module, to a subquotient of a lattice in a semistable $\bQ_p[G_{F_v}]$-representation with Hodge--Tate weights in $\{ 0, 1 \}$ (the existence of $R_v'$ follows from \cite[Theorem 2.3.4]{WWE_pseudo_def_cond}, together with the observation that `torsion semistable of Hodge--Tate weights in $\{0, 1\}$' is a stable condition, cf.~\cite[\S 5.2]{WWE_pseudo_def_cond}). We need to show that the map  $R_v^\square \to R_v$ factors through $R_v'$. The ring $R_v$ is reduced and the maximal ideals of $R_v[1/p]$ are dense in $\Spec R_v$, so it is enough to show that if $R_v^\square \to \overline{\bQ}_p$ factors through $R_v$ then it factors through $R_v'$. It follows from \cite[Proposition 4.3.1]{Sno18} that any such homomorphism corresponds to a lifting $\widetilde{r} : G_{F_v} \to \GL_2(\overline{\bZ}_p)$ of $\overline{r}|_{G_{F_v}}$ such that $\widetilde{r} \otimes_{\overline{\bZ}_p} \overline{\bQ}_p$ is an extension of $\epsilon^{-1}$ by the trivial character. In particular, any such extension is semistable. This completes the proof.
\end{proof}
\begin{cor}
The assignment $\widetilde{r} \mapsto (\widetilde{r} \otimes s)|_{G_K}  \otimes r_{X^{-1} Y^{1-r}, \iota}$ determines a morphism $P \to R$ in $\cC_\cO$. In particular, $t, t'$ determine homomorphisms $P \to \cO$.
\end{cor}
\begin{proof}
This follows from the definitions and Lemma \ref{lem:Kisin vs WWE def rings}. 
\end{proof}
In this section of the paper, we will call a Taylor--Wiles datum of level $N \geq 1$ a tuple 
\[ (Q, \widetilde{Q}, (\alpha_v, \beta_v)_{v \in Q}, (\gamma_{v, 1}, \dots, \gamma_{v, r})_{v \in Q}) \] satisfying the following conditions:
\begin{itemize}
\item $Q$ is a finite set of places of $F$, split in $K$, such that for each $v \in Q$, $q_v \equiv 1 \text{ mod }p^N$.
\item For each $v \in Q$, we have fixed a factorization $v = \wv \wv^c$ in $K$, and $\widetilde{Q} = \{ \wv \mid v \in Q \}$.
\item For each $v \in Q$, $(\overline{r} \otimes \overline{s})(\Frob_\wv)$ has $2r$ distinct eigenvalues, and $\alpha_v, \beta_v \in k$ are the eigenvalues of $\overline{r}(\Frob_\wv)$, while $\gamma_{v, 1}, \dots, \gamma_{v, r} \in k$ are the eigenvalues of $(\overline{s}|_{G_{K}} \otimes \overline{r}_{X^{-1} Y^{1-r}, \iota})(\Frob_\wv)$.
\end{itemize}
Note that this definition differs to the one given in \cite[\S 2]{New21b}, reflecting our interest here in tensor products (as opposed to symmetric powers). If $Q$ is a Taylor--Wiles datum then $R_Q$ has a natural structure of $\cO[\Delta_Q]$-algebra, where $\Delta_Q = \prod_{v \in Q} k(v)^\times(p)$ is the maximal $p$-power quotient of the finite abelian group $\prod_{v \in Q} k(v)^\times$. The natural map $R_Q \to R$ factors through an isomorphism $R_Q \otimes_{\cO[\Delta_Q]} \cO \cong R$. 

We write $P_Q \in \cC_\cO$ for the quotient of the quotient $R_{S \cup Q}$ of $R_{\overline{t}, S \cup Q}^{[a, b]}$ introduced in \cite[\S 2.19]{New19a} corresponding to pseudodeformations $\widetilde{t}$ of $\overline{t}$ satisfying the following additional conditions:
\begin{itemize}
\item For each $v \in Q$, $\widetilde{t}|_{G_{K_\wv}}$ factors through the maximal Hausdorff abelian quotient $G_{K_\wv} \to G_{K_\wv}^{ab}$. As in \cite[pg.~125]{New21b}, there are characters $\chi_{v, 1}, \dots, \chi_{v, n}$ of $W_{K_\wv} $ such that $\widetilde{t}|_{G_{K_\wv}}$ is the pseudocharacter associated to $\chi_{v, 1} \oplus \dots \oplus \chi_{v, n}$, and such that each character $\overline{\chi}_{v, i}$ is unramified, satisfying $\overline{\chi}_{v, i} = \alpha_v \gamma_{v, i}$ if $1 \leq i \leq r$ and $\overline{\chi}_{v, i} = \beta_v \gamma_{v, i-r}$ if $r+1 \leq i \leq n$. (The characters $\chi_{v, 1}, \dots, \chi_{v, n}$ are then uniquely determined.)
\end{itemize} 

The characters $\chi_{v, i}|_{I_{K_\wv}} \circ \Art_{K_\wv}$ ($v \in Q, 1\le i \le n$) collectively endow $P_Q$ with the structure of a $\cO[\Delta'_Q]$-algebra, where $\Delta'_Q = (\Delta_Q)^n$. Tensor product with $s|_{G_K}  \otimes r_{X^{-1} Y^{1-r}, \iota}$ determines a morphism $P_Q \to R_Q$ of $\cO$-algebras.

We now introduce spaces of automorphic forms, beginning with forms on a definite quaternion algebra. We define the following notation, following \cite[p. 126]{New21b}:
\begin{itemize}
\item $D$ is a quaternion algebra over $F$ ramified exactly at $\Sigma \cup S_\infty$.
\item $\cO_D \subset D$ is a fixed choice of maximal order, equipped with, for each finite place $v\not\in \Sigma$, a choice of identification $\cO_D \otimes_{\cO_F} \cO_{F_v} \cong M_2(\cO_{F_v})$.
\item If $U = \prod_v U_v \subset  (D \otimes_F \A_{F}^\infty)^\times$ is an open subgroup, then $H_D(U)$ is the set of functions $f : (D \otimes_F \A_F^\infty)^\times \to \cO$ such that
for all $\gamma \in D^\times$, $z \in (\A_F^\infty)^\times$, $g \in (D \otimes_F \A_F^\infty)^\times$, and $u \in U$, we have $f(\gamma gzu) = f(g)$.
\item $U_0 = \prod_v U_{0, v} = \prod_{v \not\in \Sigma} (\cO_D \otimes_{\cO_F} \cO_{F_v})^\times \times (\prod_{v \in \Sigma} (D \otimes_F F_v)^\times$ is an open subgroup of $(D \otimes_F \A_{F}^\infty)^\times$. 
\item If $(Q, \widetilde{Q}, \{ \alpha_v, \beta_v \}_{v \in Q})$ is a Taylor--Wiles datum of level $N \geq 1$, then the subgroups $U_1(Q; N) = \prod_v U_1(Q; N)_v \subset U_0(Q) = \prod_v U_0(Q)_v \subset U_0$ are defined by $U_1(Q; N)_v = U_0(Q)_v = U_{0, v}$ if $v \not\in Q$, and $U_0(Q)_v = \Iw_v$ and $U_1(Q; N)_v = \{ \begin{psmallmatrix} a & b \\ c & d \end{psmallmatrix} \in \Iw_v \mid ad^{-1} \mapsto 1 \in k(v)^\times(p) / (p^N) \}$ if $v \in Q$. (Here we recall from \S \ref{subsec_notation} that $\Iw_v$ denotes the standard Iwahori subgroup of $\GL_2(\cO_{F_v})$.)
\item If $v \not\in \Sigma \cup Q$, then $\bT^{\univ}_{D, \Sigma \cup Q}$ denotes the polynomial ring over $\cO$ in the indeterminates $T_v^{(1)}, T_v^{(2)}$ ($v \not\in \Sigma \cup Q$), and $\bT^{univ, Q}_{D, \Sigma \cup Q} = \bT^{univ}_{D, \Sigma \cup Q}[\{ U_{\varpi_v}^{(1)} \}_{v \in Q}]$. Then $H_D(U_0(Q))$ and $H_D(U_1(Q; N))$ are $\bT^{univ, Q}_{D, \Sigma \cup Q}$-modules, where each indeterminate $T_v^{(i)}$ and $U_{\varpi_v}^{(1)}$ acts by the Hecke operator of the same name. 
\end{itemize} 
There is a unique maximal ideal $\mathfrak{m}_D \subset \bT^{univ}_{D, \Sigma}$ of residue field $k$ and with the following properties:
\begin{itemize}
\item If $v \not\in S$, then $\det(X - \overline{r}(\Frob_v))$ equals $X^2 - T_v^{(1)} X + q_v T_v^{(2)} \text{ mod }\ffrm_D$.
\item If $v \in S_p - \Sigma_p$ and $r|_{G_{F_v}}$ is non-ordinary, then $T_v^{(1)} \in \ffrm_D$. If $v \in S_p - \Sigma_p$ and $r|_{G_{F_v}}$ is ordinary, then $T_v^{(1)} - 1 \in \ffrm_D$. In either case, $T_v^{(2)} - 1 \in \ffrm_D$. 
\end{itemize}
If $ (Q, \widetilde{Q}, (\alpha_v, \beta_v)_{v \in Q}, (\gamma_{v, 1}, \dots, \gamma_{v, r})_{v \in Q}) $ is a Taylor--Wiles datum then we write $\ffrm_{D, Q} \subset \bT_{D, \Sigma \cup Q}^{univ, Q}$ generated by $\ffrm_D \cap \bT^{univ}_{D, \Sigma \cup Q}$ and the elements $U_{\varpi_v}^{(1)} - \alpha_v$ ($v \in Q$). 

If 
\[ \chi : F^\times \backslash (\bA_F^\infty)^\times / \det U_1(Q; N) \to \cO^\times \]
is a quadratic character and $f \in H_D(U_1(Q; N))$, then we define $f \otimes \chi \in H_D(U_1(Q; N))$ by the formula $(f \otimes \chi)(g) = \chi(\det(g)) f(g)$. We observe that if $p = 2$ and $f \in H_D(U_1(Q; N))_{\ffrm_{D, Q}}$ then $f \otimes \chi \in H_D(U_1(Q; N))_{\ffrm_{D, Q}}$. 

Let $m_D \geq 0$ denote the $p$-adic valuation of the least common multiple of the exponents of the Sylow $p$-subgroups of the finite groups $F^\times \backslash (U(\bA_F^\infty)^\times \cap t^{-1} D^\times t)$, for $t \in (D \otimes_F \bA_F^\infty)^\times$ (cf.~\cite[\S 7.2]{Kha09b}). 
\begin{prop}
Let $N \geq 1$ and let \[ (Q, \widetilde{Q}, (\alpha_v, \beta_v)_{v \in Q}, (\gamma_{v, 1}, \dots, \gamma_{v, r})_{v \in Q}) \]  be a Taylor--Wiles datum of level $N + m_D$. Then:
\begin{enumerate}
\item The maximal ideals $\ffrm_D$ and $\ffrm_{D, Q}$ are in the support of $H_D(U_0)$ and $H_D(U_0(Q))$, respectively.
\item $H_D(U_1(Q; N))_{\ffrm_{D, Q}}$ is a $\bT^{univ}_{D, \Sigma \cup Q} \otimes_\cO \cO[\Delta_Q / (p^N)]$-module free as $\cO[\Delta_Q / (p^N)]$-module, and there is an isomorphism
\[ H_D(U_1(Q; N))_{\ffrm_{D, Q}} \otimes_{\cO[\Delta_Q / (p^N)]} \cO \cong H_D(U_0)_{\ffrm_D} \]
of $\bT^{univ}_{D, \Sigma \cup Q}$-modules. 
\item There exists a structure on $H_D(U_1(Q; N))_{\ffrm_{D, Q}}$ of $R_Q$-module such that for any representative $r_Q^{univ}$ of the universal deformation of $\overline{r}$ and for any finite place $v \not\in S \cup Q$ of $F$, $\tr r_Q^{univ}(\Frob_v)$ acts as $T_v^{(1)}$ and $\det r_Q^{univ}(\Frob_v)$ acts as $q_v T_v^{(2)}$. Moreover, the $\cO[\Delta_Q]$-module structure induced by the map $\cO[\Delta_Q] \to R_Q$ agrees with the one in the second part of the lemma.
\end{enumerate}
\end{prop}
\begin{proof}
This is essentially the same as \cite[Proposition 2.3]{New21b} (and is exactly the same when  $\Sigma_p = \emptyset$ and $r|_{G_{F_v}}$ is non-ordinary for each $v \in S_p$). We explain what is different when there are places $v \in S_p$ such that either $r|_{G_{F_v}}$ is crystalline ordinary or $v \in \Sigma_p$. The two things to check are that $\ffrm_D$ occurs in the support of $H_D(U_0)$ and that any automorphic representation occurring in $H_D(U_0)_{\ffrm_D}$ gives rise to a lifting of $\overline{r}$ which is ``of type $R_v$'' (the analysis of the extra structures determined by the Taylor--Wiles datum is unchanged). If $r|_{G_{F_v}}$ is crystalline ordinary then the trace of Frobenius on $\mathrm{WD}(r|_{G_{F_v}})$ is a $p$-adic unit which is congruent to $1 \text{ mod }\ffrm_{\overline{\bbZ}_p}$, because we assume $\overline{r}|_{G_{F_v}}$ is the trivial representation. Thus the condition $T_v^{(1)} - 1 \in \ffrm_D$ is the right one to ensure $\ffrm_D$ occurs in the support of $H_D(U_0)$.

Now suppose that $\pi_1$ is an automorphic representation of $\GL_2(\bA_F)$ whose Jacquet--Langlands transfer to $D$ contributes to $H_D(U_0)_{\ffrm_D}$. We need to check that if $v \in S_p$ and $r|_{G_{F_v}}$ is crystalline ordinary, then $r_{\pi_1, \iota}|_{G_{F_v}}$ is crystalline ordinary; and that if $v \in \Sigma_p$, then $r_{\pi_1, \iota}|_{G_{F_v}}$ is an extension of $\epsilon^{-1}$ by the trivial character. In the first instance, $\iota^{-1} \pi_{1, v}$ is unramified and the eigenvalue of $T^{(1)}_v$ is a $p$-adic unit. If $v \in \Sigma_p$, then $\pi_{1, v}$ is the Steinberg representation (not just a twist of the Steinberg representation -- recall that by definition, $U_{0, v} = (D \otimes_F F_v)^\times$). In either case, the desired result follows from local-global compatibility. 
\end{proof}
We next describe automorphic forms on a certain definite unitary group $G$. We introduce the following notation, again following \cite{New21b}:
\begin{itemize}
\item $G$ is the reductive group scheme over $\cO_F$ whose functor of points is given by $G(R) = \{ g \in \GL_{n} (R \otimes_{\cO_{F}} \cO_K) \mid g = (1 \otimes c)(g)^{-t} \}$. Thus $G$ is a definite unitary group in $n = 2r$ variables. 
\item If $v$ is a finite place of $F$ which splits $v = w w^c$ in $K$, then $\iota_w : G_{\cO_{F_v}} \to \Res_{\cO_{K_w} / \cO_{F_v}} \GL_{n}$ is the canonical isomorphism (induced by the inclusion $G \subset \Res_{\cO_K / \cO_F} \GL_{n}$ and then projection). 
\item If $V = \prod_v \subset G(\bA_F^\infty)$ is an open compact subgroup, then we write $H_G(V)$ for the set of functions $f : G(F) \backslash G(\bA_F^\infty) / V \to \cO$. 
\item $V_0 = \prod_v V_{0, v} = \prod_{v \in \Sigma} \iota_\wv^{-1}(\Iw_\wv) \times \prod_{v \not\in \Sigma} G(\cO_{F_v})$ is an open compact subgroup of $G(\A_F^\infty)$. 
\end{itemize} 
If $(Q, \widetilde{Q}, (\alpha_v, \beta_v)_{v \in Q}, (\gamma_{v, 1}, \dots, \gamma_{v, r})_{v \in Q})$ is a Taylor--Wiles datum of level $N \geq 1$, then we define $V_1(Q; N) = \prod_v V_1(Q; N)_v \subset V_0(Q) = \prod_v V_0(Q)_v$ by $V_1(Q; N)_v = V_0(Q)_v = V_{0, v}$ if $v \not\in Q$, and $V_0(Q)_v = \iota_\wv^{-1} \Iw_\wv$  and 
\[ V_1(Q; N)_v = \iota_\wv^{-1} \{ (a_{i,j}) \in \Iw_\wv \mid a_{i,i}  \mapsto 1 \in k(\wv)^\times(p) / (p^N) \text{ for }1 \le i \le n\} \]
if $v \in Q$. Thus $V_1(Q; N) \subset V_0(Q)$ is an open normal subgroup with quotient 
\[ V_0(Q) / V_1(Q; N) \cong \Delta'_Q / (p^N). \] Note that the level structure $V_1(Q;N)$ is defined differently than in \cite{New21b}; in fact, it coincides with the level structure denoted $V'_1(Q;N)$ in the proof of \cite[Proposition 2.5]{New21b}.

We write $\bT^{univ}_{G, S \cup Q}$ for the Hecke algebra introduced in \cite{New21b} (polynomial algebra over $\cO$ in the unramified Hecke operators $T_w^{(i)}$ at split places $w$ of $K$ prime to $S$). We define $\bT^{univ, Q}_{G, S \cup Q} = \bT^{univ}_{G, S \cup Q}[\{U_{\varpi_\wv}^{(i)} \}_{v \in Q}^{i = 1, \dots, n}]$. There is a unique maximal ideal $\ffrm_G \subset \bT_{G, S}^{univ}$ of residue field $k$ such that for all finite places $w$ of $K$, split over $F$ and not lying above $S$, $\overline{t}(X - \Frob_v)$ equals $\sum_{i=0}^{2r} (-1)^i q_w^{i(i-1)/2} T_w^{(i)} X^{n-i} \text{ mod }\ffrm_G$. If \[ (Q, \widetilde{Q}, (\alpha_v, \beta_v)_{v \in Q}, (\gamma_{v, 1}, \dots, \gamma_{v, r})_{v \in Q}) \]  is a Taylor--Wiles datum, then we write $\ffrm_{G, Q}$ for the maximal ideal of $\bT^{univ, Q}_{G, S \cup Q}$ generated by $\ffrm_G \cap \bT^{univ}_{G, S \cup Q}$ and the elements
\[ U_{\varpi_\wv}^{(i)}  - q_\wv^{i(1-i)/2} \alpha_v^i \prod_{j=1}^i \gamma_{v, j} \]
(if $1 \leq i \leq r$) and
\[ U_{\varpi_\wv}^{(i)} - q_\wv^{i(1-i)/2} \alpha_v^r \beta_v^{i-r} \prod_{j=1}^{i-r} \gamma_{v, j} \]
(if $r+1 \leq i \leq n$). 

The unitary group $G$ comes with a determinant map $\det : G \to U_1$, where $U_1 = \ker(\mathbf{N}_{K / F} : \Res_{K / F} \bG_m \to \bG_m)$. If
\[ \theta : U_1(F) \backslash U_1(\bA_F^\infty) / \det V_1(Q; N) \to \cO^\times \]
is a character and $f \in H_G(V_1(Q; N))$ then we define $f \otimes \theta = H_G(V_1(Q; N))$ by the formula $(f \otimes \theta)(g) = \theta(\det(g))f(g)$. If $f \in H_G(v_1(Q; N))_{\ffrm_{G, Q}}$ and $\overline{\theta}$ is trivial, then $f \otimes \theta \in H_G(V_1(Q; N))_{\ffrm_{G, Q}}$. 
\begin{lemma}\label{lem_twisting_characters}
Suppose that $p = 2$ and that 
\[ \chi : F^\times \backslash (\bA_F^\infty)^\times / \det U_1(Q; N) \to \cO^\times \]
is a quadratic character. Then there exists a unique quadratic character
\[ \theta_\chi : U_1(F) \backslash U_1(\bA_F^\infty) / \det V_1(Q; N) \to \cO^\times \]
such that for all $z \in (\bA_K^\infty)^\times$, we have $\theta_\chi(z / z^c) = \chi \circ \mathbf{N}_{K / F}(z)$. 
\end{lemma}
\begin{proof}
The argument used to prove \cite[Lemma 2.4]{New21b} shows that there is a unique quadratic character $\theta_\chi : U_1(F) \backslash U_1(\bA_F^\infty) \to \cO^\times$ such that $\theta_\chi(z / z^c) = \chi \circ \mathbf{N}_{K / F}(z)$. We need to check that $\theta_\chi$ is trivial on $\det V_1(Q; N)$. Note that the level structure $V_1(Q; N)_v$ is defined slightly differently here to in \cite{New21b} if $v \in Q$. Unwinding the definitions, we see that we need to check that if $v \in Q$ and $(a_1, \dots, a_n) \in \prod_{i=1}^n \cO_{K_{\wv}}^\times$ satisfies $a_i \mapsto 1 \in  k(v)^\times / (2^N)$, then $\chi( a_1 \dots a_n ) = 1$. This follows from the fact that $\chi$ is trivial on $\det U_1(Q; N)_v$, which contains the pre-image of $1 \in  k(v)^\times / (2^N)$.
\end{proof}
We remark that any quadratic character $\chi : F^\times \backslash (\bA_F^\infty)^\times \to \cO^\times$ that is unramified away from $Q$ is in fact trivial on $\det U_1(Q; N)$ (because if $v \in Q$ then $\det U_1(Q; N)_v \subset (\cO_{F_v}^\times)^2$). 

Let $m_G \geq 0$ denote the $p$-adic valuation of the least common multiple of the exponents of the Sylow $p$-subgroups of the finite groups $G(F) \cap t V_0 t^{-1}$ ($t \in G(\bA_F^\infty)$). The following is our analogue of \cite[Proposition 2.5]{New21b}:
\begin{prop}
Let $N \geq 1$ and let \[ (Q, \widetilde{Q}, (\alpha_v, \beta_v)_{v \in Q}, (\gamma_{v, 1}, \dots, \gamma_{v, r})_{v \in Q}) \]  be a Taylor--Wiles datum of level $N + m_G$. Then:
\begin{enumerate}
\item The maximal ideals $\ffrm_G$, $\ffrm_{G, Q}$ are in the support of $H_G(V_0)$ and $H_G(V_0(Q))$, respectively.
\item $H_G(V_1(Q; N))_{\ffrm_{G, Q}}$ is a $\bT^{univ}_{G, S \cup Q} \otimes_\cO \cO[\Delta'_Q / (p^N)]$-module free as $\cO[\Delta'_Q / (p^N)]$-module, and there is an isomorphism $H_G(V_1(Q; N))_{\ffrm_{G, Q}} \otimes_{\cO[\Delta'_Q / (p^N)]}\cO \cong H_G(V_0)_{\ffrm_G}$ of $\bT^{univ}_{G, S \cup Q}$-modules.
\item There exists a structure on $H_G(V_1(Q; N))_{\ffrm_{G, Q}}$ of $P_Q$-module such that if $\Lambda_i^{univ} : G_K \to P_Q$ ($i = 1, \dots, n$) are the coefficients of the universal characteristic polynomial, then for any finite place $w$ of $K$ which is split over a place $v\not\in S\cup Q$ of $F$, $\Lambda_i^{univ}(\Frob_w)$ acts on $H_G(V_1(Q; N))_{\ffrm_{G, Q}}$ as $q_w^{i(i-1)/2} T_w^{(i)}$. Moreover, the $\cO[\Delta'_Q]$-module structure on $H_G(V_1(Q; N))_{\ffrm_{G, Q}}$ induced by the map $\cO[\Delta'_Q] \to P_Q$ agrees with the one in the second part of the lemma.
\end{enumerate} 
\end{prop}
\begin{proof}
The statement is the same as that of \cite[Proposition 2.5]{New21b}, except that the level structure differs at places $v \in Q$. The exact statement we need here appears as an intermediate result in the proof of \emph{loc.~cit.}.
\end{proof}
\begin{lemma}\label{lem:enormous}
	The subgroup $\left((r \otimes s)|_{G_K}  \otimes r_{X^{-1} Y^{1-r}, \iota}\right)(G_{K(\zeta_{p^\infty})}) \subset \GL_n(\cO)$ is enormous, in the sense of \cite[Definition 2.23]{New19a}.
\end{lemma}
\begin{proof}
Let $G \subset \GL_n$ denote the Zariski closure of the image of $(r \otimes s)|_{G_K} \otimes r_{X^{-1}Y^{1-r}, \iota}$, and let $G^0$ denote its identity component. According to \cite[Lemma 2.28]{New19a}, it suffices to know that the derived subgroup $G^1$ of $G^0$ contains regular semisimple elements and acts absolutely irreducibly. Indeed, writing $\rho = ((r \otimes s)|_{G_K} \otimes r_{X^{-1}Y^{1-r}, \iota})$, we see that $G^1$ is contained in the Zariski closure of $\rho(G_{K(\zeta_{p^\infty})})$. 

The group $G^0$ acts (absolutely) irreducibly because $\rho$ is strongly irreducible. In particular, its centre is contained in $\bG_m$ (i.e. the subgroup of $\GL_n$ consisting of scalar matrices), and we therefore have either $G^0 = G^1$ or $G^0 = G^1 \bG_m$. In either case we see that $G^0$ and $G^1$ have the same invariant subspaces, and so $G^1$ acts irreducibly. 

Take $v \in S_p$ and fix an embedding $\jmath : \overline{\bQ}_p \to \widehat{\overline{K}}_\wv$. According to \cite[Lemma 2.2.5, Lemma 2.2.7]{Pat19},
$\Lie G^0 \otimes_{\overline{\bQ}_p, \jmath}  \widehat{\overline{K}}_\wv$ contains an element $\Theta_{\rho, \jmath}$ (the Sen operator) which has $n$ distinct eigenvalues when considered in $\Lie \GL_n \otimes_{\overline{\bQ}_p, \jmath}  \widehat{\overline{K}}_\wv$ (here we use the fact that $\rho|_{G_{K_\wv}}$ is Hodge--Tate regular). The condition of having $n$ distinct eigenvalues is Zariski open, so we see that $\Lie G^0$ contains elements of $\Lie \GL_n$ with $n$ distinct eigenvalues. This implies that $G^0$ contains elements of $\GL_n$ with $n$ distinct eigenvalues. Indeed, let $T$ be a maximal torus of $G^0$. Every semisimple element of $\Lie G^0$ is $G^0$-conjugate to an element of $\Lie T$, so it's enough to know that if $\Lie T$ contains elements of $\Lie \GL_n$ with $n$ distinct eigenvalues, then $T$ contains elements of $\GL_n$ with $n$ distinct eigenvalues. Conjugating so that $T$ is contained in the diagonal maximal torus $D$, we want to show that $T$ is not contained in the union of $\GL_n$-root subgroups of $D$. Since $T$ is irreducible, it is equivalent to show that is it not contained in any single root subgroup, and this now follows from what we have already observed about $\Lie T$, together with \cite[Ch. II, \S 7.1]{Bor91}. 

Since $G^0 \bG_m = G^1 \bG_m$, we see that $G^0$ contains elements with $n$ distinct eigenvalues if and only if $G^1$ does. This completes the proof. \end{proof}
We are now ready to complete the proof of Theorem \ref{thm_FLT_for_TP}. As in \cite{New21b}, we treat the cases $p > 2$ and $p = 2$ separately. Define $H_G = H_G(V_0)_{\ffrm_G}$ and $H_D = H_D(U_0)_{\ffrm_D}$.
\begin{proof}[Proof of Theorem \ref{thm_FLT_for_TP}, case $p > 2$] First we state a version of \cite[Proposition 2.6]{New21b}. \begin{proposition}\label{prop_construction_of_patched_data_p_odd}
	We can find an integer $q \geq 0$ with the following property: let 
	\[ W_\infty = \cO \llbracket Y_1, \dots, Y_q, Z_1, \dots, Z_{4 |S \cup S_\infty| -1 } \rrbracket \]
	and
	\[ W'_\infty = \cO \llbracket (Y^{(i)}_1, \dots, Y^{(i)}_q)_{1\le i \le n }, Z_1, \dots, Z_{4 |S \cup S_\infty| -1} \rrbracket. \]
	 Then we can find the following data:
	\begin{enumerate}
		\item A complete Noetherian local $W'_\infty$-algebra $P_\infty$ and a complete Noetherian local $W_\infty$-algebra $R_\infty$, equipped with isomorphisms $P_\infty \otimes_{W'_\infty} \cO \cong P$ and $R_\infty \otimes_{W_\infty} \cO \cong R$ in $\cC_\cO$.
		\item A surjection $R_{loc} \llbracket X_1, \dots, X_g \rrbracket \to R_\infty$ in $\cC_\cO$, where $g = q + |S \cup S_\infty| - 1$.
		\item A $P_\infty$-module $H_{G,\infty}$ and an $R_\infty$-module $H_{D,\infty}$, finite free over $W'_\infty$ and $W_\infty$ respectively, together with isomorphisms $H_{G,\infty} \otimes_{W'_\infty} \cO \cong H_G$ (as $P$-module) and $H_{D,\infty} \otimes_{W_\infty} \cO \cong H_D$ (as $R$-module).
		\item A morphism $P_\infty \to R_\infty$ of $\cO$-algebras making the diagram
		\[ \xymatrix{ P_\infty \ar[r] \ar[d] &  R_\infty \ar[d] \\
			P \ar[r] & R } \]
		commute.
	\end{enumerate}
\end{proposition}
The proof of this Proposition is essentially identical to that of \cite[Proposition 2.6]{New21b}, we have simply made a different (more standard) choice for the unitary group level structures at the Taylor--Wiles places, which means that the patched module $H_{G,\infty}$ lives over the bigger power series algebra $W'_\infty$. In \emph{loc.~cit.} we choose level structures so that $H_{G,\infty}$ would live over the same ring $W_\infty$ as $H_{D,\infty}$, but we subsequently noticed that this is unnecessary.

Let $\p' \subset P$ denote the kernel of the morphism $P \to \cO$ associated to $t'$. It is enough to show that $\p'$ is in the support of $H_G$ as $P$-module. Equivalently, we must show that if $\p'_\infty$ is the pre-image of $\p'$ under the morphism $P_\infty \to P$, then $\p'_\infty$ is in the support of $H_{G,\infty}$ as $P_\infty$-module. 

The $P_\infty$-module $H_{G,\infty}$ is a Cohen--Macaulay module, since $H_{G,\infty}$ is a finite free $W'_\infty$-module. It follows that each irreducible component of $\Supp_{P_\infty} H_{G,\infty}$ has dimension $qn + 4|S \cup S_\infty|$. Similarly, we see that each irreducible component of $\Supp_{R_\infty}  H_{D,\infty}$ has dimension $q + 4|S \cup S_\infty|$. Since $R_\infty$ is a quotient of $R_{loc} \llbracket X_1, \dots, X_g \rrbracket $, a domain of Krull dimension $q + 4|S \cup S_\infty|$, we see that $R_{loc} \llbracket X_1, \dots, X_g \rrbracket \to R_\infty$ is an isomorphism, that $R_\infty$ is a domain, and that $H_{D,\infty}$ is a faithful $R_\infty$-module.

Let $\p \subset P$ denote the kernel of the morphism $P \to \cO$ associated to $t$, and let $\p_\infty$ denote the pre-image of $\p$ under the morphism $P_\infty \to P$. Then $\p_\infty \in \Supp_{P_\infty} H_{G,\infty}$, by the automorphy hypothesis, and therefore $\dim(P_{\infty, (\p_\infty)}) \ge qn + 4|S \cup S_\infty|-1$. We claim that the Zariski tangent space to the local ring $P_{\infty, (\p_\infty)}$ has dimension at most $qn + 4|S \cup S_\infty|-1$. Indeed, it suffices to show that the quotient $P_{\infty, (\p_\infty)} / ((Y^{(i)}_1, \dots, Y^{(i)}_q)_{1 \le i \le n}, Z_1, \dots, Z_{4 |S \cup S_\infty|-1}) = P_{(\p)}$ equals its residue field $E$. This follows from the vanishing of the adjoint Bloch--Kato Selmer group of $(r \otimes s)|_{G_K}  \otimes r_{X^{-1} Y^{1-r}, \iota}$ (i.e.\ from \cite[Theorem 4.32]{New19a}, which applies thanks to Lemma \ref{lem:enormous}). We deduce that $P_{\infty, (\p_\infty)}$ is a regular local ring of dimension $qn + 4|S \cup S_\infty|-1$, so there is a unique irreducible component $Z$ of $\Spec P_\infty$ containing the point $\p_\infty$, which has dimension $qn + 4|S \cup S_\infty|$ and is contained in $\Supp_{P_\infty} H_{G,\infty}$.

Since $\Spec R_\infty$ is irreducible and the image of the morphism $\Spec R_\infty \to \Spec P_\infty$ contains $\p_\infty$, we find that the morphism $\Spec R_\infty \to \Spec P_\infty$ factors through $Z$. In particular, $\p'_\infty$ lies in $Z$, hence in $\Supp_{P_\infty} H_{G,\infty}$. This completes the proof of Theorem \ref{thm_FLT_for_TP} for $p > 2$. 
\end{proof}
\begin{proof}[Proof of Theorem \ref{thm_FLT_for_TP}, case $p = 2$] As on \cite[p.\ 136]{New21b}, we follow \cite{Kha09b} in writing $\operatorname{Sp}_A : \cC_\cO \to \operatorname{Sets}$ for the functor represented by an object $A \in \cC_\cO$, and $\widehat{\bG}_m : \cC_\cO \to \operatorname{Groups}$ for the functor $A \mapsto \ker(A^\times \to (A / \ffrm_A)^\times)$. We can now state a version of \cite[Proposition 2.8]{New21b}. 
\begin{proposition}\label{prop_construction_of_patched_data_p_equals_2}
	We can find an integer $q \geq |S \cup S_\infty|$ with the following property: let 
	\[ W_\infty = \cO \llbracket Y_1, \dots, Y_q, Z_1, \dots, Z_{4 |S \cup S_\infty| -1 } \rrbracket, \]
	
	\[ W'_\infty = \cO \llbracket (Y^{(i)}_1, \dots, Y^{(i)}_q)_{1\le i \le n }, Z_1, \dots, Z_{4 |S \cup S_\infty| -1} \rrbracket,\] 
	and let $\gamma = 2 - |S \cup S_\infty| + q$.
	 Then we can find the following data:
	\begin{enumerate}
		\item A complete Noetherian local $W'_\infty$-algebra $P_\infty$ and a complete Noetherian local $W_\infty$-algebra $R_\infty$, equipped with isomorphisms $P_\infty \otimes_{W'_\infty} \cO \cong P$ and $R_\infty \otimes_{W_\infty} \cO \cong R$ in $\cC_\cO$.
		\item A complete Noetherian local $\cO$-algebra $R'_\infty$, together with surjections $R'_\infty \to R_\infty$ and $R_{loc} \llbracket X_1, \dots, X_g \rrbracket \to R'_\infty$ in $\cC_\cO$, where $g = 2q+1$.
		\item A $P_\infty$-module $H_{G,\infty}$ and an $R_\infty$-module $H_{D,\infty}$, finite free over $W'_\infty$ and $W_\infty$ respectively, together with isomorphisms $H_{G,\infty} \otimes_{W'_\infty} \cO \cong H_G$ (as $P$-module) and $H_{D,\infty} \otimes_{W_\infty} \cO \cong H_D$ (as $R$-module).
		\item A morphism $P_\infty \to R_\infty$ of $\cO$-algebras making the diagram
		\[ \xymatrix{ P_\infty \ar[r] \ar[d] &  R_\infty \ar[d] \\
			P \ar[r] & R } \]
		commute.
		\item A free action of $\widehat{\bG}_m^\gamma$ on $\Sp_{R'_\infty}$ and a $\widehat{\bG}_m^\gamma$-equivariant morphism $\delta : \Sp_{R'_\infty} \to \widehat{\bG}_m^\gamma$, where $\widehat{\bG}_m^\gamma$ acts on itself by the square of the standard action. 
	\end{enumerate}
	These objects have the following additional properties:
	\begin{enumerate} 
	\setcounter{enumi}{5}
	\item We have $\delta^{-1}(1) = \operatorname{Sp}_{R_\infty}$. The induced action of $\widehat{\bG}_m^\gamma[2](\cO)$ on $R_\infty$ lifts to $H_{D, \infty}$.
	\item There exists an action of $\widehat{\bG}_m^\gamma[2]$ on $\Sp_{P_\infty}$ such that the morphism $\Sp_{R_\infty} \to \Sp_{P_\infty}$ is $\widehat{\bG}_m^\gamma[2]$-equivariant, and the induced action of $\widehat{\bG}_m^\gamma[2](\cO)$ on $P_\infty$ lifts to $H_{G, \infty}$. 
	\end{enumerate}
\end{proposition}
The proof of this Proposition is very similar to that of \cite[Proposition 2.8]{New21b}. We briefly describe the necessary modifications. First, we have chosen a different (finer) level-structure for $G$ at Taylor--Wiles places (as in the case $p > 2$ treated above). This necessitates the introduction of the ring $W'_\infty$ but does not require any other modifications to the proof. Second, we need to adapt \cite[Lemma 2.9]{New21b} (existence of Taylor--Wiles data), which is itself a strengthening of \cite[Lemma 5.10]{Kha09b}. Looking at the proof of \emph{loc.~cit.}, we see that we need to be able find $g \in G_{K \cdot \widetilde{F}_n}$ such that $(\overline{r} \otimes \overline{s})(g)$ has $2r$ distinct eigenvalues, where $\widetilde{F}_n$ is a certain Kummer extension of $F(\zeta_{2^n})$. The existence of such elements follows from assumption (\ref{assm:H}) of Theorem \ref{thm_FLT_for_TP}. The third thing we need to do is explain how, if
\[ (Q_N, \widetilde{Q}_N, (\alpha_v, \beta_v)_{v \in Q_N}, (\gamma_{v, 1}, \dots, \gamma_{v, r})_{v \in Q_N}) \]
is a Taylor--Wiles datum, the group functor $\check{\Theta}_{Q_N}$ should act on $\operatorname{Sp}_{R'_N}$, and how $\check{\Theta}_{Q_N}[2]$ should act on $P_N$ and $\check{\Theta}_{Q_N}[2](\cO)$ should act on $H_G(V_1(Q; N))_{\ffrm_{G, Q_N}}$. We take the same actions as the one defined on \cite[p.\ 138]{New21b} in the case that $n-1$ is odd (twisting by quadratic characters). 

We now show how to use Proposition \ref{prop_construction_of_patched_data_p_equals_2} to complete the proof of Theorem \ref{thm_FLT_for_TP}. Let $\p, \p' \subset P$ denote the kernels of the morphism $P \to \cO$ associated to $t$, $t'$, respectively, and let $\p_\infty, \p'_\infty$ be the pre-images of $\p, \p'$ under the morphism $P_\infty \to P$. We will show that $\p'_\infty$ is in the support of $H_{G,\infty}$ as $P_\infty$-module. As in the case $p > 2$, we see that  $H_{G,\infty}$ is a Cohen--Macaulay module, that $P_{\infty, (\p_\infty)}$  is a regular local ring of dimension $\dim W'_\infty - 1$, and that there is a unique irreducible component $Z$ of $\Spec P_\infty$ containing the point $\p_\infty$, which has dimension $qn + 4|S \cup S_\infty|$ and is contained in $\Supp_{P_\infty} H_{G,\infty}$.

 An identical argument to the one in the first paragraph of \cite[p.\ 137]{New21b} shows that $\Supp_{R_\infty} H_{D, \infty} = \Spec R_\infty$ and that $\widehat{\bG}_m^\gamma[2](\cO)$ acts transitively on the set of irreducible components of $\Spec R_\infty$. Let $\mathfrak{r}$, $\mathfrak{r}' \in \Spec R$ be the points associated to $r, r'$, respectively, and let $\mathfrak{r}_\infty$, $\mathfrak{r}'_\infty \in \Spec R_\infty$ be their respective images. Then there is a point $\mathfrak{r}''_\infty \in \Spec R_\infty$ which is in the same irreducible component of $\Spec R_\infty$ as $\mathfrak{r}_\infty$ and which is in the $\widehat{\bG}_m^\gamma[2](\cO)$-orbit of $\mathfrak{r}'_\infty$. Let $\mathfrak{p}''_\infty \in \Spec P_\infty$ denote the image of $\mathfrak{r}''_\infty$. By local irreducibility of $\Spec P_\infty$ at $\mathfrak{p}_\infty$, we find that $\mathfrak{p}''_\infty$ and $\mathfrak{p}_\infty$ lie on a common irreducible component of $\Spec P_\infty$, namely $Z$, and therefore that $\mathfrak{p}''_\infty \in \Supp_{P_\infty} H_{G, \infty}$. Since the map $\Spec R_\infty \to \Spec P_\infty$ is $\widehat{\bG}_m^\gamma[2](\cO)$-equivariant and $\Supp_{P_\infty} H_{G, \infty}$ is invariant under the action of this group, we find that $\mathfrak{p}'_\infty \in \Supp_{P_\infty} H_{G, \infty}$, as desired.  This completes the proof of Theorem \ref{thm_FLT_for_TP} in the remaining case $p = 2$.
\end{proof}

\section{Automorphy of a twisted tensor product}\label{sec: aut twisted}

In this section we will study automorphy of the twisted tensor products $r_{{}^\gamma \pi, \iota} \otimes \Sym^{r-1} r_{\sigma, \iota}$, where $F$ is a totally real field, $p \geq 5$ is a prime, $\iota : \overline{\bQ}_p \to \bC$ is an isomorphism, $0 < r < p$ is an integer, $\gamma \in \Aut(\bC)$, and $\pi, \sigma$ are cuspidal, regular algebraic automorphic representations of $\GL_2(\A_F)$ which have the same associated residual representation $\overline{r}_{\pi, \iota} \cong \overline{r}_{\sigma, \iota}$, and which satisfy certain further conditions. 

We first describe the argument, which has already been sketched in the introduction. Taking in hand the results of \S \ref{sec:aut untwisted}, we can assume the automorphy of the tensor product $r_{ \pi, \iota} \otimes \Sym^{r-1} r_{\sigma, \iota}$ (corresponding to the choice $\gamma = 1$). The path to the general case uses the field of definition $K_\pi \subset \bC$ of $\pi$ (which is a number field) and its Galois closure $\widetilde{K}_\pi \subset \bC$. The isomorphism class of the representation ${}^\gamma \pi$ depends only on the image of $\gamma$ in $\Gal(\widetilde{K}_\pi / \bQ)$; we suppose first that this image lies in the inertia subgroup $I_{w / l}$ associated to some prime $l$ and place $w$ of $\widetilde{K}_\pi$, and let $\iota_l : \overline{\bQ}_l \to \bC$ be an isomorphism such that $\iota_l^{-1}$ induces the place $w$ of $\widetilde{K}_\pi$. In this case, a simple computation shows that there is an isomorphism of residual representations $\overline{r}_{{}^\gamma \pi, \iota_l} \cong \overline{r}_{\pi, \iota_l}$, and therefore an isomorphism
\[ \overline{r}_{{}^\gamma \pi, \iota_l} \otimes \Sym^{r-1} \overline{r}_{\sigma, \iota_l} \cong \overline{r}_{ \pi, \iota_l} \otimes \Sym^{r-1} \overline{r}_{\sigma, \iota_l}. \]
We can therefore hope to deduce the automorphy of $r_{{}^\gamma \pi, \iota_l} \otimes \Sym^{r-1} r_{\sigma, \iota_l}$ from that of $r_{ \pi, \iota_l} \otimes \Sym^{r-1} r_{\sigma, \iota_l}$ 
using the functoriality lifting theorem, Theorem \ref{thm_FLT_for_TP}, established in the previous section. This would imply the automorphy of $r_{{}^\gamma \pi, \iota} \otimes \Sym^{r-1} r_{\sigma, \iota}$ because this representation lies in the same compatible system. We remark that our choices will mean that this tensor product is Hodge--Tate regular, essentially self-dual, and irreducible (this last by Lemma \ref{lem_irreducibility_of_tensor_products}). For a general $\gamma \in \Aut(\bC)$, we will write the image of $\gamma$ in $\Gal(\widetilde{K}_\pi / \bQ)$ as a product $\delta_1 \dots \delta_s$ of elements of inertia groups $I_{w_1 / l_1}, \dots, I_{w_s / l_s}$  associated to primes $l_1, \dots, l_s$ (which is possible as $\Gal(\widetilde{K}_\pi / \bQ)$ is generated by its inertia subgroups), and argue by induction on $s$. 

What are the obstacles to putting this into practice? We need to apply Theorem \ref{thm_FLT_for_TP} $s$ times, and we have no control over the residue characteristics $l_1, \dots, l_s$ that may appear. Looking at conditions (3), (6) of Theorem \ref{thm_FLT_for_TP}, we see that this theorem requires the residual representations $\overline{r}_{\pi, \iota_l}$, $\overline{r}_{\sigma, \iota_l}$ to have large image, in some sense, potentially for any prime $l$ that is ramified in the coefficient field $K_\pi$. For the representation $\sigma$, this is not a serious issue: using the powerful automorphy lifting theorems proved in \cite{BLGGT}, we can replace $\sigma$ if needed by another representation $\sigma'$ that does have this property, boosting the size of the image by a suitable choice of local conditions. For the representation $\pi$, the problem is more serious, since if we change $\pi$ by e.g. passing along a level-raising congruence modulo some large prime $t$, then we may well change $K_\pi$ (a number field whose dependence on $\pi$ is rather mysterious) at the same time. To circumvent this issue, we employ a generalization of an argument given in \cite{Die11}, which shows that if $\pi$ satisfies suitable local conditions, then \emph{every} member of its associated compatible system of Galois representations has large residual image. 

We can now describe the structure of \S \ref{sec: aut twisted}. In \S \ref{subsec_preparations}, we first recall the basic facts about the coefficient fields of regular algebraic automorphic representations of $\GL_n(\A_F)$, and write down the needed generalization of \cite[Proposition 6.1]{Die11}. In \S \ref{subsect_aut_twisted_first_version}, we follow the above sketch to prove Theorem \ref{thm_automorphic_tensor_product}, which gives the automorphy of the twisted tensor products $r_{{}^\gamma \pi, \iota} \otimes \Sym^{r-1} r_{\sigma, \iota}$ when $\pi$ has the favourable (but hard to verify!) property that every member of its associated compatible system has large residual image. (In order to give a self-contained statement, we in actual fact phrase the conclusion of the theorem in terms of the residual automorphy of the representation ${}^{\varphi_p} \overline{r}_{ \pi, \iota} \otimes \Sym^{r-1} \overline{r}_{\pi, \iota}$, where $\varphi_p$ is arithmetic Frobenius -- this makes no reference to $\sigma$ and is what is needed in applications.)

Finally, in \S \ref{subsect_aut_twisted_second_version}, we use the results of \S \ref{subsec_preparations} to deduce Theorem \ref{thm_big_image_implies_tensor_product}, which gives the residual automorphy of the representation ${}^{\varphi_p} \overline{r}_{ \pi, \iota} \otimes \Sym^{r-1} \overline{r}_{\pi, \iota}$ under only local conditions on $\pi$. This is the result that we will use in \S \ref{sec: the end} to complete the proofs of our main theorems. 
 
\subsection{Preparations for the proof}\label{subsec_preparations}

Let $F$ be a number field. We first recall the definition of the coefficient field $K_\pi$ of a regular algebraic, cuspidal automorphic representation of $\GL_n(\bA_F)$, and how it interacts with the Galois representations that may be associated to $\pi$. These results were proved by Clozel \cite{Clo90} using the rational structure determined by the Betti cohomology of the arithmetic locally symmetric spaces associated to $\GL_n(\bA_F)$. 

The representation $\pi^\infty$ is an irreducible admissible $\bC[\GL_n(\bA_F^\infty)]$-module. If $\sigma \in \Aut(\bC)$, then ${}^\sigma \pi^\infty = \bC \otimes_{\bC, \sigma} \pi^\infty$
 is again an irreducible admissible $\bC[\GL_n(\bA_F^\infty)]$-module\footnote{Note that in this tensor product, we have $\lambda \otimes v = 1 \otimes \sigma^{-1} (\lambda) v$.}, and if $\tau \in \Aut(\bC)$ then ${}^{\sigma \tau} \pi^\infty = {}^\sigma({}^\tau \pi^\infty)$. We define $K_\pi$ to be the subfield of $\bC$ fixed by the stabilizer in $\Aut(\bC)$ of the isomorphism class of $\pi^\infty$. The following theorem is \cite[Th\'eor\`eme 3.13]{Clo90}:
\begin{theorem}\label{thm_field_of_definition}
\begin{enumerate}
\item $K_\pi$ is a number field, and $\pi^\infty$ may be defined over $K_\pi$. 
\item If $\sigma \in \Aut(\bC)$ then there exists a cuspidal, regular algebraic  automorphic representation ${}^\sigma \pi$ of $\GL_n(\bA_F)$ such that $({}^\sigma \pi)^\infty \cong {}^\sigma \pi^\infty$.
\end{enumerate} 
\end{theorem}
Let $\widetilde{K}_\pi \subset \bC$ be the normal closure of $K_\pi$, and let $\Gamma_\pi = \Gal(\widetilde{K}_\pi / \bQ)$. We see that the isomorphism class of ${}^\sigma \pi$ depends only on the image of $\sigma$ in $\Gamma_\pi$. 

We now relate this to Galois representations. Suppose that $F$ is a totally real (resp. CM) field and that $\pi$ is RAESDC (resp. RAECSDC). Then for any prime number $p$ and isomorphism $\iota : \overline{\bQ}_p \to \bC$, there is an associated Galois representation $r_{\pi, \iota} : G_F \to \GL_n(\overline{\bQ}_p)$. If $\sigma_p \in G_{\bQ_p}$, then we can act on the coefficients to get another continuous representation ${}^{\sigma_p} r_{\pi, \iota} = \overline{\bQ}_p \otimes_{\overline{\bQ}_p, \sigma_p} r_{\pi, \iota}$. If $\tau_p \in G_{\bQ_p}$, then ${}^{\sigma_p \tau_p} r_{\pi, \iota} \cong {}^{\sigma_p} ({}^{\tau_p} r_{\pi, \iota})$. The following lemma follows from the definitions.
\begin{lemma}\label{lem_galois_conjugation}
\begin{enumerate} \item Let $\sigma \in \Aut(\bC)$. Then there is an isomorphism $r_{{}^\sigma \pi, \iota} \cong r_{\pi, \sigma^{-1}\iota }$.
\item Let $\sigma_p \in G_{\bQ_p}$, and let $\sigma = \iota \sigma_p \iota^{-1} \in \Aut(\bC)$. Then there are isomorphisms 
\[ {}^{\sigma_p} r_{\pi, \iota} \cong  r_{\pi, \iota \sigma_p^{-1}}  \cong r_{\pi, \sigma^{-1} \iota} \cong r_{{}^\sigma \pi, \iota}. \]
\end{enumerate}
\end{lemma}
Next, we take $F$ be a totally real number field, and describe conditions on a cuspidal, regular algebraic automorphic representation $\pi$ of $\GL_2(\A_F)$ under which each member of its associated compatible system of Galois representations has large residual image. Let $H / F$ denote the ray class field of modulus $(8) \cdot \{ v | \infty \}$. If $w$ is a finite place of $F$ of odd residue characteristic which splits in $H$ then there exists a totally positive element $\lambda_w \in \cO_F$ such that $(\lambda_w) = \p_w$ (i.e.\ $\lambda_w$ generates the prime ideal $\p_w \subset \cO_F$ associated to $w$) and for each place $v | 2$ of $F$, $\lambda_w \in (F_v^\times)^2$. If $S$ is any set of finite places of $F$, then we write $F(S)$ for the maximal abelian extension of $F$ of exponent 2 which is unramified outside $S$. 
\begin{lemma}\label{lem_quadratic_reciprocity}
Let $S$ be a finite set of finite places of $F$, and let $v, w \nmid 2$ be finite places of $F$, not in $S$, which split in $F(S) \cdot H / F$. Then $v$ splits in $F(S \cup \{ w \})$ if and only if $w$ splits in $F(S \cup \{ v \})$.
\end{lemma}
\begin{proof}
Write $\p_v = (\lambda_v)$, $\p_w = (\lambda_w)$, as above. By Kummer theory, $F(S \cup \{ v \}) = F(S) \cdot F(\sqrt{\lambda_v})$ and $F(S \cup \{ w \}) = F(S) \cdot F(\sqrt{\lambda_w})$. We therefore just need to show that $w$ splits in $F(\sqrt{\lambda_v})$ if and only if $v$ splits in $F(\sqrt{\lambda_w})$, or in other words that $(\lambda_v, \lambda_w)_{F_v} = (\lambda_v, \lambda_w)_{F_w}$ (equality of Hilbert symbols). By reciprocity, we have
\[ \prod_u (\lambda_v, \lambda_w)_u = 1, \]
so it's enough to show that $(\lambda_v, \lambda_w)_u = 1$ if $u \neq v, w$ is any other place of $F$. If $u | 2 \infty$, this is true because $\lambda_v$, $\lambda_w$ are squares locally in $F_u^\times$. If $u \nmid 2 \infty$, this is true because $\lambda_v, \lambda_w$ are units. This completes the proof. 
\end{proof}

The following result is an analogue over $F$ of \cite[Proposition 6.1]{Die11}.
\begin{proposition}\label{prop_large_image_everywhere}
Fix $M > 5$. Let $\pi$ be a RAESDC automorphic representation of $\GL_2(\bA_F)$ of weight 0 satisfying the following conditions:
\begin{enumerate}
\item There exist prime numbers $t_1, t_2 > M$ and places $v_1, v_2$ of $F$ of residue characteristics $p_1, p_2 > 3$ unramified in $F$ such that $\pi_{v_i}$ is tamely dihedral of order $t_i$ ($i = 1, 2$) and $p_1, p_2, t_1, t_2$ are all distinct. 
\item Let $S$ denote the set of finite places $v$ of $F$ such that $v \neq v_1, v_2$ and either $\pi_v$ is ramified or $v$ is ramified over $\bQ$, together with the places of residue characteristic $2$ or $3$. Then $v_1$ splits in $F(S \cup \{ v_2 \})$ and $v_2$ splits in $F(S \cup \{ v_1 \})$.
\end{enumerate}
Then for any prime number $p$ and any isomorphism $\iota : \overline{\bQ}_p \to \bC$, $\overline{r}_{\pi, \iota}(G_F)$ contains a conjugate of $\SL_2(\bF_{p^a})$ for some $p^a > M$.
\end{proposition}
\begin{proof}
Take a prime number $p$ and isomorphism $\iota : \overline{\bQ}_p \to \bC$. Suppose first that $p \not\in \{ p_1, t_1 \}$. Then $\overline{r}_{\pi, \iota}|_{G_{F_{v_1}}}$ is irreducible, so $\overline{r}_{\pi, \iota}$ is irreducible and its projective image contains elements of order $t_1$. The classification of finite subgroups of $\PGL_2(\overline{\bF}_p)$ shows that the projective image of $\overline{r}_{\pi, \iota}$ is either conjugate to one of $\PSL_2(\bF_{p^a})$ or $\PGL_2(\bF_{p^a})$ for some $a \geq 1$, or that $\overline{r}_{\pi, \iota}$ is dihedral. In the first case, we see that $\PSL_2(\bF_{p^a})$ contains an element of order $t_1$ (and $t_1 \neq p$), and therefore that $t_1$ divides $p^a - 1$ or $p^a + 1$. In particular, we must have $p^a > M$.

We rule out the case that $\overline{r}_{\pi, \iota}$ is dihedral, i.e.\ that there is an isomorphism $\overline{r}_{\pi, \iota} \cong \Ind_{G_K}^{G_F} \overline{\chi}$ for some quadratic extension $K / F$ and character $\overline{\chi} : G_K \to \overline{\bF}_p^\times$. If $K/F$ is ramified at a place $v$ of $F$, then $\overline{r}_{\pi, \iota}$ is likewise ramified with inertial image of order twice that of $\overline{\chi}$; thus we must have $v \in S \cup \{ v_1, v_2 \}$ or $v | p$. In fact $\overline{r}_{\pi, \iota}(I_{F_{v_1}})$ has odd order, so $v \in S \cup \{ v_2 \}$ or $v | p$. If $v | p$ and $v \not\in S \cup \{ v_2 \}$, then $F_v / \bQ_p$ is unramified, $p \geq 5$ and $\pi_v$ is unramified. There is moreover an isomorphism $\overline{r}_{\pi, \iota}|_{I_{F_v}} \cong \chi_v \oplus \chi_v \delta_{K_v / F_v}$ for some character $\chi_v : I_{F_v} \to \overline{\bF}_p^\times$. This contradicts the fact that $\overline{r}_{\pi, \iota}|_{G_{F_v}}$ is Fontaine--Laffaille with weights $\{ 0, 1 \}$ with respect to any embedding $\tau : F_v \to \overline{\bQ}_p$ (which implies that $\overline{r}_{\pi, \iota}|_{I_{F_v}}$ is a sum of two products of distinct fundamental characters). Therefore $K / F$ is ramified only at places $v \in S \cup \{ v_2 \}$, hence $K \subset F(S \cup \{v_2\})$. This leads to a contradiction. Indeed, our hypothesis implies that $v_1$ splits in $K$, contradicting the fact that $\overline{r}_{\pi, \iota}|_{G_{F_{v_1}}}$ is irreducible.

In the remaining case $p \in \{ p_1, t_1 \}$, we see that $p \not\in \{ p_2, t_2 \}$, so we can apply the same argument after reversing the roles of $v_1$ and $v_2$.
\end{proof}

\subsection{Automorphy of a twisted tensor product -- first version}\label{subsect_aut_twisted_first_version}

We now state and prove the first theorem in this section concerning the automorphy of certain twisted tensor products. 
\begin{theorem}\label{thm_automorphic_tensor_product}
Let $F$ be a totally real field and let $p \geq 5$ be a prime. Let $0 < r < p$, and suppose that $\mathrm{SP}_{r-1}$, $\mathrm{SP}_r$, and $\mathrm{SP}_{r+1}$ hold. 
Let $\iota : \overline{\bQ}_p \to \bC$ be an isomorphism, and let $\pi$ be a RAESDC automorphic representation of $\GL_2(\bA_F)$ satisfying the following conditions:
\begin{enumerate}
\item $\det r_{\pi, \iota} = \epsilon^{-1}$.
\item $\pi$ is of weight $0$ and for each place $v | p$ of $F$, $\pi_v$ is an unramified twist of the Steinberg representation. In particular, $\pi$ is $\iota$-ordinary. 
\item There exists a place $v_0$ of $F$ such that $q_{v_0} \equiv -1 \text{ mod }p$ and $\pi_{v_0}$ is tamely dihedral of order $p$. 
\item For any prime number $l$ and any isomorphism $\iota_l : \overline{\bQ}_l \to \bC$, $\overline{r}_{\pi, \iota_l}(G_F)$ contains a conjugate of $\SL_2(\bF_{l^a})$ for some $l^a > \max(4r, 5)$. Moreover, $\overline{r}_{\pi, \iota}(G_F)$ contains a conjugate of $\SL_2(\bF_{p^a})$ for some $a > a_0(p)$ ($a_0(p)$ as defined in the statement of Lemma \ref{lem_tensor_adequacy}). \item For each finite place $v \nmid p$ of $F$, $\pi_v$ is potentially unramified.
\end{enumerate}
Let $\varphi_p \in G_{\bQ_p}$ be an arithmetic Frobenius lift. Then there exists an $\iota$-ordinary RAESDC automorphic representation $\Pi$ of $\GL_{2r}(\bA_F)$ such that $\overline{r}_{\Pi, \iota} \cong {}^{\varphi_p} \overline{r}_{\pi, \iota} \otimes \Sym^{r-1} \overline{r}_{\pi, \iota}$.
\end{theorem}
We assume that the hypotheses of Theorem \ref{thm_automorphic_tensor_product} are in effect for the remainder of \S \ref{subsect_aut_twisted_first_version}. A first consequence, which uses the assumption of $\mathrm{SP}_{r-1}$, $\mathrm{SP}_r$, and $\mathrm{SP}_{r+1}$, together with the results of \S \ref{sec:aut untwisted}, is as follows. 
\begin{lemma}\label{lem_sigma_with_tensor_product}
There exists a RAESDC automorphic representation $\sigma$ of $\GL_2(\bA_F)$ satisfying the following conditions:
\begin{enumerate}
\item $\sigma$ is $\iota$-ordinary and for each embedding $\tau : F \to \overline{\bQ}_p$, $\mathrm{HT}_\tau(r_{\sigma, \iota}) = \{ 0, 2 \}$. 
\item We have $\det r_{\sigma, \iota} = \epsilon^{-2} \omega$, where $\omega$ denotes the Teichm\"uller lift of $\epsilon \text{ mod }p$.
\item There is an isomorphism $\overline{r}_{\sigma, \iota} \cong \overline{r}_{\pi, \iota}$.
\item $\sigma_{v_0}$ is an unramified twist of the Steinberg representation.
\item For each place $v \nmid p$ of $F$ such that $v \neq v_0$, $r_{\sigma, \iota}|_{G_{F_v}} \sim r_{\pi, \iota}|_{G_{F_v}}$. 
\item $r_{\pi, \iota} \otimes \Sym^{r-1} r_{\sigma, \iota}$ is automorphic. 
\end{enumerate}
\end{lemma}
\begin{proof}
This combines Lemma \ref{lem_hida_theory} and Theorem \ref{thm_untwisted_tensor_product}.
\end{proof}
 Let $K_\pi \subset \bC$ denote the field of definition of $\pi$, and let $B \geq 1$ be an integer such that $K_\pi$ is ramified only at primes $l < B$. Let $\widetilde{K}_\pi$ denote the Galois closure of $K_\pi$ in $\bC$, and let $\Gamma_\pi = \Gal(\widetilde{K}_\pi / \bQ)$. Let $\varphi = \iota \varphi_p \iota^{-1} \in \Aut(\bC)$. 
\begin{lemma}\label{lem_amenable_sigma}
We can find a prime number $t$, an isomorphism $\iota_t : \overline{\bQ}_t \to \bC$, and a RAESDC automorphic representation $\sigma'$ of $\GL_2(\bA_F)$ with the following properties:
\begin{enumerate}
\item $t > \max(5,4r+2, B)$, $t$ is unramified in $F$, and $\pi, \sigma, \sigma'$ are unramified at the $t$-adic places of $F$.
\item The representations $\overline{r}_{\pi, \iota_t} \otimes \Sym^{r-1} \overline{r}_{\sigma, \iota_t}$ and $\overline{r}_{{}^\varphi \pi, \iota_t} \otimes \Sym^{r-1} \overline{r}_{\sigma, \iota_t}$ are absolutely irreducible.  
\item There is an isomorphism $\overline{r}_{\sigma, \iota_t} \cong \overline{r}_{\sigma', \iota_t}$. Moreover, $\sigma'_{v_0}$ is an unramified twist of the Steinberg representation. 
\item The representation $r_{\pi, \iota_t} \otimes \Sym^{r-1} r_{\sigma', \iota_t}$ is automorphic. Moreover, $r_{{}^\varphi \pi, \iota_t} \otimes \Sym^{r-1} r_{\sigma, \iota_t}$ is automorphic if and only if $r_{{}^\varphi \pi, \iota_t} \otimes \Sym^{r-1} r_{\sigma', \iota_t}$ is automorphic. 
\item For each prime $l < B$ and for each isomorphism $\iota_l : \overline{\bQ}_l \to \bC$, $\overline{r}_{\sigma', \iota_l}(G_F)$ contains a conjugate of $\SL_2(\bF_{l^{a}})$ for some $l^a > \max(4r,5)$. Moreover, for any $\gamma \in \Gamma_\pi$ the intersection of the extensions of $F$ cut out by the representations $\operatorname{Proj} \overline{r}_{\sigma', \iota_l}$ and $\operatorname{Proj} \overline{r}_{{}^\gamma \pi, \iota_l}$ has degree at most 2 over $F$.
\end{enumerate} 
\end{lemma}
\begin{proof}
Let $\lambda$ be a RAESDC automorphic representation of $\GL_2(\bA_F)$ without CM. According to \cite[Proposition 3.8]{MR2172950}, the set of prime numbers $l$ for which there exists an isomorphism $\iota_l : \overline{\bQ}_l \to \bC$ such that $\overline{r}_{\lambda, \iota_l}(G_F)$ does not contain a conjugate of $\SL_2(\bF_l)$ is finite. Call such primes $l$ exceptional for $\lambda$. We claim that if $t$ is prime satisfying the first point of the lemma and that is not exceptional for $\pi$, ${}^\varphi \pi$, or $\sigma$, and $\iota_t : \overline{\bQ}_t \to \bC$ is any isomorphism, then the second point of the lemma is satisfied. We just give the argument for $\overline{r}_{\pi, \iota_t} \otimes \Sym^{r-1} \overline{r}_{\sigma, \iota_t}$. 

Indeed, let $M_\pi$, $M_{{}^\varphi \pi}$, $M_\sigma$ be the extensions of $F$ cut out by the associated projective representations $\operatorname{Proj} \overline{r}_{\pi, \iota_t}$, $\operatorname{Proj} \overline{r}_{{}^{\varphi} \pi, \iota_t}$, and $\operatorname{Proj} \overline{r}_{\sigma, \iota_t}$. If $M_\pi \cap M_\sigma \neq F$ then either $M_\pi \cap M_\sigma / F$ is quadratic or $M_\pi = M_\sigma$ (as $\Gal(M_\pi / F) \cong \PSL_2(\bF_{t^{a}})$ or $\PGL_2(\bF_{t^{a}})$ for some $a \geq 1$ and $\PSL_2(\bF_{t^{a}})$ is simple and is the unique normal subgroup of $\PGL_2(\bF_{t^a})$, and likewise for $M_\sigma$).

In the quadratic case the restriction of each of $\overline{r}_{\pi, \iota_t}$ and $\Sym^{r-1} \overline{r}_{\sigma, \iota_t}$ to $G_{M_\pi \cap M_\sigma}$ is irreducible, so the tensor product is irreducible. In the case $M_\pi = M_\sigma$ we find that $\operatorname{Proj} \overline{r}_{\pi, \iota_t}(G_F)$ and $\operatorname{Proj} \overline{r}_{\sigma, \iota_t}(G_F)$ are actually conjugate (since there is a unique conjugacy class of subgroups of $\PGL_2(\overline{\bF}_p)$ isomorphic to $\Gal(M_\pi / F)$). The automorphisms of  $\operatorname{Proj} \overline{r}_{\pi, \iota_t}(G_F)$ are generated by $\PGL_2(\bF_{t^a})$ and the Frobenius automorphism (see \cite[Ch. IV, \S 6]{Die63}).  It follows that, writing $\Ad$ for the adjoint representation on the Lie algebra of $\mathfrak{sl}_2$, the representations $\Ad \overline{r}_{\pi, \iota_t}$, $\Ad \overline{r}_{\sigma, \iota_t}$ are isomorphic, up to a Frobenius twist in the coefficients. This is a contradiction. Indeed, $\Ad r_{\pi, \iota_t}$ is crystalline with Hodge--Tate weights $\{ -1,  0, 1 \}$ (with respect to any embedding $F \to \overline{\bQ}_t$) and $\Ad r_{\sigma, \iota_t}$ is crystalline with Hodge--Tate weights $\{-2, 0, 2 \}$. Since $t$ is unramified in $F$ and $t > 5$, Fontaine--Laffaille theory \cite{fl} implies  that the associated residual representations cannot be isomorphic. 

It follows that all isomorphisms $\iota_t : \overline{\bQ}_t \to \bC$, where $t$ is a sufficiently large prime satisfying the first point of the lemma, also satisfy the second point. We fix a choice of $t$ and $\iota_t$ with this property such that moreover $t \equiv 1 \text{ mod }4$ and $t$ splits in $K_\sigma$, so $\overline{r}_{\sigma, \iota_t}(G_F)$ can be assumed (after replacing by a conjugate) to contain $\SL_2(\bF_t)$ and be contained in $\GL_2(\bF_t)$. We can then find (by arguing e.g.~as in Lemma \ref{lem_hida_theory}) a RAESDC automorphic representation $\sigma'$ of $\GL_2(\bA_F)$ satisfying the following conditions:
\begin{itemize}
\item $\det r_{\sigma, \iota_t} = \det r_{\sigma', \iota_t}$ and $\overline{r}_{\sigma, \iota_t} \cong \overline{r}_{\sigma', \iota_t}$.
\item Let $S$ denote the set of places of $F$ of residue characteristic $l < B$ or at which $\sigma$ is ramified. Then there is a place $v_1 \nmid t$ of $F$, split in $F(S)$ ($F(S)$ as defined in \S \ref{subsec_preparations} above), of residue characteristic $p_1 > B$, at which $\sigma_{v_1}$ is unramified, but such that $\sigma'_{v_1}$ is tamely dihedral of order $t$. \item For each place $v \neq v_1$ of $F$, $r_{\sigma, \iota_t}|_{G_{F_v}} \sim r_{\sigma', \iota_t}|_{G_{F_v}}$.
\end{itemize}
We explain why there exists a place $v_1$ of $F$, split in $F(S)$, such that $q_{v_1} \equiv -1 \text{ mod }t$ and $\overline{r}_{\sigma, \iota_t}(\Frob_{v_1})$ has eigenvalues with ratio $-1$, which is a necessary condition for $\sigma'$ to exist.
We consider the extension of $F$ cut out by $\operatorname{Proj} \overline{r}_{\sigma, \iota_t}$. It has Galois group isomorphic either to $\PSL_2(\bF_t)$ or $\PGL_2(\bF_t)$. In either case the image of complex conjugation is in $\PSL_2(\bF_t)$ (because $t \equiv 1 \text{ mod 4}$) so we can find $g \in G_{F(S)}$ whose image under $\operatorname{Proj} \overline{r}_{\sigma, \iota_t}$ is conjugate to the image of complex conjugation. Denote by $L$ the extension of $F(S)$ cut out by $\operatorname{Proj} \overline{r}_{\sigma, \iota_t}$. Since $[\PGL_2(\bF_t) : \PSL_2(\bF_t)] = 2$, the cyclic quotient $\Gal(L\cap F(S)(\zeta_t)/F(S))$ of $\Gal(L/F(S))$ has order $\le 2$. Since $t$ is unramified in $F(S)$ and $t \equiv 1 \text{ mod }4$, there is an element $h \in \Gal(F(S)(\zeta_t)/L\cap F(S)(\zeta_t))$ with $\epsilon(h) \equiv -1 \text{ mod }t$. We can therefore additionally impose the condition that $\epsilon(g) \equiv -1 \text{ mod }t$. The existence of $v_1$ then follows by the Chebotarev density theorem.

We claim that this $\sigma'$ satisfies the remaining requirements of the lemma. The third point holds by construction. The fourth point holds because the representations $r_{\pi, \iota_t} \otimes \Sym^{r-1} r_{\sigma', \iota_t}$ and $r_{{}^\varphi \pi, \iota_t} \otimes \Sym^{r-1} r_{\sigma', \iota_t}$ are residually irreducible (hence adequate, as $t > 4r+2$) and potentially diagonalisable, by construction, so we can apply e.g.\ \cite[Theorem 4.2.1]{BLGGT} (using Lemma \ref{lem_sigma_with_tensor_product} to verify the residual automorphy of the first of these representations). To check the fifth point, take a prime $l < B$ and an isomorphism $\iota_l : \overline{\bQ}_l \to \bC$. The argument to check the size of the image of $\overline{r}_{\sigma', \iota_l}$ uses that $\sigma'_{v_1}$ is tamely dihedral of order $t$, as in the proof of Proposition \ref{prop_large_image_everywhere}: this implies that $\overline{r}_{\sigma', \iota_l}$ is irreducible, and that $\operatorname{Proj} \overline{r}_{\sigma', \iota_l}(G_F)$ contains an element of order $t > 5$, showing that this projective image either contains $\SL_2(\bF_{l^a})$ for some $l^a > t$ or is dihedral. In the dihedral case the inducing field is ramified only at places at which $\sigma$ is ramified or at places of residue characteristic $< B$ -- this implies that $v_1$ splits in the inducing field (since it splits in $F(S)$), a contradiction to the irreducibility of $\overline{r}_{\sigma', \iota_l}|_{G_{F_{v_1}}}$.

Finally we see that $\operatorname{Proj} \overline{r}_{\sigma', \iota_l}$ and $\operatorname{Proj} \overline{r}_{{}^\gamma \pi, \iota_l}$ cut out extensions which have at most quadratic intersection because the first extension is ramified at $v_1$ (and inertia has image of odd order $t$) while the second is unramified at $v_1$.
\end{proof}
We see that in order to prove the automorphy of $r_{{}^\varphi \pi, \iota} \otimes \Sym^{r-1} r_{\sigma, \iota}$, it is enough to show that $r_{{}^{\varphi} \pi, \iota} \otimes \Sym^{r-1} r_{\sigma', \iota}$ is automorphic. We will deduce this from the automorphy of $r_{\pi, \iota} \otimes \Sym^{r-1} r_{\sigma', \iota}$ using the following lemma.
\begin{lemma}\label{lem_FLT_along_an_inertia_group}
Let $\gamma \in \Gamma_\pi$. Let $l$ be a prime, let $w | l$ be a place of $\widetilde{K}_\pi$, and let $\delta \in I_{w / l} \subset \Gamma_\pi$. Suppose that $r_{{}^\gamma \pi, \iota} \otimes \Sym^{r-1} r_{\sigma', \iota}$ is automorphic. Then $r_{{}^{\delta \gamma} \pi, \iota} \otimes \Sym^{r-1} r_{\sigma' ,\iota}$ is automorphic.
\end{lemma}
Before giving the proof of Lemma \ref{lem_FLT_along_an_inertia_group}, we show how to conclude the proof of Theorem \ref{thm_automorphic_tensor_product}. We first show that for any $\gamma \in \Gamma_\pi$, $r_{{}^\gamma \pi, \iota} \otimes \Sym^{r-1} r_{\sigma', \iota}$ is automorphic. Since $\bQ$ admits no everywhere unramified extensions, $\Gamma_\pi$ is generated by its inertia groups. More precisely, we can find prime numbers $l_1, \dots, l_s$, places $w_1, \dots, w_s$ of $\widetilde{K}_\pi$ lying above these primes, and elements $\delta_i \in I_{w_i / l_i}$ ($i = 1, \dots, s$) such that $\gamma = \delta_s \dots \delta_1$. We know by Lemma \ref{lem_amenable_sigma} that $r_{\pi, \iota} \otimes \Sym^{r-1} r_{\sigma', \iota}$ is automorphic. We now apply Lemma \ref{lem_FLT_along_an_inertia_group} $s$ times (at the $i$th step, taking $\delta = \delta_{i}$ and $\gamma =\delta_{i-1} \dots \delta_1$) to find that $r_{{}^\gamma \pi, \iota} \otimes \Sym^{r-1} r_{\sigma', \iota}$ is automorphic.

Next, we take $\gamma = \varphi$ and find, using Lemma \ref{lem_amenable_sigma}, that $r_{{}^\varphi \pi, \iota} \otimes \Sym^{r-1} r_{\sigma, \iota}$ is automorphic. Let $\Pi$ be the associated RAESDC automorphic representation of $\GL_{2r}(\A_F)$. To finish the proof, we need to check that $\Pi$ is $\iota$-ordinary. However, this is a simple check from Definition \ref{def_ordinary} using that both ${}^\varphi \pi$ and $\sigma$ are $\iota$-ordinary and of parallel weight (that is, with weight vector $\lambda = (\lambda_\tau)$ such that $\lambda_\tau$ is independent of $\tau$).  \qed
\vspace{\baselineskip}
\begin{proof}[Proof of Lemma \ref{lem_FLT_along_an_inertia_group}]
We
can assume that $\delta \neq 1$, in which case $l < B$. It suffices to verify that  $r_{{}^{\delta \gamma} \pi, \jmath} \otimes \Sym^{r-1} r_{\sigma' , \jmath}$ is automorphic for a single prime $q$ and isomorphism $\jmath : \overline{\bQ}_q \to \bC$ (cf.~Lemma \ref{lem_the_meaning_of_existence}). We choose $q = l$, $\jmath = \iota_l : \overline{\bQ}_l \to \bC$ to be an isomorphism such that $\iota_l^{-1}$ induces the place $w$ of $\widetilde{K}_\pi$. Let $\widetilde{\delta} \in I_{\bQ_l}$ be any element such that $(\iota_l \widetilde{\delta} \iota_l^{-1})|_{\widetilde{K}_\pi} = \delta$. By Lemma \ref{lem_galois_conjugation}, there is an isomorphism
\[ r_{{}^{\delta \gamma} \pi, \iota_l} \cong {}^{\widetilde{\delta}} r_{{}^\gamma \pi, \iota_l}. \]
Since $\widetilde{\delta} \in I_{\bQ_l}$, the two representations ${}^{\widetilde{\delta}} r_{{}^\gamma \pi, \iota_l}$, $r_{{}^\gamma \pi, \iota_l}$ have isomorphic residual representations. We are therefore in a position to apply Theorem \ref{thm_FLT_for_TP} with the aim of deducing the automorphy of ${}^{\widetilde{\delta}} r_{{}^\gamma \pi, \iota_l} \otimes \Sym^{r-1} r_{\sigma', \iota_l}$. We first specify the data to which we will apply the theorem. The automorphic representations of $\GL_2(\A_F)$ are taken to be ${}^\gamma \pi$ and ${}^{\delta \gamma} \pi$. The automorphic representation of $\GL_r(\A_F)$ is taken to be $\Sym^{r-1} \sigma'$ (which exists since we are assuming $\mathrm{SP}_r$ holds -- see the statement of Theorem \ref{thm_automorphic_tensor_product}). We work with respect to the isomorphism $\iota_l : \overline{\bQ}_l \to \bC$. We now check the hypotheses (1)--(9) of Theorem \ref{thm_FLT_for_TP} in turn:
\begin{enumerate}
\item ${}^\gamma \pi$ and ${}^{\delta \gamma} \pi$ are of weight 0 and are non-CM by construction.
\item There is an isomorphism $\overline{r}_{{}^{\delta \gamma} \pi, \iota_l} \cong \overline{r}_{{}^\gamma \pi, \iota_l}$ by construction.
\item $\overline{r}_{\pi, \iota_l}(G_F)$ contains a conjugate of $\SL_2(\bF_{l^a})$ for some $l^a > 5$ by hypothesis. 
\item We need to show that for each place $v | l$ of $F$, $r_{{}^\gamma \pi, \iota_l}|_{G_{F_v}}$ is ordinary (resp. potentially crystalline) if and only if $r_{{}^{\delta \gamma} \pi, \iota_l}|_{G_{F_v}} \cong {}^{\widetilde{\delta}} r_{{}^\gamma \pi, \iota_l}$ is ordinary (resp. potentially crystalline), and that if $l = 2$, then they are both potentially crystalline. If $l = p$ (in which case $l \geq 5)$, both $r_{{}^\gamma \pi, \iota_l}|_{G_{F_v}}$ and ${}^{\widetilde{\delta}} r_{{}^\gamma \pi, \iota_l}$ are ordinary and not potentially crystalline (cf. point (5) below), so we're done. If $l \neq p$ then they are both potentially crystalline (as $\pi_v$ is potentially unramified, a property which is invariant under Galois conjugation), so we just need to check that the property of a representation $\rho : G_{F_v} \to \GL_2(\overline{\bQ}_l)$ being ordinary of weight 0 is invariant under the action of $I_{\bQ_l}$ on the coefficients. This is an easy check.
\item For each place $v$ of $F$, the property of $\pi_v$ being a character twist of the Steinberg representation is invariant under Galois conjugation, so in particular ${}^\gamma \pi_v$ has this property if and only if ${}^{\delta \gamma} \pi_v$ does. 
\item We need to justify  the existence of a perfect subgroup $H  \leq (\overline{r}_{{}^\gamma \pi, \iota_l} \otimes \Sym^{r-1} \overline{r}_{\sigma', \iota_l})^{ss}(G_F)$ containing regular semisimple elements. By construction,  $\operatorname{Proj} \overline{r}_{{}^\gamma \pi, \iota_l}(G_F) = G_1$ is conjugate to $\PSL_2(\bF_{l^{a_1}})$ or $\PGL_2(\bF_{l^{a_1}})$ for some $l^{a_1} > \max(4r,5)$, and $\operatorname{Proj} \overline{r}_{ \sigma', \iota_l}(G_F) = G_2$ is conjugate to $\PSL_2(\bF_{l^{a_2}})$ or $\PGL_2(\bF_{l^{a_2}})$ for some $l^{a_2} > \max(4r,5)$, and the intersection of the extensions of $F$ cut out by these two projective representations is at most quadratic. In particular, the image of the product of these two projective representations contains (a conjugate of) $\PSL_2(\bF_{l^{a_1}}) \times \PSL_2(\bF_{l^{a_2}})$, and the image of $\overline{r}_{{}^\gamma \pi, \iota_l}\times \overline{r}_{ \sigma', \iota_l}$ contains a conjugate of $\SL_2(\bF_{l^{a_1}}) \times \SL_2(\bF_{l^{a_2}})$. We take $H$ to be the image of this conjugate under $(\overline{r}_{{}^\gamma \pi, \iota_l} \otimes \Sym^{r-1} \overline{r}_{\sigma', \iota_l})^{ss}$. It is a perfect group because $l^{a_1}, l^{a_2} > 5$. We need to show that we can find $(g_1, g_2) \in \SL_2(\bF_{l^{a_1}}) \times \SL_2(\bF_{l^{a_2}})$ such that $g_1 \otimes \Sym^{r-1} g_2$ is regular semisimple. We can choose $g_2 = \diag(x, x^{-1})$ for any $x \in \bF^\times_{l^{a_2}}$ of order greater than $2(r-1)$, and then $g_1 = \diag(y, y^{-1})$ for any $y \in \bF^\times_{l^{a_1}}$ avoiding the elements $\pm x^{i}$ ($i = r-1, r-2, \dots, 1-r$). This is possible since $l^{a_1} > 4r$.
\item The Hodge--Tate weights of $\Sym^{r-1} r_{\sigma' ,\iota}$ are $\{ 0, 2, \dots, 2r-2 \}$, with respect to any embedding $F \to \overline{\bQ}_l$, by construction.
\item The tensor product $r_{{}^\gamma \pi, \iota} \otimes \Sym^{r-1} r_{\sigma', \iota}$ is automorphic, by assumption, and strongly irreducible by Lemma \ref{lem_irreducibility_of_tensor_products} (note that $\sigma'$ is without CM, since $\sigma'_{v_0}$ is an unramified twist of the Steinberg representation). 
\item The tensor product $r_{{}^{\delta \gamma} \pi, \iota} \otimes \Sym^{r-1} r_{\sigma', \iota}$ is again irreducible, and even strongly irreducible, by Lemma \ref{lem_irreducibility_of_tensor_products}. 
\end{enumerate} 
Since all of the hypotheses are satisfied, we can therefore apply Theorem \ref{thm_FLT_for_TP} to deduce that $r_{{}^{\delta \gamma} \pi, \iota_l} \otimes \Sym^{r-1} r_{\sigma' ,\iota_l}$ is automorphic. This completes the proof of the lemma. 
\end{proof}

\subsection{Automorphy of a twisted tensor product -- second version}\label{subsect_aut_twisted_second_version}

We next state and prove the  variant of Theorem \ref{thm_automorphic_tensor_product} that will be useful for our main application, requiring as it does only local conditions on the automorphic representation $\pi$. This variant will also allow for a place $v\nmid p$ where $\pi_v$ is an unramified twist of Steinberg. A Steinberg place will be needed in the proof of Proposition \ref{prop_automorphy_in_favourable_case} when we apply the level raising and automorphy lifting theorems \cite[Proposition 7.4, Theorem 7.5]{Tho22a}.

\begin{theorem}\label{thm_big_image_implies_tensor_product}
 Let $F$ be a totally real field and let $p \geq 5$ be a prime. Let $M > p^{a_0(p)}$. Let $0 < r < p$. Let $\iota : \overline{\bQ}_p \to \bC$ be an isomorphism and let $\pi$ be a RAESDC automorphic representation of $\GL_2(\bA_F)$ satisfying the following conditions:
\begin{enumerate} 
\item $\det r_{\pi, \iota} = \epsilon^{-1}$.
\item $\pi$ is of weight 0 and for each place $v | p$ of $F$, $\pi_v$ is an unramified twist of the Steinberg representation. In particular, $\pi$ is $\iota$-ordinary.
\item There exists a place $v_0$ of $F$ such that $q_{v_0} \equiv -1 \text{ mod }p$ and $\pi_{v_0}$ is tamely dihedral of order $p$.
\item There exists a place $v_1$ of $F$ such that $q_{v_1} \equiv 1 \text{ mod }p$ and $\pi_{v_1}$ is an unramified twist of the Steinberg representation.
\item For each place $v \nmid p v_1$ of $F$, $\pi_v$ is potentially unramified.
\item There exist prime numbers $t_1, t_2 > M$ and places $w_1, w_2$ of $F$ of residue characteristics $p_1, p_2 > 3$ unramified in $F$ such that $\pi_{w_i}$ is tamely dihedral of order $t_i$ and $p_1, p_2, t_1, t_2$ are all distinct. 
\item Let $H / F$ denote the ray class field of modulus $(8) \cdot \{ v | \infty \}$, and let $S$ denote the set of places $v \neq w_1, w_2$ of $F$ such that $\pi_v$ is ramified, or such that $v$ is ramified over $\bQ$, or is of  residue characteristic $2$ or $3$. Then $w_1$ splits in $H \cdot F(S)$ and $w_2$ splits in $H \cdot F(S \cup \{ w_1 \})$.
\end{enumerate} 
Let $\varphi_p \in G_{\bQ_p}$ be an arithmetic Frobenius lift. Then $\overline{r}_{\pi, \iota}(G_F)$ contains a conjugate of $\SL_2(\bF_{p^a})$ for some $p^a > M$ and if $\mathrm{SP}_{r-1}$, $\mathrm{SP}_r$ and $\mathrm{SP}_{r+1}$ hold then there exists an $\iota$-ordinary RAESDC automorphic representation $\Pi$ of $\GL_{2r}(\bA_F)$ such that $\overline{r}_{\Pi, \iota} \cong {}^{\varphi_p} \overline{r}_{\pi, \iota} \otimes \Sym^{r-1} \overline{r}_{\pi, \iota}$.
\end{theorem}
\begin{proof}
We first remark that Lemma \ref{lem_quadratic_reciprocity} shows that $w_1$ splits in $H \cdot F(S \cup \{ w_2 \})$. The hypotheses of Proposition \ref{prop_large_image_everywhere} therefore apply to $\pi$ and we conclude that for any prime number $l$ and isomorphism $\iota_l : \overline{\bQ}_l \to \bC$, $\overline{r}_{\pi, \iota_l}(G_F)$ contains a conjugate of $\SL_2(\bF_{l^b})$ for some $l^b > M$. 
Employing the Khare--Wintenberger method as in Lemma \ref{lem_hida_theory}, we can find another RAESDC automorphic representation $\pi'$ of $\GL_2(\bA_F)$ satisfying the following conditions:
\begin{itemize}
\item $\det r_{\pi', \iota} = \epsilon^{-1}$ and $\pi'$ is of weight 0.
\item There is an isomorphism $\overline{r}_{\pi', \iota} \cong \overline{r}_{\pi, \iota}$.
\item For each place $v \nmid p v_1$ of $F$, $r_{\pi, \iota}|_{G_{F_v}} \sim r_{\pi', \iota}|_{G_{F_v}}$. For each place $v | p$ of $F$, $\pi'_v$ is an unramified twist of the Steinberg representation. 
\item There is an isomorphism $\rec_{F_{v_1}}(\pi'_{v_1}) \cong \chi_1 \oplus \chi_2$ for characters $\chi_1, \chi_2 : F_{v_1}^\times \to \bC^\times$ such that $(\chi_1 / \chi_2)|_{\cO_{F_{v_1}}^\times}$ has order $p$. In particular, $\pi'_{v_1}$ is potentially unramified. 
\end{itemize}
(As in the proof of Lemma \ref{lem_hida_theory}, we need to verify that the corresponding local lifting ring at the place $v_1$ is non-zero; the situation considered here is  simpler since we do not change the local behaviour at the $p$-adic places of $F$. This non-vanishing follows from the congruence modulo $\varpi$ between the unipotent and tamely-ramified-type lifting rings, cf. \cite[\S 3]{tay} or indeed \cite{shottonGLn}. We note that the same congruence is exploited in \cite{Ski01}.)

The hypotheses of Proposition \ref{prop_large_image_everywhere} still apply to $\pi'$ (because $\pi$ and $\pi'$ are ramified at the same sets of places of $F$). Therefore the representations $\overline{r}_{\pi', \iota_l}$ also have large image. Theorem \ref{thm_automorphic_tensor_product} applies to $\pi'$ (because it is potentially unramified at every place $v \nmid p$ of $F$; note too that we have $p^{a_0(p)} > \max(4r, 5)$, as $p \geq 5$ and $a_0(p) \geq 3$, so the condition $l^b > M$ implies that the condition on residual images is satisfied). We conclude the existence of an $\iota$-ordinary RAESDC automorphic representation $\Pi'$ of $\GL_{2r}(\bA_F)$ such that $\overline{r}_{\Pi', \iota} \cong {}^{\varphi_p} \overline{r}_{\pi, \iota} \otimes \Sym^{r-1} \overline{r}_{\pi, \iota}$. This completes the proof.
\end{proof}

\section{Getting to the end}\label{sec: the end}
In this section, we assemble the ingredients laid out so far and deduce Theorem \ref{introthm_symmetric_power_functoriality} from the introduction (which is stated again as Theorem \ref{thm_sym_power_functoriality_general_case} below). We explain the argument. By a standard reduction, it is enough to prove Conjecture \ref{conj_SP_n} (i.e.\ symmetric power power functoriality in degree $n$ for those cuspidal, automorphic representations $\pi$ of $\GL_2(\A_F)$, without CM, that are furthermore regular algebraic, and so have associated Galois representations). We argue by induction on $n$, and (assuming $n \geq 6$) write $n = p + r$ for a prime $p \geq 5$ and some integer $0 < r < p$. We then have, for any isomorphism $\iota : \overline{\bQ}_p \to \bC$, the congruence
\[ \Sym^{n-1} \overline{r}_{\pi, \iota} \overset{ss}{\cong} ((\det \overline{r}_{\pi, \iota})^r \otimes \Sym^{p-r-1} \overline{r}_{\pi, \iota}) \oplus ({}^{\varphi_p} \overline{r}_{\pi, \iota} \otimes \Sym^{r-1} \overline{r}_{\pi, \iota}). \]
By induction, and by Theorem \ref{thm_big_image_implies_tensor_product}, we have now established the residual automorphy of the summands on the right-hand side for representations $\pi$ satisfying suitable local conditions. Our tasks are therefore (i) to complete the proof of the automorphy of $\Sym^{n-1} r_{\pi, \iota}$ for such representations, using the automorphy lifting theorems established in \cite{jackreducible, All20} (in fact, we find it convenient to use a slightly more general statement, given in \cite{Tho22a}), and (ii) to show that the case of a general $\pi$ can be reduced to this one.  Both steps are a more-or-less routine variation on similar arguments given in \cite{Clo14}. 

We begin with the reduction to the case of $\pi$ satisfying favourable local conditions:
\begin{proposition}\label{prop_reduction_to_favourable_case}
Let $F$ be a totally real field and let $\pi$ be a non-CM RAESDC automorphic representation of $\GL_2(\bA_F)$. Let $p \geq 5$ be a prime, and let $0 < r < p$ be an integer. Then we can find a soluble totally real extension $E / F$, an isomorphism $\iota : \overline{\bQ}_p \to \bC$, and a RAESDC automorphic representation $\pi'$ of $\GL_2(\bA_E)$ satisfying the following conditions:
\begin{enumerate}  
\item $\Sym^{p+r-1} r_{\pi, \iota}$ is automorphic if and only if $\Sym^{p+r-1} r_{\pi', \iota}$ is. 
\item $\pi'$ is of weight 0 and for each place $v | p$ of $E$, $\pi'_v$ is an unramified twist of the Steinberg representation. In particular, $\pi'$ is $\iota$-ordinary.
\item $\det r_{\pi', \iota} = \epsilon^{-1}$.
\item There exists a place $v_0$ of $E$ such that $q_{v_0} \equiv -1 \text{ mod }p$ and $\pi'_{v_0}$ is tamely dihedral of order $p$.
\item There exists a place $v_1$ of $E$ such that $q_{v_1} \equiv 1 \text{ mod }p$ and $\pi'_{v_1}$ is an unramified twist of the Steinberg representation. 
\item For each place $v \nmid p v_1$ of $E$, $\pi'_v$ is potentially unramified.
\item There exist prime numbers $t_1, t_2 > p^{a_0(p)}$ and places $w_1, w_2$ of $E$ of residue characteristics $p_1, p_2 > 3$ unramified in $E$ such that $\pi'_{w_i}$ is tamely dihedral of order $t_i$ and $p_1, p_2, t_1, t_2$ are all distinct. 
\item Let $H / E$ denote the ray class field of modulus $(8) \cdot \{ v | \infty \}$, and let $S$ denote the set of places $v \neq w_1, w_2$ of $E$ such that $\pi'_v$ is ramified, such that $v$ is ramified over $\bQ$, or is of characteristic $2$ or $3$. Then $w_1$ splits in $H \cdot E(S)$ and $w_2$ splits in $H \cdot E(S \cup \{ w_1 \})$.
\end{enumerate} 
\end{proposition}
\begin{proof}
Let $n = p+r$. We
 first claim that we can find an isomorphism $\iota : \overline{\bQ}_p \to \bC$, a soluble totally real extension $E / F$, and a RAESDC automorphic representation $\pi_0$ of $\GL_2(\bA_E)$, all satisfying the following conditions:
\begin{itemize}
\item[(a)] $\pi_0$ is of weight 0 and for each place $v | p$ of $E$, $\pi_{0, v}$ is an unramified twist of the Steinberg representation. In particular, $\pi_0$ is $\iota$-ordinary.
\item[(b)] $\det r_{\pi_0, \iota} = \epsilon^{-1}$.
\item[(c)] $\overline{r}_{\pi_0, \iota}(G_E)$ contains a conjugate of $\SL_2(\bF_{p^a})$ for some $a > a_0(p)$.
\item[(d)] There exists a place $v_1 \nmid p$ of $E$ such that $q_{v_1} \equiv 1 \text{ mod }p$ and $\pi_{0, v_1}$ is an unramified twist of the Steinberg representation.
\item[(e)] $\Sym^{p+r-1} r_{\pi, \iota}$ is automorphic if and only if $\Sym^{p + r - 1} r_{\pi_0, \iota}$ is automorphic.
\item[(f)] For each place $v \nmid p v_1$ of $E$, $\pi_{0, v}$ is potentially unramified.
\end{itemize} 
(In other words, $\pi_0$ satisfies conditions (1)-(3) and (5)-(6) from the statement of the Proposition.) To prove the claim, we follow closely the proofs of  \cite[Propositions 5.2, 5.3]{Clo14}. By \cite[Proposition 3.8]{MR2172950}, we can find a prime $q$ and an isomorphism $\iota_q : \overline{\bQ}_q \to \bC$ with the following properties:
\begin{itemize}
\item $q$ is unramified in $F$ and prime to the conductor of $\pi$.
\item For each embedding $\tau : F \to \overline{\bQ}_q$, the Hodge--Tate weights $\mathrm{HT}_\tau(r_{\pi, \iota_q})$ differ by at most $q-2$.
\item $q > 2n+2$. In particular, $q > p$.
\item $\overline{r}_{\pi, \iota_q}(G_F)$ contains a conjugate of $\SL_2(\bF_q)$. 
\end{itemize} 
In particular, $r_{\pi, \iota_q}$ is potentially diagonalisable, by \cite[Lemma 1.4.1]{BLGGT}. Let $E / F$ be a soluble totally real extension with the following properties: 
\begin{itemize}
\item $E$ is linearly disjoint over $F$ from the composite of $F(\zeta_q)$ and the fixed field $\overline{F}^{\ker \overline{r}_{\pi, \iota_q}}$. In particular, $\overline{r}_{\pi, \iota_q}(G_E)$ contains a conjugate of $\SL_2(\bF_q)$ and $[E(\zeta_q):E] = q-1$. 
\item For each place $v | p$ of $E$, $\overline{r}_{\pi, \iota_q}|_{G_{E_{ v}}}$ is trivial and $q_v \equiv 1 \text{ mod } q$.
\item For each place $v | q$ of $E$, $\overline{r}_{\pi, \iota_q}|_{G_{E_{v}}}$ is trivial and $\zeta_q \in E_{v}$. 
\item For each place $v \nmid pq$ of $E$, $\overline{r}_{\pi, \iota_q}|_{G_{E_{v}}}$ is unramified. 
\item There is an everywhere unramified character $\chi : G_{E} \to \overline{\bQ}_q^\times$ of finite order such that $\overline{\chi}^2 \det \overline{r}_{\pi, \iota_q} = \epsilon^{-1}$. (The obstruction to finding a square-root of the totally even Galois character $\epsilon\det\overline{r}_{\pi, \iota_q}$ is given by a two-torsion element of the Brauer group of $F$. The Brauer class becoming trivial on restriction to $E$ corresponds to purely local conditions on $E$ at the (finite) places where the class is non-zero. Extending $E$ further if necessary, the square-root can be made everywhere unramified.)
\end{itemize} 
By \cite[Theorem 6.1.9]{blggord}, we can find a cuspidal, regular algebraic automorphic representation $\pi_a$ of $\GL_2(\A_{E})$ and a place $v_1 \nmid pq$ of $E$ with the following properties: 
\begin{itemize}
\item $\pi_a$ has weight 0, and is unramified away from $p v_1$. 
\item $\det r_{\pi_a, \iota_q} = \epsilon^{-1}$. 
\item For each place $v | p$ of $E$, $\pi_{a, v}$ is an unramified twist of the Steinberg representation.
\item $q_{v_1} \equiv 1 \text{ mod }p$ and $\pi_{a, v_1}$ is an unramified twist of the Steinberg representation.
\item There is an isomorphism $\overline{r}_{\pi_a, \iota_q} \cong \overline{r}_{\pi, \iota_q}|_{G_{E}} \otimes \overline{\chi}$. 
\end{itemize}

We choose the place $v_1$ to split completely in the composite of $E(\zeta_p)$ and the extension of $E$ cut out by $(\overline{r}_{\pi, \iota_q}|_{G_{E}} \otimes \overline{\chi})$. This also implies that $q_{v_1} \equiv 1 \text{ mod }q$ (considering the determinant of $\overline{r}_{\pi, \iota_q}$), so the $q$-adic local lifting ring at $v_1$ has a Steinberg component. 

Now  choose a prime $t > \max(q,p^{a_0(p)})$ with the following properties: 
\begin{itemize}
\item $t$ is unramified in $E$ and prime to the conductor of $\pi$.
\item $t$ splits in $K_{\pi_a}(\sqrt{-1})$ (recall that $K_{\pi}$ denotes the field of definition of an automorphic representation $\pi$). 
\item There is an isomorphism $\iota_t : \overline{\bQ}_t \to \bC$ such that $\overline{r}_{\pi_a, \iota_t}(G_E)$ contains a conjugate of $\SL_2(\bF_t)$. 
\end{itemize} 
After conjugation, we can assume that $\overline{r}_{\pi_a, \iota_t}(G_E)$ is sandwiched between $\SL_2(\bF_t)$ and $\GL_2(\bF_t)$. Let $T$ denote the set of places of $E$ at which $\pi_a$ is ramified, and choose a place $v_a$ of $E$, prime to $p t v_1$ and split in $E(T)$ (the extension defined in \S\ref{subsec_preparations}), such that $q_{v_a} \equiv -1 \text{ mod }t$ and $\Proj \,\overline{r}_{\pi_a, \iota_t}(\Frob_{v_a})$ is conjugate to complex conjugation (this is possible for the same reasons as explained in the proof of Lemma \ref{lem_amenable_sigma} -- in particular, as $t \equiv 1 \text{ mod }4$, the image of complex conjugation under the projective representation lies in $\PSL_2(\bF_t)$). Raising the level at $v_a$ (e.g.~using the Khare--Wintenberger method as in the proof of Lemma \ref{lem_hida_theory}), we can then find another cuspidal, regular algebraic automorphic representation $\pi_0$ of $\GL_2(\A_{E})$ with the following properties:
\begin{itemize}
\item $\pi_0$ has weight 0.
\item $\det r_{\pi_0, \iota_t} = \epsilon^{-1}$. 
\item For each place $v | p$ of $E$, $\pi_{0, v}$ is an unramified twist of the Steinberg representation.
\item $\pi_{0, v_1}$ is an unramified twist of the Steinberg representation.
\item $\pi_{0, v_a}$ is tamely dihedral of order $t$.
\item For each place $v \nmid pv_1v_a$ of $E$, $\pi_{0, v}$ is unramified.
\item There is an isomorphism $\overline{r}_{\pi_0, \iota_t} \cong \overline{r}_{\pi_a, \iota_t}$. 
\end{itemize}
In particular, repeating the argument of Lemma \ref{lem_amenable_sigma} (cf. \cite[Proposition 5.3]{Clo14}) shows that $\overline{r}_{\pi_0, \iota}(G_E)$ contains a conjugate of $\SL_2(\bF_{p^a})$ for some $a > a_0(p)$. This pair $(E, \pi_0)$ has all the properties (a) -- (f) claimed above, by construction, except possibly (e). We justify this now. In fact, there is an isomorphism of residual representations $\overline{\chi}^{p+r-1} \otimes \Sym^{p+r-1} \overline{r}_{\pi, \iota_q}|_{G_E} \cong \Sym^{p+r-1} \overline{r}_{\pi_a, \iota_q}$, so we see that $\Sym^{p+r-1} r_{\pi, \iota}$ is automorphic if and only if $\Sym^{p+r-1} r_{\pi_a, \iota_q}$ is automorphic using soluble base change and \cite[Theorem 4.2.1]{BLGGT}, noting that $r_{\pi_a, \iota_q}$ is also potentially diagonalisable by \cite[Lemma 4.4.1]{geekisin}. There is also an isomorphism of residual representations $\Sym^{p+r-1} \overline{r}_{\pi_a, \iota_t} \cong \Sym^{p+r-1} \overline{r}_{\pi_0, \iota_t}$, so another application of \cite[Theorem 4.2.1]{BLGGT} shows that $\Sym^{p+r-1} r_{\pi_a, \iota_q}$ is automorphic if and only if $\Sym^{p+r-1} r_{\pi_0, \iota_t}$ is automorphic. We note also that $q, t > 2n+2$, $\rbar_{\pi,\iota_q}(G_{E(\zeta_q)})$ contains a conjugate of $\SL_2(\bF_q)$, and $\overline{r}_{\pi_a, \iota_t}(G_{E(\zeta_t)})$ contains a conjugate of $\SL_2(\bF_t)$. This implies the irreducibility hypothesis in \cite[Theorem 4.2.1]{BLGGT} holds in both cases where we apply this automorphy lifting theorem. Thus we have completed the proof of the claim. 

Our next step is to construct an automorphic representation $\pi_1$ satisfying conditions (1)-(6) in the statement of the Proposition. Raising the level modulo $p$ at a place $v_0$ of $E$ such that $\overline{r}_{\pi_0, \iota}(\Frob_{v_0}) = \overline{r}_{\pi_0, \iota}(c)$, we can find another RAESDC automorphic representation $\pi_1$ of $\GL_2(\bA_E)$ such that $\overline{r}_{\pi_0, \iota} \cong \overline{r}_{\pi_1, \iota}$, satisfying (a)--(d) and (f) above, and such that $\pi_1$ is tamely dihedral of order $p$ at $v_0$ (note that $\det \overline{r}_{\pi_0, \iota} = \epsilon^{-1}$, so our choice of $v_0$ ensures that $q_{v_0} \equiv -1 \text{ mod }p$). We can  apply \cite[Theorem 7.5]{Tho22a}, checking the hypotheses on the residual representation as in the proof of  \cite[Theorem 7.6]{Tho22a}, to conclude that (e) also holds with $\pi_0$ replaced by $\pi_1$. 

Finally, we need to pass through two further congruences to achieve the remaining conditions (7)-(8) in the Proposition. Let $S$ denote the union of the set of places of $E$ at which $\pi_1$ is ramified with the set of places of $E$ which are ramified over $\bQ$ or which have residue characteristic dividing $6$. Choose a prime $t_1 > \max(2n+2,p^{a_0(p)})$, prime to the conductor of $\pi_1$, split in $K_{\pi_1}(\sqrt{-1})$ and unramified in $E$, and an isomorphism $\iota_{t_1} : \overline{\bQ}_{t_1} \to \bC$ such that $\overline{r}_{\pi_1, \iota_{t_1}}(G_E)$ is conjugate to a subgroup sandwiched between $\SL_2(\bF_{t_1})$ and $\GL_2(\bF_{t_1})$.
As in the proof of Lemma \ref{lem_amenable_sigma}, choose a place $w_1$ of $E$, prime to $t_1$ and every element of $S$, which splits in $H \cdot E(S)$ ($H$ is the ray class field appearing in condition (8)), and such that $q_{w_1} \equiv -1 \text{ mod }t_1$ and $\operatorname{Proj} \overline{r}_{\pi_1, \iota_{t_1}}(\Frob_{w_1}) = \operatorname{Proj} \overline{r}_{\pi_1, \iota_{t_1}}(c)$. 
Raising the level modulo $t_1$ at $w_1$ and applying \cite[Theorem 4.2.1]{BLGGT}, we can find a RAESDC automorphic representation $\pi_2$ of $\GL_2(\bA_E)$ satisfying the following conditions:
\begin{itemize}
\item There is an isomorphism $\overline{r}_{\pi_2, \iota_{t_1}} \cong \overline{r}_{\pi_1, \iota_{t_1}}$.
\item $\pi_2$ is of weight 0 and for each place $v | p$ of $E$, $\pi_{2, v}$ is an unramified twist of the Steinberg representation. In particular, $\pi_2$ is $\iota$-ordinary.
\item $\det r_{\pi_2, \iota} = \epsilon^{-1}$.
\item  $\pi_{2, v_1}$ is an unramified twist of the Steinberg representation. $\pi_2$ is tamely dihedral of order $p$ at $v_0$ and of order $t_1$ at $w_1$.
\item $\Sym^{p+r-1} r_{\pi, \iota}$ is automorphic if and only if $\Sym^{p + r - 1} r_{\pi_2, \iota}$ is automorphic.
\item For each place $v \nmid p v_1$ of $E$, $\pi_{2, v}$ is potentially unramified; for each place $v \not\in S \cup \{ w_1 \}$, $\pi_{2, v}$ is unramified.
\end{itemize} 
Let $p_1$ denote the residue characteristic of $w_1$. Now choose a prime $t_2 > \max(p^{a_0(p)}, p_1, t_1)$, prime to the conductor of $\pi_2$, split in $K_{\pi_2}(\sqrt{-1})$ and unramified in $E$, and an isomorphism $\iota_{t_2} : \overline{\bQ}_{t_2} \to \bC$ such that $\overline{r}_{\pi_2, \iota_{t_2}}$ is conjugate to a subgroup sandwiched between $\SL_2(\bF_{t_2})$ and $\GL_2(\bF_{t_2})$. Choose a place $w_2$ of $E$, prime to $p_1 t_1 t_2$ and every element of $S$, which splits in $H \cdot E(S \cup \{ w_1 \})$, and such that $q_{w_2} \equiv -1 
\text{ mod }t_2$ and $\operatorname{Proj} \overline{r}_{\pi_2, \iota_{t_2}}(\Frob_{w_2}) = \operatorname{Proj} \overline{r}_{\pi_2, \iota_{t_2}}(c)$. Raising the level modulo $t_2$ at $w_2$, we can find a RAESDC automorphic representation $\pi_3$ of $\GL_2(\bA_E)$ satisfying the following conditions:
\begin{itemize}
\item There is an isomorphism $\overline{r}_{\pi_3, \iota_{t_2}} \cong \overline{r}_{\pi_2, \iota_{t_2}}$.
\item $\pi_3$ is of weight 0 and for each place $v | p$ of $E$, $\pi_{3, v}$ is an unramified twist of the Steinberg representation. In particular, $\pi_3$ is $\iota$-ordinary.
\item $\det r_{\pi_3, \iota} = \epsilon^{-1}$.
\item  $\pi_{3, v_1}$ is an unramified twist of the Steinberg representation. $\pi_3$ is tamely dihedral of order $p$ at $v_0$, of order $t_1$ at $w_1$, and of order $t_2$ at $w_2$.
\item $\Sym^{p+r-1} r_{\pi, \iota}$ is automorphic if and only if $\Sym^{p + r - 1} r_{\pi_3, \iota}$ is automorphic.
\item For each place $v \nmid p v_1$ of $E$, $\pi_{3, v}$ is potentially unramified; for each place $v \not\in S \cup \{ w_1, w_2 \}$, $\pi_{3, v}$ is unramified. 
\end{itemize} 
Let $p_2$ be the residue characteristic of $w_2$. Then the primes $p_1, p_2, t_1, t_2$ are distinct, by construction.  The proof is complete on taking $\pi' = \pi_3$.
\end{proof}
\begin{cor}\label{cor_reduction_to_favourable_case}
Let $F$ be a totally real field and let $\pi$ be a non-CM RAESDC automorphic representation of $\GL_2(\bA_F)$. Let $p \geq 5$ be a prime, and let $0 < r < p$ be an integer. Suppose that $\mathrm{SP}_{r-1}$, $\mathrm{SP}_r$, and $\mathrm{SP}_{r+1}$ hold. Then we can find a soluble totally real extension $E / F$, an isomorphism $\iota : \overline{\bQ}_p \to \bC$, and a RAESDC automorphic representation $\pi'$ of $\GL_2(\bA_E)$ satisfying the following conditions:
\begin{enumerate} 
\item $\Sym^{p+r-1} r_{\pi, \iota}$ is automorphic if and only if $\Sym^{p+r-1} r_{\pi', \iota}$ is. 
\item $\pi'$ is of weight 0 and for each place $v | p$ of $F$, $\pi'_v$ is an unramified twist of the Steinberg representation. In particular, $\pi'$ is $\iota$-ordinary.
\item $\overline{r}_{\pi', \iota}(G_E)$ contains a conjugate of $\SL_2(\bF_{p^a})$ for some $a > a_0(p)$.
\item There exists a place $v_1 \nmid p$ such that $\pi'_{v_1}$ is an unramified twist of the Steinberg representation.
\item There exists an $\iota$-ordinary RAESDC automorphic representation $\Pi$ of $\GL_{2r}(\bA_E)$ such that $\overline{r}_{\Pi, \iota} \cong {}^{\varphi_p} \overline{r}_{\pi', \iota} \otimes \Sym^{r-1} \overline{r}_{\pi', \iota}$. 
\end{enumerate}
\end{cor}
\begin{proof}
Let $\pi'$ be as in Proposition \ref{prop_reduction_to_favourable_case}. We see that $\pi'$ satisfies the hypotheses of Theorem \ref{thm_big_image_implies_tensor_product}, and the Corollary follows immediately from this.
\end{proof}
\begin{proposition}\label{prop_automorphy_in_favourable_case}
Let $p \geq 5$ be a prime, let $0 < r < p$ be an integer, and let $F$ be a totally real field. Let $\pi$ be a RAESDC automorphic representation of $\GL_2(\bA_F)$ satisfying the following conditions:
\begin{enumerate} 
\item There exists an isomorphism $\iota : \overline{\bQ}_p \to \bC$ such that $\pi$ is $\iota$-ordinary and $\overline{r}_{\pi, \iota}(G_F)$ contains a conjugate of $\SL_2(\bF_{p^a})$ for some $a > a_0(p)$.
\item $\pi$ has weight 0.
\item There exists a place $v_1 \nmid p$ such that $\pi_{v_1}$ is an unramified twist of the Steinberg representation.
\item There exists an $\iota$-ordinary RAESDC automorphic representation $\pi_1$ of $\GL_{p-r}(\bA_F)$ such that $\overline{r}_{\pi_1, \iota} \cong \Sym^{p-r-1} \overline{r}_{\pi, \iota}$.
\item There exists an $\iota$-ordinary RAESDC automorphic representation $\pi_2$ of $\GL_{2r}(\bA_F)$ such that $\overline{r}_{\pi_2, \iota} \cong {}^{\varphi_p} \overline{r}_{\pi, \iota} \otimes \Sym^{r-1} \overline{r}_{\pi, \iota}$.
\end{enumerate}
Then the $(p+r-1)^\text{th}$ symmetric power of $\pi$ exists: there is an $\iota$-ordinary RAESDC automorphic representation $\Pi$ of $\GL_{p+r}(\bA_F)$ such that $r_{\Pi, \iota} \cong \Sym^{p+r-1} r_{\pi, \iota}$. 
\end{proposition}
\begin{proof}
This will follow from the automorphy lifting theorem \cite[Theorem 7.5]{Tho22a}, applied to an RACSDC twist of $\Sym^{p+r-1} r_{\pi,\iota}$ over a CM extension of $F$ as in the proof of \cite[Theorem 7.6]{Tho22a}. To verify the residual automorphy, we apply \cite[Proposition 7.4]{Tho22a} with hypotheses (4) and (5) (which are all that are needed) replacing the assumptions in \emph{loc.~cit.} of the automorphy of $\Sym^{r-1}r_{\pi,\iota}$ and $\Sym^{p-r-1}r_{\pi,\iota}$, as well as $\mathrm{TP}_r$ ($\GL_2 \times \GL_r \to \GL_{2r}$ tensor product functoriality). Checking the remaining hypotheses on the residual representation is done exactly as in the proof of \cite[Theorem 7.6]{Tho22a}.
\end{proof}
\begin{theorem}\label{thm_sym_power_functoriality_alg_case}
Conjecture $\mathrm{SP}_n$ holds for all $n \geq 2$.
\end{theorem}
\begin{proof}
We prove it by induction on $n \geq 6$, the cases $n \le 5$ being known already. In general we follow \cite[Corollary 7.2]{Tho22a}: choose a prime $p \geq 5$ such that $p < n < 2p$ and write $n = p + r$. Choosing a non-CM RAESDC automorphic representation $\pi$ of $\GL_2(\bA_F)$, we see that the existence of $\Sym^{n-1} \pi$ follows from Corollary \ref{cor_reduction_to_favourable_case}, Proposition \ref{prop_automorphy_in_favourable_case}, and the induction hypothesis. \end{proof}
Finally, we deduce the following more general statement, which includes those automorphic representations associated to Hilbert modular forms of regular (but possibly non-paritious) weight. (We recall that these automorphic representations do not have associated Galois representations, although they do have algebraic Hecke eigenvalues; their description in terms of Langlands parameters is given in \cite[\S 8]{Tho22a}. They fit into the broader class of `$W$-algebraic' automorphic representations introduced in \cite{Pat19}.)

The natural class of automorphic representations for us to consider is those cuspidal automorphic representations $\pi$ of $\GL_2(\bA_F)$, without CM, such that $\pi_\infty$ is essentially square-integrable. In other words, for each $v|\infty$, the local factor $\pi_v$ is essentially discrete series (of some weight $k_v \ge 2$), an infinite dimensional subquotient of the normalized parabolic induction $\Ind_{B(\R)}^{\GL_2(\R)}\chi_1\otimes\chi_2$ with $\chi_i = \mathrm{sgn}^{\epsilon_i}|\cdot|^{s_i}: \R^\times \to \bC^\times$ for some $s_i \in \bC$ and $\epsilon_i = 0,1$ satisfying $s_1 - s_2 \in \Z_{>0}$ and $k_v = s_1 - s_2 + 1 \equiv \epsilon_1+\epsilon_2 \text{ mod }2$. 
\begin{theorem}\label{thm_sym_power_functoriality_general_case}
\begin{enumerate} \item Let $F$ be a totally real number field, and let $\pi$ be a cuspidal automorphic representation of $\GL_2(\bA_F)$, without CM, such that $\pi_\infty$ is essentially square-integrable. Let $n \geq 2$. Then $\Sym^{n-1} \pi$ exists, in the sense that there is a cuspidal automorphic representation $\Pi_n$ of $\GL_n(\bA_F)$ such that for every place $v$ of $F$, we have
\[ \rec_{F_v}(\Pi_{n, v}) \cong \Sym^{n-1} \circ \rec_{F_v}(\pi_v). \]
\item Let $E$ be a CM number field, and let $\pi$ be a RAECSDC automorphic representation of $\GL_2(\bA_E)$ which is not automorphically induced from a quadratic extension. Let $n \geq 2$. Then $\Sym^{n-1} \pi$ exists, in the sense that there is a cuspidal automorphic representation $\Pi_n$ of $\GL_n(\bA_E)$ such that for every place $w$ of $E$, we have
\[ \rec_{E_w}(\Pi_{n, w}) \cong \Sym^{n-1} \circ \rec_{E_w}(\pi_w). \]
\end{enumerate} 
\end{theorem}
Part (1) is exactly Theorem \ref{introthm_symmetric_power_functoriality} from the introduction.
\begin{proof}
The first part follows from Theorem \ref{thm_sym_power_functoriality_alg_case} and \cite[Theorem 8.1]{Tho22a}. The second part follows, in the case that $\pi$ is RACSDC, from Theorem \ref{thm_sym_power_functoriality_alg_case} and \cite[Proposition 7.6]{ctiii}. The general case can be reduced to this one by twisting (using \cite[Lemma 4.1.4]{cht}).
\end{proof}

\bibliographystyle{amsalpha}
\bibliography{CMpatching}

\end{document}